\documentclass[11pt,english]{amsart}
\usepackage[LGR,T1]{fontenc}
\usepackage[a4paper]{geometry}
\usepackage{textcomp}
\usepackage{mathrsfs}
\usepackage{amstext}
\usepackage{amsthm}
\usepackage{amssymb}
\usepackage{graphicx}
\usepackage{wasysym}
\usepackage{esint}

\makeatletter

\DeclareRobustCommand{\greektext}{%
  \fontencoding{LGR}\selectfont\def\encodingdefault{LGR}}
\DeclareRobustCommand{\textgreek}[1]{\leavevmode{\greektext #1}}
\ProvideTextCommand{\~}{LGR}[1]{\char126#1}

\numberwithin{equation}{section}
\numberwithin{figure}{section}
\theoremstyle{plain}
\newtheorem{thm}{\protect\theoremname}
  \theoremstyle{definition}
  \newtheorem{example}[thm]{\protect\examplename}
  \theoremstyle{definition}
  \newtheorem{defn}[thm]{\protect\definitionname}
  \theoremstyle{remark}
  \newtheorem{rem}[thm]{\protect\remarkname}
  \theoremstyle{remark}
  \newtheorem{claim}[thm]{\protect\claimname}
  \theoremstyle{plain}
  \newtheorem{lem}[thm]{\protect\lemmaname}
  \theoremstyle{plain}
  \newtheorem{prop}[thm]{\protect\propositionname}
  \theoremstyle{plain}
  \newtheorem{cor}[thm]{\protect\corollaryname}
  \theoremstyle{remark}
  \newtheorem*{rem*}{\protect\remarkname}
  \theoremstyle{plain}
  \newtheorem*{prop*}{\protect\propositionname}

\makeatother

\usepackage{babel}
  \providecommand{\claimname}{Claim}
  \providecommand{\corollaryname}{Corollary}
  \providecommand{\definitionname}{Definition}
  \providecommand{\examplename}{Example}
  \providecommand{\lemmaname}{Lemma}
  \providecommand{\propositionname}{Proposition}
  \providecommand{\remarkname}{Remark}
\providecommand{\theoremname}{Theorem}

\begin{document}
\global\long\def\bbZ{\mathbb{Z}}

\global\long\def\RR{\mathbb{R}}

\global\long\def\bbN{\mathbb{N}}

\global\long\def\bbX{\mathbb{X}}

\global\long\def\MM{\mathbb{M}}

\global\long\def\cF{\mathcal{F}}

\global\long\def\cB{\mathcal{B}}

\global\long\def\fT{\mathfrak{\tau}}

\global\long\def\cCT{\mathcal{C}_{\fT-1}}

\global\long\def\cC{\mathcal{C}}

\global\long\def\RR{\mathbb{R}}

\global\long\def\BB{\mathcal{B}}

\global\long\def\P{\mathrm{P}}

\global\long\def\C#1{\mathcal{#1}}

\global\long\def\ZZ{\mathbb{Z}}

\global\long\def\RN#1{\left(T^{#1}\right)'}

\global\long\def\F#1{\mathcal{F}_{#1}}

\global\long\def\lloc{\ll^{loc}}

\global\long\def\PP{\mathbb{P}}

\global\long\def\ep{\underline{\epsilon}}

\global\long\def\p{\varphi}

\global\long\def\l{\lambda}

\global\long\def\NN{\mathbb{N}}

\global\long\def\M#1{\text{M}\left\{  #1\right\}  }

\global\long\def\P{{\rm P}}

\global\long\def\TT{\mathbb{T}}

\global\long\def\g{\mathfrak{g}}

\global\long\def\h{\mathfrak{h}}

\global\long\def\r{\varrho}

\global\long\def\qq{\mathtt{q}}

\global\long\def\rr{\mathtt{r}}

\title{Conservative Anosov diffeomorphisms of $\TT^{2}$ without an absolutely
continuous invariant measure. }

\author{Zemer Kosloff}

\address{Mathematics institute, University of Warwick, Gibbet Hill Road, Coventry,
CV47AL.}

\email{zemer.kosloff@mail.huji.ac.il}

\curraddr{Einstein institute of Mathematics, Hebrew University of Jerusalem,
Kiryat Safra campus, Jerusalem. }
\begin{abstract}
We construct examples of $C^{1}$ Anosov diffeomorphisms on $\TT^{2}$
which are of Krieger type ${\rm III}_{1}$ with respect to Lebesgue
measure. This shows that the Gurevic Oseledec phenomena that conservative
$C^{1+\alpha}$ Anosov diffeomorphisms have a smooth invariant measure
does not hold true in the $C^{1}$ setting. 
\end{abstract}

\maketitle

\section{Introduction}

This paper provides the first examples of Anosov diffeomorphisms of
$\TT^{2}$ which are conservative and ergodic yet there is no Lebesgue
absolutely continuous invariant measure.

Let $M$ be a compact, boundaryless smooth manifold and $f:M\to M$
be a diffeomorphism. A natural question which arises is whether $f$
preserves a measure which is absolutely continuous with respect to
the volume measure on $M$. In order to avoid confusion in what follows,
we would like to stress out that in this paper, the term \textbf{conservative}
means the definition from ergodic theory which is non existence of
wandering sets of positive measure. That is $f$ is conservative if
and only for every $W\subset M$ so that $\left\{ f^{n}W\right\} _{n\in\ZZ}$
are disjoint (modulo the volume measure), $vol(W)=0$. 

It follows from \cite{Livsic Sinai} that for a generic $C^{2}$ Anosov
diffeomorphism there exists no absolutely continuous invariant measure
(a.c.i.m.), \cite[p. 72, Corollary 4.15.]{Bowen Rufus}. Following
this result Sinai asked whether a generic Anosov diffeomorphism will
satisfy Poincare recurrence. This question was answered by Gurevic
and Oseledec \cite{Gurevic Oseledec} who proved that the set of conservative
(Poincare recurrent) $C^{2}$ Anosov diffeomorphism is meagre in the
$C^{2}$ topology (restricted to the open set of Anosov diffeomorphisms).
Indeed, they have proved that if $f$ is a conservative $C^{2}$ Anosov
(hyperbolic) diffeomorphism, then $f$ preserves a probability measure
in the measure class of the volume measure which combined with the
result of Livsic and Sinai proves the non-genericity result. The proof
in \cite{Gurevic Oseledec} uses the absolute continuity of the foliations
and existence of SRB measures to show that if the $SRB$ measure for
$f$ is not equal to the SRB measure for $f^{-1}$ then there exists
a continuous function $g:M\to\RR$ and a set $A\subset M$ of positive
volume so that, 
\[
\lim_{n\to\infty}\frac{1}{n}\sum_{k=0}^{n-1}g\left(f^{k}(x)\right)\neq\lim_{n\to\infty}\frac{1}{n}\sum_{k=0}^{n-1}g\left(f^{-k}(x)\right),\ \forall x\in A.
\]
It is then a straightforward argument to construct a set $B\subset A$
of positive volume measure so that for almost every $x\in B$, the
set $\left\{ k\in\NN:\ f^{k}x\in B\right\} $ is finite, in contradiction
with Halmos Recurrence Theorem \cite{Aar}.

This result remains true for $C^{1+\alpha},\ \alpha>0$ Anosov diffeomorphisms.
However, since there exists $C^{1}$ Anosov diffeomorphisms whose
stable and unstable foliations are not absolutely continuous \cite{Robinson Young},
this proof can not be generalized for the $C^{1}$ setting. This paper
is concerned with the question whether every conservative $C^{1}$-
Anosov diffeomorphism has an absolutely continuous invariant measure?

An easier version of this question was studied before in the context
of smooth expanding maps. Every $C^{2}$ expanding map of a manifold
has an absolutely continuous invariant measure \cite{Krz}. In contrast
to the higher regularity case, Avila and Bochi \cite{Avilla Bochi generic C^1 ...,Avila Bochi generic expanding (interval maps)},
extending previous results of Campbell and Quas \cite{Campbell Quas},
have shown that a generic $C^{1}$ expanding map has no a.c.i.m and
a generic $C^{1}$ expanding map of the circle is not recurrent \cite{Campbell Quas}.
It seems natural to argue that these generic statements for expanding
maps can be transferred to Anosov diffeormorphisms via the natural
extension. However there are several problems with this approach which
could be summarised into roughly two parts:
\begin{itemize}
\item The natural extension construction is an abstract theorem and in many
cases it is not clear if it has an Anosov model. See \cite{VYa} for
constructions of smooth natural extensions. 
\item In order that the natural extension will be conservative, the expanding
map has to be recurrent \cite[Th. 4.4]{S-T} and a generic $C^{1}$
expanding map is not recurrent. 
\end{itemize}
Another natural approach in finding $C^{1}$ examples with a certain
property is to prove that the property is generic in the $C^{1}$
topology, see for example \cite{Bonatti Crovisier Wilkinson,Avila Bochi generic expanding (interval maps),Avilla Bochi generic C^1 ...}.
However since by the result of Sinai and Livsic, a generic $C^{1}$
Anosov map is dissipative, it is not clear to us how to use this approach
to find a conservative example without an a.c.i.m. Nonetheless we
prove the following. 
\begin{thm}
There exists a $C^{1}$- Anosov diffeomorphism of the two torus $\mathbb{T}^{2}$
which is ergodic, conservative and there exists no $\sigma$-finite
invariant measure which is absolutely continuous with respect to the
Lebesgue measure on $\mathbb{T}^{2}.$
\end{thm}
In fact, the ergodic type ${\rm III}$ transformations ( a transformation
without an a.c.i.m.) can be further decomposed into the Krieger Araki-Woods
classes ${\rm III}_{\lambda},\ 0\leq\lambda\leq1$ \cite{Kri} , see
Section \ref{sec:Prelminary-definition-and}, and our examples are
of type ${\rm III}_{1}$.

The examples are constructed by modifying a linear Anosov diffeomorphism
to obtain a change of coordinates which takes the Lebesgue measure
to a measure which is equivalent to a type ${\rm III}$ Markovian
measure (on a Markov partition of the linear diffeomorphism). These
examples are greatly inspired by the ideas of Bruin and Hawkins \cite{Bruin Hawkins}
where they modify the map $f(x)=2xmod1$ using the push forward (with
respect to the dyadic representation) of a Hamachi product measure
on $\{0,1\}^{\NN}$ to the circle. Since by embedding a horseshoe
in a linear transformation one loses explicit formula for the Radon
Nykodym derivatives of the modified transformations we couldn't use
measures on a full shift space but rather measures supported on topological
Markov shifts. The measures which play the role of the Hamachi measures
in our construction are the type ${\rm III}_{1}$ (for the shift)
inhomogeneous Markov measures.

This paper is organized as follows. In Section \ref{sec:Prelminary-definition-and}
we start by introducing the definitions and background material from
nonsingular ergodic theory and smooth dynamics which are used in this
paper. We end this section with a discussion on the method of the
construction. Section \ref{sec:Type--Markov} presents the inductive
construction of the type ${\rm III}_{1}$ Markov shift examples. In
Section \ref{sec:Type III perturbations of the Golden Mean Shift}
we show how to use the one sided Markov measures from the previous
section to obtain a modification of the golden mean shift. In section
\ref{sec:Type III Anosov} we show how to embed and modify the one
dimensional perturbations of the previous sections to obtain homeomorphisms
of the two torus, which when applied as conjugation to a certain total
automorphism (the natural extension of the golden mean shift)  are
examples of type ${\rm III}_{1}$ Anosov diffeomorphisms. Finally
in the appendix we prove that these Markovian measures satisfy the
aforementioned properties (ergodic, conservative and type ${\rm III}_{1}$).

\section{\label{sec:Prelminary-definition-and}Preliminary definitions and
a discussion on the method of construction}

\subsection{Basics of nonsingular ergodic theory }

This subsection is a very short introduction to nonsingular ergodic
theory. For more details and explanations please see \cite{Aar}.

Let $\left(X,\BB,\mu\right)$ be a standard probability space. In
what follows equalities (and inclusions) of sets are modulo the measure
$\mu$ on the space. A measurable map $T:X\to X$ is \textit{nonsingular}
if $T_{*}\mu:=\mu\circ T^{-1}$ is equivalent to $\mu$ meaning that
they have the same collection of negligible sets. If $T$ is invertible
one has the Radon Nykodym derivatives 
\[
\left(T^{n}\right)'(x):=\frac{d\mu\circ T^{n}}{d\mu}(x):\ X\to\RR_{+},
\]
A set $W\subset X$ is \textit{wandering} if $\left\{ T^{n}W\right\} _{n\in\ZZ}$
are pairwise disjoint and as was stated before we say that $T$ is
\textit{conservative} if there exists no wandering set of positive
measure. By the Halmos' Recurrence Theorem a transformation is conservative
if and only if it satisfies Poincare recurrence, that is given a set
of positive measure $A\in\BB$, almost every $x\in A$ returns to
itself infinitely often. A transformation $T$ is \textit{ergodic}
if there are no non trivial $T$ invariant sets. That is $T^{-1}A=A$
implies $A\in\{\emptyset,X\}$. 

We end this subsection with the definition of the \textit{Krieger
ratio set} $R(T)$. We say that $r\geq0$ is in $R(T)$ if for every
$A\in\BB$ of positive $\mu$ measure and for every $\epsilon>0$
there exists an $n\in\ZZ$ such that 
\[
\mu\left(A\cap T^{-n}A\cap\left\{ x\in X:\left|\left(T^{n}\right)'(x)-r\right|<\epsilon\right\} \right)>0.
\]
The ratio set of an ergodic measure preserving transformation is a
closed multiplicative subgroup of $[0,\infty)$ and hence it is of
the form $\{0\},\{1\},\{0,1\},\{0\}\cup\left\{ \lambda^{n}:n\in\ZZ\right\} $
for $0<\lambda<1$ or $[0,\infty)$. Several ergodic theoretic properties
can be seen from the ratio set. One of them is that $0\in R(T)$ if
and only if there exists no $\sigma$-finite $T$-invariant $\mu$-
a.c.i.m. Another interesting relation is that $1\in R(T)$ if and
only if $T$ is conservative (Maharams Theorem). If $R(T)=[0,\infty)$
we say that $T$ is of type ${\rm III}_{1}$. 

\subsection{Anosov diffeomorphisms and topological Markov shifts}

Smooth dynamics deals with the case where $M$ is a Riemannian manifold
and $f:M\to M$ is a diffeomorphism on $M$. In this paper we would
only talk about the class of Anosov (uniformly hyperbolic) automorphisms.
A diffeomorphism $f$ is Anosov if for every $x\in M$ there is a
decomposition of the tangent bundle at $x$, $T_{x}M=E_{x}^{s}\oplus E_{x}^{u}$,
such that 
\begin{itemize}
\item The decomposition is $D_{f}$- equivariant, here $D_{f}$ denotes
the differential of $f$. That is $\left(D_{f}\right)_{x}\left(E_{x}^{s}\right)=E_{f(x)}^{s}$
and $\left(D_{f}\right)_{x}\left(E_{x}^{u}\right)=E_{f(x)}^{u}$. 
\item There exists $0<\lambda<1$ and $C>0$ so that 
\[
\left\Vert \left(D_{f^{n}}\right)_{x}v\right\Vert \leq C\lambda^{n}\left\Vert v\right\Vert ,\ \text{for every }v\in E_{x}^{s},\ n\geq0
\]
and 
\[
\left\Vert \left(D_{f^{-n}}\right)_{x}u\right\Vert \leq C\lambda^{n}\|u\|,\ \text{for every }u\in E_{x}^{u},\ n\geq0.
\]
\end{itemize}
A \textit{topological Markov shift} (TMS) on $S$ is the shift on
a shift invariant subset $\Sigma\subset S^{\mathbb{Z}}$ of the form 

\[
\Sigma_{A}:=\left\{ x\in S^{\ZZ}:A_{x_{i},x_{i+1}}=1\right\} ,
\]
where $A=\left\{ A_{s,t}\right\} _{s,t\in S}$ is a $\{0,1\}$ valued
matrix on $S$. A TMS is mixing if there exists $n\in\NN$ such that
$A_{s,t}^{n}>0$ for every $s,t\in S$. 

Markov partitions of the manifold $M$ as in \cite{Adler Weiss,Si,Bowen Rufus,Adler}
are an important tool in the study of $C^{1+\alpha}$ Anosov diffeomorphisms.
They provide a semiconjugacy between a TMS and the Anosov transformation
$f$. One of the important contributions of this paper is that it
uses a connection between inhomogeneous Markov chains supported on
a TMS to the Anosov diffeomorphism with the push forward of the Markov
measure. 
\begin{example}
\label{Example of a TMS}Consider $f:\mathbb{T}^{2}\to\mathbb{T}^{2}$
the Toral automorphism defined by 
\[
f(x,y)=(\{x+y\},x)=\left(\begin{array}{cc}
1 & 1\\
1 & 0
\end{array}\right)\binom{x}{y}\ \mod1,
\]
where $\{t\}$ is the fractional part of $t$. Since $\left|\det\left(\begin{array}{cc}
1 & 1\\
1 & 0
\end{array}\right)\right|=1$, $f$ preserves the Lebesgue measure on $\mathbb{T}^{2}$. In addition
for every $z\in\mathbb{T}^{2}$, the tangent space can be decomposed
as $\text{span}\left\{ v_{s}\right\} \oplus{\rm span}\left\{ v_{u}\right\} $
where $v_{u}=\left(1,1/\varphi\right)$ and $v_{s}=\left(1,-\varphi\right).$
Here and throughout the paper $\p$ denotes the golden mean ($\p:=\frac{1+\sqrt{5}}{2}$).
\end{example}
For every $w\in V_{u}:=\text{span}\left\{ v_{u}\right\} $ and $z\in\mathbb{T}^{2}$
\[
D_{f}(z)w=\left(\begin{array}{cc}
1 & 1\\
1 & 0
\end{array}\right)w=\varphi w,
\]
 For every $u\in V_{s}:=\text{span}\left\{ v_{s}\right\} $ and $z\in\mathbb{T}^{2},$
$D_{f}(z)u=\left(-\frac{1}{\varphi}\right)u$. These facts can be
used \cite{Adler,Adler Weiss} to construct the Markov partition for
$f$ with three elements $\{R_{1},R_{2},R_{3}\}$, see figure \ref{fig:The-construction-of}.

\begin{figure}[h]
\includegraphics[scale=0.5]{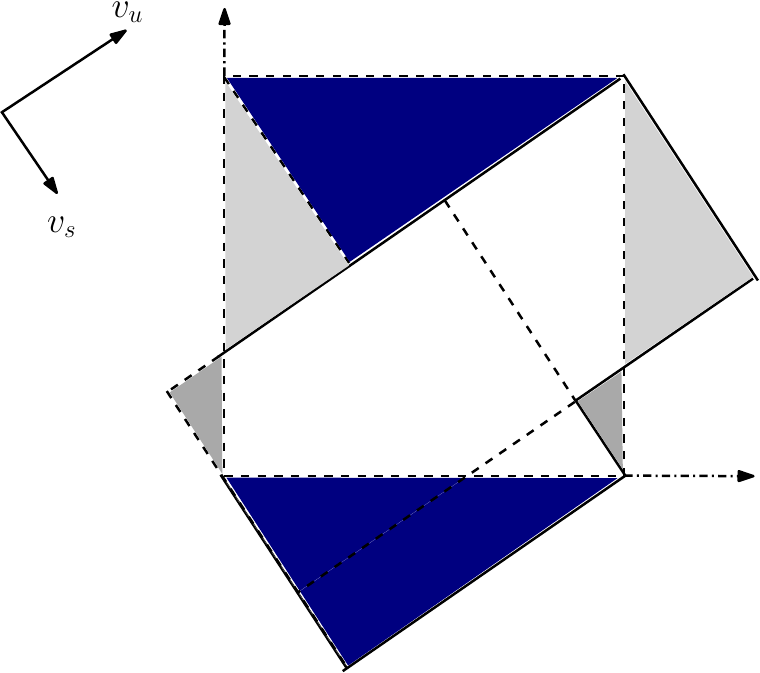}\qquad{}\includegraphics[scale=0.65]{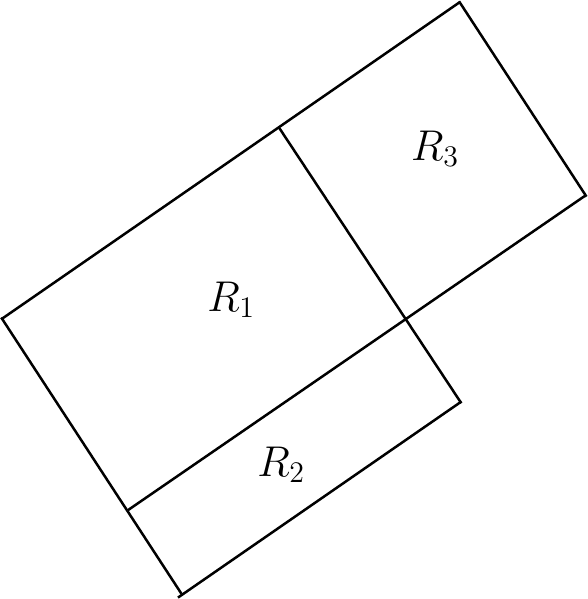}\caption{The construction of t\label{fig:The-construction-of}he Markov partition}
\end{figure}

The adjacency Matrix of the Markov partition is then defined by $A_{i,j}=1$
if and only if ${\rm int}R_{i}\cap f^{-1}\left({\rm int}R_{j}\right)\neq\emptyset$.
Here the adjacency Matrix is 
\[
{\bf A}=\left(\begin{array}{ccc}
1 & 0 & 1\\
1 & 0 & 1\\
0 & 1 & 0
\end{array}\right).
\]
Let $\Phi:\Sigma_{A}\to\mathbb{T}^{2}$ be the map defined by $\Phi(x):=\bigcap_{n=-\infty}^{\infty}\overline{f^{-n}R_{x_{n}}}$.
Note that by the Baire Category Theorem since $\left\{ \cap_{n=-N}^{N}\overline{f^{-n}R_{x_{n}}}\right\} _{N=1}^{\infty}$
is a decreasing sequence of compact sets, $\Phi(x)$ is well defined.
The map $\Phi:\Sigma_{{\bf A}}\to\mathbb{T}^{2}$ is continuous, finite
to one, and for every $x\in\Sigma_{{\bf A}}$, 
\[
\Phi\circ T(x)=f\circ\Phi(x).
\]
In other words, $\Phi$ is a semi-conjugacy (topological factor map)
between $\left(\Sigma_{{\bf A}},T\right)$ to $\left(\mathbb{T}^{2},f\right)$.
In addition, for every $x\in\TT^{2}\backslash\cup_{n\in\ZZ}\cup_{i=1}^{3}f^{-n}\left(\partial R_{i}\right)$
there exists a unique $w\in\Sigma_{{\bf A}}$ so that $\Phi(w)=x$
. The Lebesgue measure $m_{\mathbb{T}^{2}}$ on $\mathbb{T}^{2}$
is invariant under $f$. One can check that $m_{\mathbb{T}^{2}}\left(\cup_{i=1}^{3}\partial R_{i}\right)=0$
and thus $\Phi^{-1}$ defines an isomorphism between $\left(\mathbb{T}^{2},m_{\mathbb{T}^{2}},f\right)$
and $\left(\Sigma_{{\bf A}},\mu_{\pi_{{\bf Q}},{\bf Q}},T\right)$
where $\mu_{\pi_{{\bf Q}},{\bf Q}}$ is the stationary Markov measure
with 
\begin{equation}
P_{j}\equiv\mathbf{Q}:=\left(\begin{array}{ccc}
\frac{\varphi}{1+\varphi} & 0 & \frac{1}{1+\varphi}\\
\frac{\varphi}{1+\varphi} & 0 & \frac{1}{1+\varphi}\\
0 & 1 & 0
\end{array}\right)\label{eq: Lebesgue Transition function}
\end{equation}
and 
\begin{equation}
\pi_{j}=\mathbf{{\bf \pi_{Q}}}:=\left(1/\sqrt{5},\ 1/\p\sqrt{5},\ 1/\p\sqrt{5}\right)=\left(m_{\mathbb{T}^{2}}\left(R_{1}\right),m_{\mathbb{T}^{2}}\left(R_{2}\right),m_{\mathbb{T}^{2}}\left(R_{3}\right)\right).\label{eq:Lebesgue stationary dist.}
\end{equation}

\vspace{-3.5mm}

\subsubsection{Nonsingular Markov shifts: }

Let $\left\{ P_{n}\right\} _{n=-\infty}^{\infty}\subset M_{S\times S}$
be a sequence of aperiodic and irreducible stochastic matrices on
$S.$ In addition let $\left\{ \pi_{n}\right\} _{n=-\infty}^{\infty}$
be a sequence of probability distributions on $S$ so that for every
$s\in S$ and $n\in\ZZ$,
\begin{equation}
\sum_{t\in S}\pi_{n-1}(t)\cdot P_{n}\left(t,s\right)=\pi_{n}(s).\label{eq: Kolmogorov's consistency criteria}
\end{equation}
Then one can define a measure on the collection of cylinder sets,
\[
[b]_{k}^{l}:=\left\{ x\in S^{\ZZ}:\ x_{j}=b_{j}\ \forall j\in[k,l]\cap\ZZ\right\} 
\]
by 
\[
\mu\left(\left[b\right]_{k}^{l}\right):=\pi_{k}\left(b_{k}\right)\prod_{j=k}^{l-1}P_{j}\left(b_{j},b_{j+1}\right).
\]
Since the equation (\ref{eq: Kolmogorov's consistency criteria})
is satisfied, $\mu$ satisfies the consistency condition and therefore
by Kolmogorov's extension theorem $\mu$ defines a measure on $S^{\ZZ}$.
In this case we say that $\mu$ is the Markov measure generated by
$\left\{ \pi_{n},P_{n}\right\} _{n\in\ZZ}$ and denote $\mu=\text{M}\left\{ \pi_{n},P_{n}:\ n\in\ZZ\right\} $.
By $\text{M}\left\{ \pi,P\right\} $ we mean the measure generated
by $P_{n}\equiv P$ and $\pi_{n}\equiv\pi$. We say that $\mu$ is
nonsingular for the shift $T$ on $S^{\ZZ}$ if $T_{*}\mu\sim\mu$. 

\subsection{\label{subsec:Non-singular-Markov shifts} Local absolute continuity }

Let $(X,\BB)$ be a measure space and $\mathcal{F}_{n}\subset\BB$
be a filtration of $X$. That is an increasing sequence of $\sigma$-algebras
such that $\mathcal{F}_{n}\uparrow\BB$. The method in \cite{Kabanos-Lipzer-Sh,Shi}
uses ideas from Martingale theory in order to determine whether two
Borel probability measures $\mu,\nu$ are absolutely continuous. 
\begin{defn}
Given a filtration $\left\{ \F n\right\} $, we say that $\nu\ll^{loc}\mu$
($\nu$ is locally absolutely continuous with respect to $\mu$) if
for every $n\in\mathbb{N}$, $\nu_{n}\ll\mu_{n}$ where $\nu_{n}=\nu|_{\F n}.$ 
\end{defn}
Suppose that $\nu\lloc\mu$ w.r.t $\left\{ \F n\right\} $, set $z_{n}:=\frac{d\nu_{n}}{d\mu_{n}}$.
The sequence $z_{n}$ is a nonnegative martingale with respect to
$\mathcal{F}_{n}$ and thus by the martingale convergence theorem
there exists a $[0,\infty]$ valued random variable $z_{\infty}$
such that $\lim_{n\to\infty}z_{n}=z_{\infty}$ a.s. It follows that
if $\nu\lloc\mu$ then $\nu\ll\mu$ if and only if $z_{n}\xrightarrow[n\to\infty]{}z_{\infty}$
in $L^{1}\left(\mu\right)$. The latter holds if and only if the sequence
$\left\{ z_{n}\right\} _{n=1}^{\infty}$ is uniformly integrable meaning
that for all $\epsilon>0$ there exists $M>0$ such that for all $n\in\mathbb{N}$,
$\int z_{n}1_{\left[z_{n}>M\right]}d\mu<\epsilon$. 

\subsection{Sections Overview and explanation of the method of construction }

The idea is as follows, let $f(x,y)=(x+y,x){\rm mod}1$, $\left\{ R_{1},R_{2},R_{3}\right\} $
be the corresponding Markov partition for $f$, $\Sigma_{A}$ the
resulting topological Markov shift and $\Phi:\ \Sigma_{{\bf A}}\to\mathbb{T}^{2}$
the topological semiconjugacy with the shift. In addition $\mathbf{Q}$
will always denote the transition matrix corresponding to the Lebesgue
measure.
\begin{itemize}
\item In Section \ref{sec:Type--Markov} we present an inductive construction
which produces a family of nonatomic inhomogeneous Markov measures
which are fully supported on $\Sigma_{A}\subset\{1,2,3\}^{\mathbb{Z}}$
and are of type ${\rm III}_{1}$.
\item Let $\mu$ be such a Markov measure generated by $\left\{ \pi_{k},P_{k}:\ k\in\mathbb{Z}\right\} $.
Since $\mu$ is conservative $\Phi_{*}\mu$ gives zero measure to
the images of the boundaries of the rectangles of the Markov partition.
The latter property implies that $\Phi$ is an isomorphism of $\left(\TT^{2},\Phi_{*}\mu,f\right)$
and $\left(\Sigma_{{\bf A}},\mu,Shift\right)$ and thus $\left(\TT^{2},\Phi_{*}\mu,f\right)$
is a type ${\rm III}_{1}$ dynamical system. 
\end{itemize}
The type ${\rm III}_{1}$, inhomogeneous Markov measures for the shift
on $\Sigma_{{\bf A}}$ have the additional property that for every
$k\leq0$ the transition matrices of $\mu$ at $k$ are the same as
the ones arising from the Lebesgue measure ($\forall k\leq0,\ P_{k}={\bf Q}$).
This implies that (after a rotation of the coordinates to the $v_{u},v_{s}$
coordinates) with $\Phi:\Sigma_{{\bf A}}\to\TT$ being the semiconjugacy
map arising from the Markov partition we have 
\[
d\Phi_{*}\mu(x,y)=d\nu^{+}(x)dy.
\]
Here $\nu^{+}$ is the image by the push forward on the stable manifold
of the Markov measure on $\{1,2,3\}^{\NN}$ given by $\left\{ \pi_{k},P_{k}\right\} _{k=1}^{\infty}$.
This property will be used (see Subsection \ref{subsec: NE of phix mod1})
to show that there exists an homeomorphism $G$ of $\mathbb{T}^{2}$
such that $m_{\TT^{2}}\circ G=\Phi_{*}\mu$ and the transformation
$G\circ f\circ G^{-1}:\left(\mathbb{T}^{2},m_{\TT^{2}}\right)\to\left(\mathbb{T}^{2},m_{\TT^{2}}\right)$
is measure theoretically isomorphic to $\left(\mathbb{T}^{2},\Phi_{*}\mu,f\right)$\footnote{The isomorphism $\left(\mathbb{T}^{2},m_{\TT^{2}}\circ G=\Phi_{*}\mu,f\right)\xrightarrow{\pi}\left(\mathbb{T}^{2},m_{\TT^{2}},G\circ f\circ G^{-1}\right)$
is clearly $\pi=G$. Indeed $G$ is a homeomorphism hence measurable
and inverible (and $G^{-1}$ is measurable), $G\circ f=\left(GfG^{-1}\right)\circ G$
and $\left(m_{\mathbb{T}^{2}}\circ G\right)\circ G^{-1}=m_{\mathbb{T}^{2}}$. }, hence a type ${\rm III}_{1}$ system. 

The harder part in the proof of this theorem is to construct a homeomorphism
$H:\mathbb{T}^{2}\to\mathbb{T}^{2}$ so that 
\begin{enumerate}
\item $m\circ H\sim\Phi_{*}\mu=m\circ G$. Consequently the system $\left(\mathbb{T}^{2},\mathcal{B}_{\mathbb{T}^{2}},m_{\mathbb{T}^{2}},H\circ f\circ H^{-1}\right)$
is of type ${\rm III}_{1}$ because it is measure theoretically isomorphic
to $\left(\mathbb{T}^{2},\mathcal{B}_{\mathbb{T}^{2}},m_{\mathbb{T}^{2}}\circ H,f\right)$
and the fact that the type ${\rm III}_{1}$ property is invariant
upon changing the measure to an equivalent measure. 
\item $H\circ f\circ H{}^{-1}$ is $C^{1}$ and Anosov. 
\end{enumerate}
In order to obtain this goal and to explain the definition of $G$
it is easier for us to build $f$ as the natural extension of (the
non invertible) golden mean shift $Sx=\varphi x{\rm mod}1$. 

\subsection{The map $f$ as the natural extension of the golden mean shift\label{subsec: NE of phix mod1}}

The partition $\left\{ J_{1}=[0,1/\p^{2}],\ J_{2}=[1/\p,1],\ J_{3}=\left[1/\p^{2},1/\p\right]\right\} $
is a Markov partition for the golden mean shift with ${\bf A}$ (the
same matrix as the one for $f$) as its adjacency matrix. See figure
\ref{fig:The-Markov-partition for the golden mean shift}. 

\begin{figure}[h]
\includegraphics[scale=0.5]{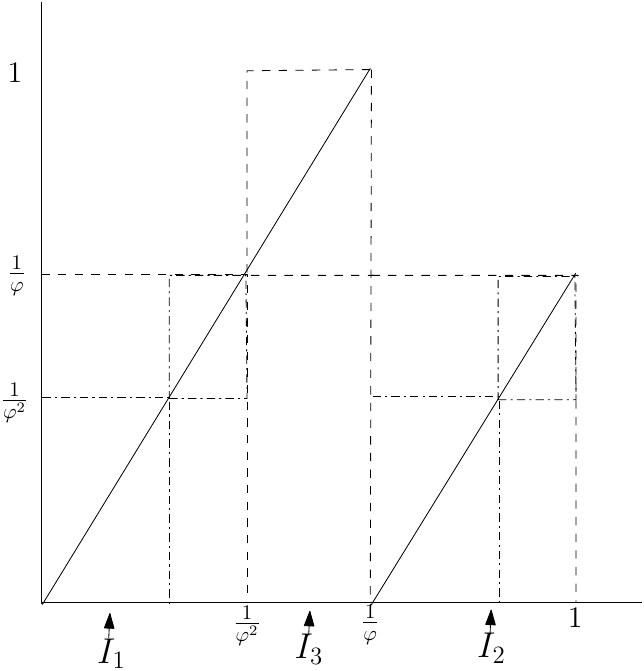}

\caption{T\label{fig:The-Markov-partition for the golden mean shift}he Markov
partition of $\protect\p x\mod1$}
\end{figure}

Denote by $\sigma$ the one sided shift on $\Sigma_{{\bf A}}^{+}$.
It can be verified that $\left(\Sigma_{{\bf A}}^{+},\nu_{\pi_{{\bf Q}},{\bf Q}},\sigma\right)$
is isomorphic to $\left(\mathbb{T},m_{{\rm \TT}},S\right)$ where
$m_{\TT}$ is the Lebesgue measure on $\mathbb{T}$. The natural extension
of $\left(\Sigma_{{\bf A}}^{+},\nu_{\pi_{{\bf Q}},{\bf Q}},\sigma\right)$
is $\left(\Sigma_{{\bf A}},\M{\pi_{{\bf Q}},{\bf Q}},\sigma\right)$
which is isomorphic to $\left(\mathbb{T}^{2},m_{{\rm Leb}},f\right)$.
This shows that $f$ is indeed the natural extension of the Golden
mean shift. To see the geometric picture of how $S$ and $f$ are
related one can look at the Markov partitions and move to the $V_{u},V_{s}$
coordinates. On those coordinates $f$ acts almost as 
\[
(u,v)\mapsto\left(\p u\mod1,-\p^{-1}v\right)=\left(Su,-\p^{-1}v\right),
\]
where the mistake is in the second coordinate. To make it precise
let $\MM=\left[0,1/\p\right]\times\left[-\p/\left(\p+2\right),\p^{2}/\left(\p+2\right)\right]\bigcup\left[1/\p,1\right]\times\left[-\p/\left(\p+2\right),1/\left(\p+2\right)\right]$.
Define $\tilde{f:}\MM\to\MM$ by 
\[
\tilde{f}(x,y)=\begin{cases}
\left(\p x,-\p^{-1}y\right), & 0\leq x\leq1/\p\\
\left(\p x-1,-\p^{-1}\left(y-\frac{\p^{2}}{\p+2}\right)\right), & 1/\p\leq x\leq1
\end{cases}.
\]
See Figure \ref{fig:Action-of-} for the way $\tilde{f}$ maps its
3 rectangles, as can be seen by this picture the action of $\tilde{f}$
is the same as how $f$ acts on its Markov partition.

\begin{figure}[h]
\includegraphics[scale=0.7]{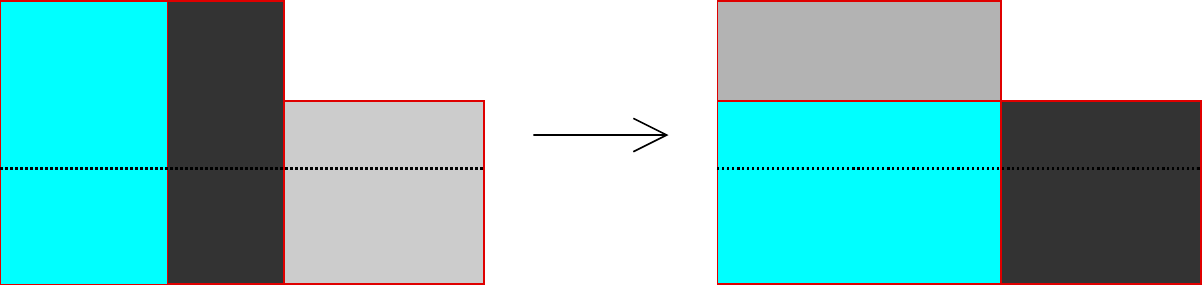}

\caption{\label{fig:Action-of-}Action of $\tilde{f}$ on it's (soon to be)
Markov partition}
\end{figure}

In order that $\tilde{f}$ will be the same as $f$, we identify by
orientation preserving piecewise translations the following intervals
(for a geometric understanding one can see that this identification
comes from the way the Markov partition of $f$ tiles the plane):
\begin{eqnarray*}
\{0\}\times\left[0,\p^{2}/\left(\p+2\right)\right] & \simeq & \{1\}\times\left[-\p/\left(\p+2\right),1/\left(\p+2\right)\right]\\
\{0\}\times\left[-\p/\left(\p+2\right),0\right] & \simeq & \left\{ 1/\p\right\} \times\left[1/\left(\p+2\right),\p^{2}/\left(\p+2\right)\right]\\
\left[0,1/\p^{3}\right]\times\left\{ \p^{2}/\left(\p+2\right)\right\}  & \simeq & \left[1/\p^{2},1/\p\right]\times\left\{ -\p/\left(\p+2\right)\right\} \\
\left[1/\p^{3},1/\p\right]\times\left\{ \p^{2}/\left(\p+2\right)\right\}  & \simeq & \left[1/\p,1\right]\times\left\{ -\p/\left(\p+2\right)\right\} \\
\left[0,1/\p^{2}\right]\times\left\{ -\p/\left(\p+2\right)\right\}  & \simeq & \left[1/\p,1\right]\times\left\{ 1/\left(\p+2\right)\right\} .
\end{eqnarray*}
The resulting Manifold (which is $\TT^{2}$) will be denoted by $\MM_{\sim}$
in order to remind the reader of this change of coordinates and the
geometric relation between $f$ and $S$. 

In $\MM_{\sim}$, $d\Phi_{*}\mu(x,y)=d\nu^{+}(x)dy$ where $\nu^{+}$
is a non atomic measure on $\TT$. This means that the circle homeomorphism
$h(x)=\nu^{+}[0,x]$ takes the Lebesgue measure on $\TT$ to $\nu^{+}$
and $h\left(1/\p\right)=\mu\left(x_{1}\neq2\right)=\frac{1}{\varphi}$.
The homeomorphism of $\MM_{\sim}$ defined by $G(x,y)=(h(x),y)$ takes
Lebesgue measure of $\MM_{\sim}$ to $\Phi_{*}\mu$. The perturbed
homeomorphism $H:\MM_{\sim}\to\MM_{\sim}$ which will be constructed
is of the form $H(x,y)=\left(h_{y}(x),y\right)$, where for $y\in\left[-\p/\left(\p+2\right),\p^{2}/\left(\p+2\right)\right]$,
$h_{y}:\TT\to\TT$ is a circle homeomorphism such that $m_{\TT}\circ h_{y}\sim\nu^{+}$.
This construction is carried out by the following steps: 
\begin{itemize}
\item the first step is to work on the action of $f$ on the unstable manifold
which is the Golden mean shift and to construct a circle homeomorphism
$\tilde{h}$ such that $\tilde{h}\circ S\circ\tilde{h}^{-1}$ is $C^{1}$
expanding and $m_{\TT}\circ\tilde{h}\sim\nu^{+}$. A further important
property of the homeomorphisms which we construct is that $\tilde{h}\left(J_{i}\right)=J_{i}$
for all elements of the Markov partition of $S$. This will imply
for example that $\left(\tilde{h}(x),y\right)$ is an homeomorphism
of $\MM_{\sim}$. This step involves adding another parameter for
the inductive construction of the measure $\mu=\M{P_{k},\pi_{k}:\ k\in\ZZ}$
and is carried out in Section \ref{sec:Type III perturbations of the Golden Mean Shift}.
\item in Section \ref{sec:Type III Anosov} we modify construction of these
homeomorphisms $\tilde{h}$ in order to construct the functions $h_{y}$
in the definition of $H$. A major challenge in this step is to ensure
that $\frac{\partial H\circ\tilde{f}\circ H^{-1}}{\partial y}$ is
defined and continuous. 
\end{itemize}

\section{\label{sec:Type--Markov}The type ${\rm {\rm III}_{1}}$ Markov shifts
supported on $\Sigma_{A}$}

Here we present the inductive construction of the inhomogeneous Markov
measures. 

\subsection{\label{subsec:Markov-Chains}Markov Chains}

\subsubsection{Basics of Stationary (homogenous) Chains}

Let $S$ be a finite set which we regard as the state space of the
chain, $\pi=\left\{ \pi(s)\right\} _{s\in S}$ a probability vector
on $S$ and ${\rm P}=\left(P_{s,t}\right)_{s,t\in S}$ a stochastic
matrix. The vector $\pi$ and $\P$ define a Markov chain $\left\{ X_{n}\right\} $
on $S$ by 
\[
\forall n\in\ZZ,\ \ \PP_{\pi}\left(X_{n}=t\right)\pi(t)\ \ \text{\ and\ \ }\ \PP\left(X_{n}=s\left|X_{1},..,X_{n-1}\right.\right):=P_{X_{n-1},s}.
\]
${\rm P}$ is \textit{irreducible} if for every $s,t\in S$, there
exists $n\in\NN$ such that $P_{s,t}^{n}>0$ and ${\rm P}$ is \textit{aperiodic}
if for every $s\in S$, ${\rm gcd}\left\{ n:\ P_{s,s}^{n}>0\right\} =1$.
Given an irreducible and aperiodic ${\rm P}$, there exists a unique
stationary probability $\pi_{{\rm P}}$ (that is $\pi_{P}\P=\pi_{P}$).
In addition for every $s,t\in S$, $P_{s,t}^{n}\xrightarrow[n\to\infty]{}\pi_{P}(t)$.
Since $S$ is a finite state space, it follows that for any initial
distribution $\pi$ on $S$, 
\[
\PP_{\pi}\left(X_{n}=t\right)=\sum_{s\in S}\pi(s)P_{s,t}^{n}\xrightarrow[n\to\infty]{}\pi_{P}(t).
\]
An important fact which will be used in the sequel is that the stationary
distribution is continuous with respect to the stochastic matrix.
That is if $\left\{ {\rm P_{n}}\right\} _{n=1}^{\infty}$ is a sequence
of irreducible and aperiodic stochastic matrices such that 
\[
\|{\rm P}_{n}-{\rm P\|_{\infty}:=}\max_{s,t\in S}\left|\left(\P_{n}\right)_{s,t}-\P_{s,t}\right|\xrightarrow[n\to\infty]{}0
\]
and ${\rm P}$ is irreducible and aperiodic then $\left\Vert \pi_{{\rm P}_{n}}-\pi_{{\rm P}}\right\Vert _{\infty}\to0$. 

\subsection{\label{subsec:Type III_1 Markov shifts}Type ${\rm III_{1}}$ Markov
Shifts}

In this subsection, let $\Omega:=\Sigma_{{\bf A}}$, $\mathcal{B}:=\mathcal{B}_{\Sigma_{{\bf A}}}$
and $T$ is the two sided shift on $\Omega$. For two integers $k<l$,
write $\mathcal{F}\left(k,l\right)$ for the algebra of sets generated
by cylinders of the form $[b]_{k}^{l},\ b\in\{1,2,3\}^{l-k}$. That
is the smallest $\sigma$-algebra which makes the coordinate mappings
$\left\{ w_{i}(x):=x_{i}:i\in[k,l]\right\} $ measurable. 

\subsubsection{Idea of the construction of the type ${\rm III}$ Markov measure. }

The construction uses the ideas in \cite{Kos1}. For every $j\leq0$
\[
P_{j}\equiv\mathbf{Q}\ {\rm and}\ \pi_{j}\equiv\pi_{\mathbf{Q}},
\]
where $\mathbf{Q}$ and $\pi_{\mathbf{Q}}$ are as in (\ref{eq: Lebesgue Transition function})
and (\ref{eq:Lebesgue stationary dist.}) respectively. On the positive
axis one defines on larger and larger chunks the stochastic matrices
which depend on a distortion parameter $\lambda_{k}\geq1$ where $1$
means no distortion. Now a cylinder set $[b]_{-n}^{n}$ fixes the
values of the first $n$ terms in the product form of the Radon Nykodym
derivatives. We would like to be able to correct the values in order
that we can enforce a given number to be in the ratio set. This corresponds
to a lattice condition on $\lambda_{k}$ which is less straightforward
then the one in \cite{Kos1}. However this is not enough for a Markov
measure, since the states are not independent, this forces us to utilize
both the convergence to the stationary distribution and the mixing
property for stationary chains. 

Another difficulty is that the measure of the set $[b]_{-n}^{n}\cap T^{-N}[b]_{-n}^{n}\cap\left\{ \left(T^{N}\right)'\approx a\right\} $
could be of very small measure with respect to $\mu\left([b]_{-n}^{n}\right)$.
To remedy this problem, and enable approximation of general sets,
we look for many approximately independent such events so that their
union covers at least a fixed proportion of $[b]_{-n}^{n}$. 

More specifically the construction goes as follows. We define inductively
$5$ sequences $\left\{ \lambda_{j}\right\} $, $\left\{ m_{j}\right\} $,
$\left\{ n_{j}\right\} $, $\left\{ N_{j}\right\} $ and $\left\{ M_{j}\right\} $
where 
\begin{eqnarray*}
M_{0} & = & 1\\
N_{j} & := & N_{j-1}+n_{j}\\
M_{j} & := & N_{j}+m_{j}.
\end{eqnarray*}
 This defines a partition of $\mathbb{N}$ into segments $\left\{ \left[M_{j-1},N_{j}\right),\ \left[N_{j},M_{j}\right)\right\} _{j=1}^{\infty}$.
The sequence $\left\{ P_{n}\right\} $ equals $\mathbf{Q}$ on the
$\left[N_{j},M_{j}\right)$ segments while on the $\left[M_{j-1},N_{j}\right)$
segments we have $P_{n}\equiv\mathbf{Q}_{\lambda_{j}}$, the $\lambda_{j}$
perturbed stochastic matrix. The $\mathbf{Q}$ segments facilitate
the form of some of the Radon Nykodym derivatives while the perturbed
segments come to ensure that $\mu\perp\M{\pi_{{\bf Q}},{\bf Q}}$
and that the ratio set condition is satisfied for cylinder sets.

Notation: By $x=a\pm b$ we mean $a-b\leq x\leq a+b$. 

\subsubsection{The construction.\label{subsec:The-construction.}}

For $\lambda\geq1$ let 
\[
\mathbf{Q_{\lambda}}:=\left(\begin{array}{ccc}
\frac{\varphi\lambda}{1+\varphi\lambda} & 0 & \frac{1}{1+\varphi\lambda}\\
\frac{\varphi}{1+\varphi} & 0 & \frac{1}{1+\varphi}\\
0 & 1 & 0
\end{array}\right).
\]

Choice of the base of induction: Let $M_{0}=1,\ \lambda_{1}>1$, $n_{1}=2$,
$N_{1}=3$ and ${\bf Q}_{1}:={\bf Q}_{\lambda_{1}}$ be the $\lambda_{1}$
perturbed matrix. Set $P_{1}=P_{2}=\mathbf{Q}_{\mathbf{1}}$ and $\pi_{0}=\pi_{\mathbf{Q}}$.
The measures $\pi_{1},\pi_{2}$ are then defined by equation (\ref{eq: Kolmogorov's consistency criteria}).
Let $m_{1}=3$ and thus $M_{1}=6.$ Set $P_{j}=\mathbf{Q}$ for $j\in\left[N_{1},M_{1}\right)=\left[3,6\right)$
and $\pi_{3},\pi_{4},\pi_{5}$ be defined by equation (\ref{eq: Kolmogorov's consistency criteria}). 

Assume that $\left\{ \lambda_{j},m_{j},n_{j},N_{j},M_{j}\right\} _{j=1}^{l-1}$
have been chosen. 

\textbf{Choice of $\lambda_{l}$:} Notice that the function $f(x):=x\frac{1+\varphi}{1+\varphi x}$
is monotone increasing and continuous in the segment $[1,\infty)$.
Therefore we can choose $\lambda_{l}>1$ which satisfies the following
three conditions:
\begin{enumerate}
\item Finite approximation of the Radon-Nykodym derivatives condition: 
\begin{equation}
\left(\lambda_{l}\right)^{2m_{l-1}}<e^{\frac{1}{2^{l}}}.\label{eq:condition on Lambda}
\end{equation}
This condition ensures an approximation of the derivatives by a finite
product. 
\item Lattice condition: 
\begin{equation}
\lambda_{l-1}\cdot\frac{1+\varphi}{1+\varphi\lambda_{l-1}}\in\left(\lambda_{l}\cdot\frac{1+\varphi}{1+\varphi\lambda_{l}}\right)^{\NN},\label{eq: Lattice condition on lambda}
\end{equation}
where $a^{\NN}:=\left\{ a^{n}:n\in\NN\right\} $.
\item Let ${\bf Q}_{l}:={\bf Q}_{\lambda_{l}}$ and $\pi_{\mathbf{Q_{l}}}$
be its unique stationary probability. Notice that when $\lambda_{l}$
is close to $1$, then $\mathbf{Q_{l}}$ is close to $\mathbf{Q}$
in the $L_{\infty}$ sense. Therefore by continuity of the stationary
distribution we can demand that 
\begin{equation}
\|\pi_{\mathbf{Q}}-\pi_{\mathbf{Q_{l}}}\|_{\infty}<\frac{1}{2^{l}}.\label{eq: lambda is small so stationary dist. are close}
\end{equation}
\end{enumerate}
\textbf{Choice of $n_{l}$:} It follows from the Lattice condition,
equation (\ref{eq: Lattice condition on lambda}), that for each $k\leq l-1$,
\[
\left(\lambda_{k}\cdot\frac{1+\varphi}{1+\varphi\lambda_{k}}\right)\in\left(\lambda_{l}\cdot\frac{1+\varphi}{1+\varphi\lambda_{l}}\right)^{\NN}.
\]
 Choose $n_{l}$ large enough so that for every $k\leq l-1$ (notice
that the demand on $k=1$ is enough) there exists $\NN\ni p=p(k,l)\leq\frac{n_{l}}{20}$
so that 
\begin{equation}
\left(\lambda_{l}\cdot\frac{1+\varphi}{1+\varphi\lambda_{l}}\right)^{p}=\left(\lambda_{k}\cdot\frac{1+\varphi}{1+\varphi\lambda_{k}}\right).\label{eq: n_t is very large w.r.t lattic condition}
\end{equation}

Till now we have defined $\left\{ P_{j},\pi_{j}\right\} _{j=-\infty}^{M_{l-1}}$.
By the mean ergodic theorem for Markov chains \cite[Th. 4.16]{Levin Peres Wilmer Mixing Times}
and (\ref{eq: lambda is small so stationary dist. are close}), one
can demand by enlarging $n_{l}$ if necessary that in addition 
\begin{equation}
\nu_{\pi_{M_{l-1}},\mathbf{Q_{l}}}\left(x:\ \frac{1}{n_{l}}\sum_{j=1}^{n_{l}}{\bf 1}_{\left[x_{j}=1\right]}=\frac{1}{\sqrt{5}}\pm2^{-l}\right)>1-\frac{1}{l},\label{Ergodic theorem cond 1}
\end{equation}
and 
\begin{equation}
\nu_{\pi_{M_{l-1}},\mathbf{Q_{l}}}\left(x:\ \frac{1}{n_{l}}\sum_{j=1}^{n_{l}}{\bf 1}_{\left[x_{j}=2,x_{j+1}=3\right]}>\frac{1}{15}\right)>1-\frac{1}{l},\label{eq:Ergodic theorem cond 2}
\end{equation}
where $\nu$ is the Markov measure on $\{1,2,3\}^{\NN}$ defined by
and ${\rm \mathbf{Q_{l}}}$ and $\pi_{M_{t-1}}$. The numbers inside
the set were chosen so that
\[
\left|\pi_{\mathbf{Q_{l}}}(1)-1/\sqrt{5}\right|<2^{-l},
\]
and similarly for $l$ large enough 
\begin{eqnarray*}
\int{\bf 1}_{\left[x_{0}=2,x_{1}=3\right]}(x)d\nu_{\pi_{\mathbf{Q_{l}}},\mathbf{Q_{l}}} & = & \pi_{\mathbf{Q_{l}}}(2)\left(\mathbf{Q_{l}}\right)_{2,3}=\left(\frac{1}{\varphi\sqrt{5}}\pm\frac{1}{2^{l}}\right)\frac{1}{\varphi+1}>\frac{1}{15}.
\end{eqnarray*}
\textbf{Choice of $N_{l}$:} Let $N_{l}:=M_{l-1}+n_{l}$. Now set
for all $j\in\left[M_{l-1},N_{l}\right)$, 
\[
P_{j}=\mathbf{Q_{l}}
\]
 and $\left\{ \pi_{j}\right\} _{j=M_{l-1}+1}^{N_{l}}$ be defined
by equation (\ref{eq: Kolmogorov's consistency criteria}). \\
\textbf{Choice of $m_{l}$:} Let $k_{l}$ be the $\left(1\pm\left(\frac{1}{3}\right)^{3N_{l}}\right)$
mixing time of $\mathbf{Q}$. That is for every ${\bf n}>k_{l}$,
$j\in\mathbb{N}$, $A\in\F{}\left(0,j\right)$, $B\in\mathcal{F}\left(j+{\bf n},\infty\right)$
and initial distribution $\tilde{\pi}$, 
\begin{equation}
\nu_{\tilde{\pi},\mathbf{Q}}\left(A\cap B\right)=\left(1\pm3^{-3N_{l}}\right)\nu_{\tilde{\pi},\mathbf{Q}}\left(A\right)\nu_{\pi_{\mathbf{Q}},\mathbf{Q}}\left(T^{{\bf n}+l}B\right).\label{eq: Mixing time}
\end{equation}
Demand in addition that $k_{l}>N_{l}$. Let $m_{l}$ be large enough
so that 
\begin{equation}
\left(1-9^{-3N_{l}}\right)^{m_{l}/4k_{l}}\leq\frac{1}{l},\label{m_l large so that the sum of bad events is small}
\end{equation}
and 
\begin{equation}
\left(m_{l}-N_{l}\right)\lambda_{1}^{-2N_{l}}\geq1.\label{m_l is large enough for shift conservative}
\end{equation}

To summarize the construction. We have defined inductively sequences
$\left\{ n_{l}\right\} ,\ \left\{ N_{l}\right\} ,\ \left\{ m_{l}\right\} ,\ \left\{ M_{l}\right\} $
of integers which satisfy 
\[
M_{l}<N_{l+1}=M_{l}+n_{l}<M_{l+1}=N_{l+1}+m_{l+1}.
\]
In addition we have defined a monotone decreasing sequence $\left\{ \lambda_{l}\right\} $
which decreases to $1$ and using that sequence we defined new stochastic
matrices $\left\{ \mathbf{Q_{l}}\right\} $. Now we set 
\begin{equation}
P_{j}:=\begin{cases}
\mathbf{Q}, & j\leq0\\
\mathbf{Q_{l}}, & M_{l-1}\leq j<N_{l}\\
\mathbf{Q}, & N_{l}\leq j<M_{l}
\end{cases},\label{eq: def of P's}
\end{equation}
and $\pi_{j}=\pi_{\P}$ for $j\leq0$. The rest of the $\pi_{j}$'s
are defined by the consistency condition, equation (\ref{eq: Kolmogorov's consistency criteria}).
Finally let $\mu$ be the Markovian measure on $\left\{ 1,2,3\right\} ^{\mathbb{Z}}$
defined by $\left\{ \pi_{j},P_{j}\right\} _{j=-\infty}^{\infty}$. 

Notice that for all $j\in\NN$,  ${\rm supp}P_{j}\equiv{\rm supp}A={\rm supp}\mathbf{Q}$. 
\begin{thm}
\label{thm: lambdas are...}The shift $\left(\left\{ 1,2,3\right\} ^{\mathbb{Z}},\mu,T\right)$
is nonsingular,conservative, ergodic and of type ${\rm III}_{1}$. 
\end{thm}
The proof of Theorem \ref{thm: lambdas are...} is given in the appendix. 

\section{Type ${\rm III}$ perturbation of the Golden Mean Shift arising from
Markovian measures \label{sec:Type III perturbations of the Golden Mean Shift}}

\subsection{\label{subsec:A-perturbation-of}A perturbation of the golden mean
shift:}

Let $\nu=\M{\pi_{k},\P_{k}}_{k=-\infty}^{\infty}$ be the type ${\rm III_{1}}$
(for the shift on $\left\{ 1,2,3\right\} ^{\ZZ}$) Markov measure
from Section \ref{sec:Type--Markov} for the two sided shift. It follows
from \cite[Thm. 4.4.]{S-T} that the one sided Markov measure $\nu^{+}=\M{\pi_{k},\P_{k}}_{k=1}^{\infty}$
on $\left\{ 1,2,3\right\} ^{\NN}$ is a type ${\rm III}$ measure
for the (one sided) shift.

Let $Sx=\p x\mod1$ and $J_{1}:=\left(0,1/\p^{2}\right)$, $J_{2}:=\left(1/\p,1\right)$
and $J_{3}:=\left(1/\p^{2},1/\p\right)$ be a Markov partition for
$S$. Denote by 
\[
{\rm Bd}\left(S\right):=\bigcup_{n=0}^{\infty}\bigcup_{i=1}^{3}\partial\left(S^{-n}J_{i}\right).
\]
 The map $\Theta:\Sigma_{{\bf A}}^{+}\to[0,1]$, $\Theta\left(w\right)=\bigcap_{n=0}^{\infty}\overline{S^{-n}J_{w_{n}}}$
is a semiconjugacy of $\left(\Sigma_{{\bf A}}^{+},\sigma\right)$
and $\left(\TT,S\right)$ and for each $x\notin{\rm Bd}\left(S\right)$,
$\Theta^{-1}(x)$ consists of one point (point of uniqueness for the
$\Theta$ representation). Since the support of $\nu^{+}$ is contained
in ${\rm G}\left(\sigma\right):=\Theta^{-1}\left(\TT\backslash{\rm Bd}\left(S\right)\right)$,
the map ${\rm \Theta}$ is a metric isomorphism between $\left(\Sigma_{{\bf A}}^{+},\nu^{+},\sigma\right)$
and $\left(\TT,\Theta_{*}\left(\nu^{+}\right),S\right)$ and therefore
the measure $\mu^{+}:=\Theta_{*}\left(\nu^{+}\right)$ is a type ${\rm III}$
measure for $S$. Since $\mu^{+}$ is a continuous measure, its cumulative
distribution function $\h(x)=\mu^{+}\left(\left[0,x\right]\right)$
is a homeomorphism of $\TT$ such that $\mu^{+}\circ\h^{-1}$ is Lebesgue
measure on $\mathbb{T}$. It follows that the map $\left(\TT,m_{\TT},\h\circ S\circ\h^{-1}\right)$
is a type ${\rm III}$ transformation, where $m_{\text{\ensuremath{\TT}}}$
denotes the Lebesgue measure. The problem is that $\h\circ S\circ\h^{-1}$
is not necessarily smooth, so we construct $\h_{\epsilon}$, as in
the idea of the examples of Bruin and Hawkins, close to $\h$ in the
$C^{0}$ norm such that
\begin{itemize}
\item $\h_{\epsilon}\circ S\circ\h_{\epsilon}^{-1}$ is $C^{1}$ and uniformly
expanding.
\item $m_{\text{\ensuremath{\TT}}}\circ\h_{\epsilon}\sim\mu^{+}$.
\item We will have in addition that $\h_{\mathbb{\epsilon}}J_{i}=J_{i}$
for every $i\in\{1,2,3\}$, this extra property is crucial for the
extension to two dimensions. 
\end{itemize}
Before we go through the construction we would like the reader to
recall that the Lebesgue measure on $\TT$ is the measure arising
from $\M{\pi,{\bf Q}}$. The main idea is to approximate the change
of measure between Lebesgue measure and $\mu^{+}$ on the semi algebras
\[
\mathcal{R}(n):=\left\{ C_{\left[w\right]_{1}^{n}}:=\bigcap_{k=0}^{n-1}S^{-k}J_{w_{k}}:x\in\Sigma_{{\bf A}}\right\} =\left\{ C_{w}:\ w\in\Sigma_{{\bf A}}(n)\right\} .
\]

The construction goes as follows: We first assume that we are given
a type ${\rm III}$ Markovian measure defined by $\left\{ \lambda_{k},M_{k},N_{k}\right\} _{k=1}^{\infty}$.
Then we would like to choose inductively, mostly by continuity arguments
a sequence $\ep=\left\{ \epsilon_{k}\right\} $ that will give us
the perturbation. However in the end we arrive at a problem that we
need that the size of $M_{k}$ is relatively large with respect to
$1/\epsilon_{k-1}$. This problem will be solved by modifying the
induction process of Section \ref{sec:Type--Markov} and adding the
choice of the sequence $\ep$ to the induction. The new induction
will be explained in Subsection \ref{subsec:The-modified-induction}. 
\begin{rem}
Before we continue with the construction we would like to remind the
reader that at each stage in the inductive construction of the Markovian
measure in Section \ref{sec:Type--Markov} we can take $\lambda_{t}$
to be as close to $1$ as we like and $n_{t},\ M_{t}/N_{t}$ to be
as large as we want. This is because the conditions on $\lambda_{t}$
((\ref{eq:condition on Lambda}), (\ref{eq: Lattice condition on lambda})
and (\ref{eq: lambda is small so stationary dist. are close})) are
that $\lambda_{t}$ is small enough whilst the conditions on $n_{t}$
((\ref{eq: n_t is very large w.r.t lattic condition}), (\ref{Ergodic theorem cond 1})
and (\ref{eq:Ergodic theorem cond 2})) and $M_{t}/n_{t}\sim m_{t}/n_{t}$
((\ref{m_l large so that the sum of bad events is small}) and (\ref{m_l is large enough for shift conservative}))
are to be large enough. 
\end{rem}

\subsubsection*{Special interpolation functions: }

Given $\alpha>0$ we would like to define a Lipschitz function $g_{\alpha}$
so that $g_{\alpha}(0)=0$, $g'_{\alpha}(0)=1$, $g_{\alpha}(1)=\int_{0}^{1}g_{\alpha}'(x)dx=\alpha$
and $g_{\alpha}'(1)=\alpha$. We will use the functions $g_{\alpha}:[0,1]\to[0,\alpha]$
defined by $g_{\alpha}(0)=0$ and 
\[
g'_{\alpha}(x)=\begin{cases}
1+3x\cdot\frac{5\alpha-5}{4} & 0\leq x\leq\frac{1}{3}\\
\frac{5\alpha-1}{4}, & \frac{1}{3}\leq x\leq\frac{2}{3}\\
\frac{5\alpha-1}{4}-(3x-2)\frac{\alpha-1}{4}, & \frac{2}{3}\leq x\leq1
\end{cases}
\]
which have the additional property that if $\alpha>1$ then 
\[
1=\inf_{x\in[0,1]}g'_{\alpha}(x)<\sup_{x\in[0,1]}g'_{\alpha}(x)=\frac{5\alpha-1}{4}<\alpha^{2}
\]
and if $\frac{1}{4}\leq\alpha<1$ then 
\[
\alpha^{2}\leq\frac{5\alpha-1}{4}=\inf_{x\in[0,1]}g'_{\alpha}(x)<\sup_{x\in[0,1]}g'_{\alpha}(x)=1.
\]

\begin{rem}
\label{rem: relation on g in the construction of psi}For all $\alpha,\epsilon>0$,
\[
\int_{0}^{\epsilon}g_{\alpha}^{'}\left(\frac{x}{\epsilon}\right)dx=\epsilon\int_{0}^{1}g'_{\alpha}(x)dx=\epsilon\alpha
\]
and for all $u>\epsilon$, 
\[
\int_{u-\epsilon}^{u}g_{\alpha}^{'}\left(\frac{u-x}{\epsilon}\right)dx=\epsilon\alpha
\]
\end{rem}

\subsection{Realization of the homeomorphism of change of measures}

For $0<\epsilon<\frac{1}{\varphi}$ and $\lambda>1$, let $\psi_{\epsilon,\lambda}:\left[0,1\right]\circlearrowleft$
be the function defined by $\psi_{\epsilon,\lambda}(0)=0$ and
\[
\psi_{\epsilon,\lambda}^{'}(x):=\begin{cases}
g_{\left(\frac{\lambda\p^{2}}{1+\lambda\p}\right)}'\left(\frac{x}{\epsilon}\right), & 0\leq x\leq\epsilon\\
\frac{\lambda\p^{2}}{1+\l\p}, & \epsilon<x\leq\frac{1}{\p}-\epsilon\\
g_{\left(\frac{\lambda\p^{2}}{1+\lambda\p}\right)}'\left(\frac{1/\p-x}{\epsilon}\right), & \frac{1}{\p}-\epsilon<x\leq\frac{1}{\p}\\
g_{\left(\frac{\p^{2}}{1+\lambda\p}\right)}'\left(\frac{x-1/\p}{\epsilon}\right), & \frac{1}{\p}\leq x\leq\frac{1}{\p}+\epsilon\\
\frac{\p^{2}}{1+\l\p}, & \frac{1}{\p}+\epsilon<x\leq1-\epsilon\\
g_{\left(\frac{\p^{2}}{1+\lambda\p}\right)}'\left(\frac{1-x}{\epsilon}\right), & 1-\epsilon<x\leq1.
\end{cases}.
\]

\begin{figure}[h]
\includegraphics[scale=0.5]{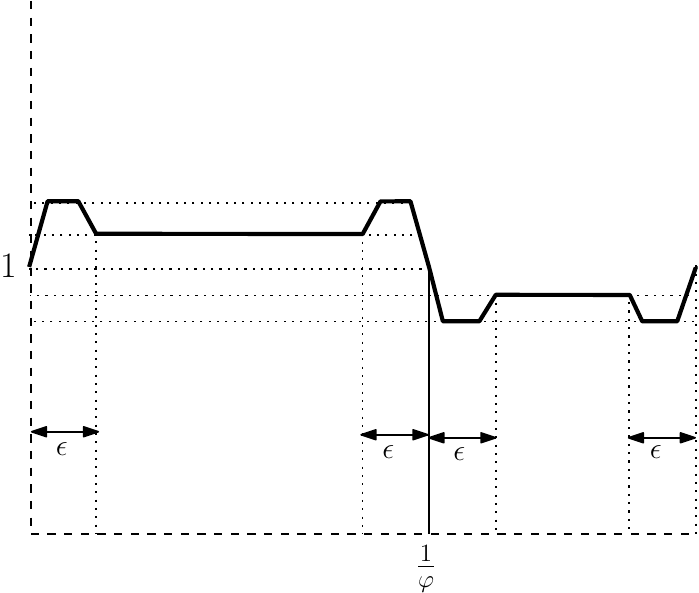}\caption{The graph of $\psi'_{\epsilon,\lambda}$}

\end{figure}

If $\epsilon=0$ then by a rescaling procedure one can use these functions
to define the cumulative distribution function of $\Theta_{*}\left(\nu_{\pi_{{\bf Q}_{\lambda}},{\bf Q}_{\lambda}}\right)$.
The function $\psi_{\epsilon,\lambda}$ is basically an interpolation
of a piecewise constant function in order to make it continuous and
that the following properties hold:
\begin{enumerate}
\item $\psi_{\epsilon,\lambda}'(0)=\psi_{\epsilon,\lambda}'(1)=1$. This
is needed in order to glue $\psi_{\epsilon,\lambda}$ with the identity
function and still have a $C^{1}$ function. 
\item For every $\epsilon,\lambda$, by Remark \ref{rem: relation on g in the construction of psi},
\begin{align*}
\psi_{\epsilon,\lambda}\left(\frac{1}{\varphi}\right) & =\left[\int_{0}^{\epsilon}g_{\left(\frac{\lambda\p^{2}}{1+\lambda\p}\right)}'\left(\frac{x}{\epsilon}\right)dx+\frac{\lambda\p^{2}}{1+\l\p}\left(\frac{1}{\varphi}-2\epsilon\right)\right.\\
\  & \left.\ \ \ +\int_{\frac{1}{\varphi}-\epsilon}^{\frac{1}{\varphi}}g_{\left(\frac{\lambda\p^{2}}{1+\lambda\p}\right)}'\left(\frac{1/\varphi-x}{\epsilon}\right)dx\right]=\frac{\lambda\p}{1+\lambda\varphi}.
\end{align*}
Similarly $\psi_{\epsilon,\lambda}(1)=\psi_{\epsilon,\lambda}\left(\frac{1}{\varphi}\right)+\left(\psi_{\epsilon,\lambda}\left(1\right)-\psi_{\epsilon,\lambda}\left(\frac{1}{\varphi}\right)\right)=1$. 
\item By Remark \ref{rem: relation on g in the construction of psi}, $\psi_{\epsilon,\lambda}(\epsilon)=\frac{\lambda\varphi^{2}}{1+\lambda\varphi}\cdot\epsilon$.
Thus for every $\epsilon<x<\frac{1}{\p}-\epsilon$, 
\[
\psi_{\epsilon,\lambda}(x)=\left(\psi_{\epsilon,\lambda}(x)-\psi_{\epsilon,\lambda}(\epsilon)\right)+\psi_{\epsilon,\lambda}(\epsilon)=\frac{\lambda\varphi^{2}}{1+\lambda\varphi}x,
\]
and
\[
\frac{\psi_{\epsilon,\lambda}(x)}{\psi_{\epsilon,\lambda}(1/\p)}=\p x.
\]
Similarly, $\psi_{\epsilon,\lambda}\left(\frac{1}{\varphi}+\epsilon\right)-\psi_{\epsilon,\lambda}\left(\frac{1}{\varphi}\right)=\frac{\varphi^{2}}{1+\lambda\varphi}\epsilon$,
thus for every $\frac{1}{\p}+\epsilon<x<1-\epsilon$, 
\[
\frac{\psi_{\epsilon,\lambda}(x)-\psi_{\epsilon,\lambda}(1/\p)}{\psi_{\epsilon,\lambda}(1)-\psi_{\epsilon,\lambda}(1/\p)}=\varphi^{2}\left(x-1/\varphi\right)=\frac{x-1/\p}{1-1/\p}.
\]
\item $\psi'_{\epsilon,\lambda}$ is Lipschitz with Lipschitz constant of
the order $1/\epsilon$ when $\epsilon\to0$ and for every $x\in\TT$,
\begin{equation}
\lambda^{-2}\leq\psi_{\epsilon,\lambda}'(x)<\lambda^{2}.\label{eq: bound on derivative}
\end{equation}
\end{enumerate}
Given two sequences $\epsilon_{k}\geq0$ and $\lambda_{k}\geq1$,
denote by $\psi_{k}=\psi_{\lambda_{k},\epsilon_{k}}$. 

Define an order on $\Sigma_{A}^{+}$ in the following way. For $w,z\in\Sigma_{A}^{+}$,
let 
\[
j(w,z):=\inf\left\{ n\in\NN:\ w_{n}\neq z_{n}\right\} .
\]
Then $w\prec z$ if either $w_{j(w,z)}=1$ or $w_{j(w,z)}=3$ and
$z_{j(w,z)}=2$ (notice that in the latter case $j(w,z)=1$). This
order has the following property. If $[w]_{1}^{n}\neq[y]_{1}^{n}$
for some $n\in\NN$, then $C_{[w]_{1}^{n}}$ is to the left of $C_{[y]_{1}^{n}}$
if and only if $w\prec y$.

In addition for $n\in\NN$ we write $\bar{x}_{n},\underline{x}_{n}:\Sigma_{{\bf A}}\to\TT$
to be defined by 
\[
C_{[w]_{1}^{n}}:=\left[\underline{x}_{n}(w),\bar{x}_{n}(w)\right).
\]
For $n\in\mathbb{N},$ denote by $\Sigma_{{\bf A}}(n)$ the collection
of words $w=w_{1}w_{2}\cdots w_{n}$ with $[w]_{1}^{n}\subset\Sigma_{{\bf A}}$. 

We will define inductively a sequence $\left\{ h_{n}\right\} _{n=1}^{\infty}$
of diffeomorphisms of $\TT$. Since $\TT=\bigcup_{w\in\Sigma_{{\bf A}}(n)}C_{[w]_{1}^{n}}$
and each $h_{k},\ k<n$ is onto $\TT$, 
\[
\TT=\bigcup_{w\in\Sigma_{{\bf A}}(n)}H_{n-1}\left(C_{[w]_{1}^{n}}\right),
\]
 where $H_{n-1}:=h_{n-1}\circ h_{n-2}\circ\cdots\circ h_{1}$. 
\begin{itemize}
\item If $N_{t}<n<M_{t}$ for some $t\in\NN$, then $h_{n}$ is the identity. 
\item If $M_{t-1}<n\leq N_{t}$ for some $t\in\NN$, then $h_{n}$ is made
from $\#\Sigma_{{\bf A}}(n)$ scalings of $\psi_{t}$ or the identity.
Let $w(n,1),..,w(n,\#\Sigma_{{\bf A}}(n))$ be an enumeration of $\Sigma_{{\bf A}}(n)$
with respect to $\prec$. Set $h_{n}(0)=0$. Assume we have defined
$h_{n}$ on $\bigcup_{k=1}^{l-1}H_{n-1}\left(C_{w(n,k)}\right)$,
we will now define $h_{n}$ on $H_{n-1}\left(C_{w(n,l)}\right)$.
\\

\begin{itemize}
\item If $w(n,l)_{n}=1$, we define for $z\in H_{n-1}\left(C_{w(n,l)}\right)$,
\[
h_{n}(z):=H_{n-1}\left(\underline{x}_{n}(w)\right)+\mathfrak{l}(n,w)\psi_{t}\left(\frac{z-H_{n-1}\left(\underline{x}_{n}(w)\right)}{\mathfrak{l}(n,w)}\right),
\]
where $w=w(n,l)$ and 
\[
\mathfrak{l}(n,w):=m_{\TT}\left(H_{n-1}\left(w\right)\right)=H_{n-1}\left(\bar{x}_{n}(w)\right)-H_{n-1}\left(\underline{x}_{n}(w)\right).
\]
\item If $w(n,l)_{n}\neq1$ then for all $z\in H_{n-1}\left(C_{w(n,l)}\right)$,
\[
h_{n}(z)=z.
\]
\end{itemize}
\item Note that since $\psi_{t}(1)=1$ for all $t\in\mathbb{N}$, it follows
that $h_{n}\left(H_{n-1}\left(C_{w(n,l)}\right)\right)=H_{n-1}\left(C_{w(n,l)}\right)$
for all $n$ and $l$. Consequently, $h_{n}$ is continuous. The differentiability
of $h_{n}$ at points $\left\{ H_{n-1}\left(\underline{x}_{n}(w)\right):w\in\Sigma_{{\bf A}}(n)\right\} $
follows from $\psi_{t}'(0)=\psi_{t}'(1)=1$. 
\end{itemize}
We need to define $h_{n}$ for all $n\in\left\{ M_{t}\right\} _{t=1}^{\infty}$.
Here we apply a statistical correction procedure which we will now
proceed to describe. In what follows we assume that $\epsilon_{1}$
is small enough so that 
\[
m_{{\rm \TT}}\left(\psi_{1}\left(C_{[w]_{1}^{N_{1}}}\right)\right)=m_{{\rm \TT}}\left(\psi_{1}\left(C_{[w]_{1}^{2}}\right)\right)m_{{\rm \TT}}\left(\left.C_{[w]_{3}^{N_{1}}}\right|C_{[w]_{1}^{2}}\right).
\]
The first equality follows from property $3$ of $\psi_{t}$ provided
that $\epsilon_{1}$ is small enough so that for every $w\in\Sigma_{{\bf A}}^{+},$
the end points of $C_{[w]_{1}^{N_{1}}}$ are in $\left[\epsilon_{1},\p^{-1}-\epsilon_{1}\right]\cup\left[\p^{-1}+\epsilon_{1},1-\epsilon_{1}\right]\cup\{0,\p^{-1},1\}$.
The equality then follows from $\psi_{1}(1/\p)=\frac{\lambda_{1}\p}{1+\lambda_{1}\p}$.
This relation gives for example that
\[
m_{{\rm \TT}}\left(H_{N_{1}}\left(C_{[w]_{1}^{N_{1}}}\right)\right)=\mu^{+}\left(C_{[w]_{1}^{N_{1}}}\right),
\]
and we have good knowledge of where the point in $\frac{1}{\p}$ proportion
in $H_{n-1}\left(C_{[w]_{1}^{n}}\right)$ travels. However, since
$M_{t}$ is generally much larger then $N_{t}$ we loose this control
and the useful equality 
\begin{equation}
m_{{\rm \TT}}\left(H_{M_{t}}\left(C_{[w]_{1}^{N_{t+1}}}\right)\right)=m_{\TT}\left(H_{M_{t}}\left(C_{[w]_{1}^{M_{t}}}\right)\right)m_{{\rm \TT}}\left(\left.C_{\left[w\right]{}_{M_{t}+1}^{N_{t+1}}}\right|C_{[w]_{1}^{M_{t}}}\right)\label{eq: the important equality of H_M_t}
\end{equation}
need no longer hold true. The role of $h_{M_{t}}$ is to take care
that equality (\ref{eq: the important equality of H_M_t}) holds true. 

The function $H_{N_{t}}^{'}$ being a product of bounded Lipschitz
functions, is a bounded Lipschitz function. Therefore if $M_{t}$
is large enough with respect to $N_{t}$, then (here we use the fact
that $h_{n}=Id$ for $N_{t}<n<M_{t}$) $H'_{N_{t}}=H'_{M_{t}-1}$
is almost constant on $H_{M_{t}-1}\left(C_{[w]_{1}^{M_{t}}}\right)$.
That means that for every $0\neq x\in H_{M_{t}-1}\left(C_{[w]_{1}^{M_{t}}}\right)$
in the interior of $H_{M_{t}-1}\left(C_{[w]_{1}^{M_{t}}}\right)$,
\[
\left|\frac{m_{{\rm \TT}}\left(C_{[w]_{1}^{M_{t}}}\right)}{\intop_{C_{[w]_{1}^{M_{t}}}}H_{M_{t}-1}'(s)ds}\cdot H_{M_{t-1}}^{'}(x)-1\right|\ll1.
\]
By using a similar idea as in the construction of $\psi$ with the
$g_{\alpha}$ we define $h_{M_{t}}$ restricted to $H_{M_{t}-1}\left(C_{[w]_{1}^{M_{t}}}\right)$
so that equality (\ref{eq: the important equality of H_M_t}) holds.
This is done as follows: For $\alpha_{1},\alpha_{2}\in\RR$, let $G_{\alpha_{1},\alpha_{2}}:[0,1]\to\left[0,\alpha_{2}\right]$
be defined by $G_{\alpha_{1},\alpha_{2}}(0)=0$ and 
\begin{equation}
G'_{\alpha_{1},\alpha_{2}}(x):=\begin{cases}
\alpha_{1}+\frac{15\left(\alpha_{2}-\alpha_{1}\right)}{4}x, & 0\leq x\leq1/3\\
\frac{5\alpha_{2}-\alpha_{1}}{4}, & 1/3\leq x\leq2/3\\
\frac{5\alpha_{2}-\alpha_{1}}{4}+\frac{\alpha_{2}-\alpha_{1}}{4}(3x-1) & 2/3\leq x\leq1
\end{cases}.\label{eq: Definition of G_alpha_beta}
\end{equation}
This function is a $C^{1}$ function which satisfies $\ G'_{\alpha_{1},\alpha_{2}}(0)=\alpha_{1}$
and $G'_{\alpha_{1},\alpha_{2}}(1)=G_{\alpha_{1},\alpha_{2}}(1)=\alpha_{2}$. 

Define $\alpha:\NN\times\Sigma_{{\bf A}}\to(0,\infty)$ by
\[
\alpha(t,w):=\frac{1}{m_{{\rm \TT}}\left(C_{[w]_{1}^{M_{t}}}\right)}\intop_{C_{[w]_{1}^{M_{t}}}}H_{M_{t}-1}'(s)ds=\frac{m_{\TT}\left(H_{M_{t}-1}\left(C_{[w]_{1}^{M_{t}}}\right)\right)}{m_{\TT}\left(C_{[w]_{1}^{M_{t}}}\right)}.
\]
In addition for a finite word $w\in\Sigma_{{\bf A}}\left(M_{t}\right)$
we denote by $w^{-}$ the predecessor of $w$ with respect to $\prec$
restricted on $\Sigma_{{\bf A}}\left(M_{t}\right)$. We define $h_{M_{t}}'\circ H_{M_{t}-1}(x)$
on $C_{[w]_{1}^{M_{t}}}$ to be equal to $\frac{\alpha(t,w)}{H_{M_{t}-1}'(x)}$
off an $\epsilon_{t+1}m_{{\rm Leb}}\left(C_{[w]_{1}^{M_{t}}}\right)$
neighborhood of the left endpoint of the segment $C_{[w]_{1}^{M_{t}}}$,
$\frac{\alpha\left(t,w^{-}\right)}{H'_{M_{t-1}}(x)}$ on the left
endpoint (which is in the boundary of $C_{[w]_{1}^{M_{t}}}$) and
an interpolation in between by using $G_{\alpha_{1},\alpha_{2}}$
for an appropriately chosen $\alpha_{1},\alpha_{2}$. Here $\epsilon_{t+1}$
has to be small enough so that the end points of $\left\{ H_{M_{t}-1}\left(C_{[w]_{1}^{N_{t+1}}}\right):\ [w]_{1}^{N_{t+1}}\in\Sigma_{{\bf A}}\left(N_{t+1}\right)\right\} $
are not in an $\epsilon_{t+1}$ neighborhood of the left end point
of $H_{M_{t}-1}\left(C_{[w]_{1}^{M_{t}}}\right)$. Formally $\left.h_{M_{t}}\circ H_{M_{t}-1}\right|_{C_{[w]_{1}^{M_{t}}}}$
is defined by $h_{M_{t}}\circ H_{M_{t}-1}\left(\underline{x}_{M_{t}}(w)\right)=H_{M_{t}-1}\left(\underline{x}_{M_{t}}(w)\right)$
and 
\[
h'_{M_{t}}\circ H_{M_{t}-1}(x)=\frac{1}{H'_{M_{t}-1}(x)}\cdot\begin{cases}
G'_{\alpha\left(t,w^{-}\right),\alpha(t,w)}\left(\frac{x-\underline{x}_{M_{t}}(w)}{\epsilon_{t+\text{1 }}m_{\TT}\left(C_{[w]_{1}^{M_{t}}}\right)}\right), & \underline{x}_{M_{t}}(w)\leq x<\hat{x}_{M_{t}}(w)\\
\alpha(t,w), & \hat{x}_{M_{t}}(w)\leq x<\bar{x}_{M_{t}}(w)
\end{cases}.
\]
where $\hat{x}_{M_{t}}(w)=\underline{x}_{M_{t}}(w)+\epsilon_{t+1}m_{{\rm \TT}}\left(C_{[w]_{1}^{M_{t}}}\right)$.
It follows from the chain rule that for $x\in C_{[w]_{1}^{M_{t}}}$,
\[
H_{M_{t}}^{'}(x)=\begin{cases}
G'_{\alpha\left(t,w^{-}\right),\alpha(t,w)}\left(\frac{x-\underline{x}_{M_{t}}(w)}{\epsilon_{t+\text{1 }}m_{\TT}\left(C_{[w]_{1}^{M_{t}}}\right)}\right), & \underline{x}_{M_{t}}(w)\leq x<\hat{x}_{M_{t}}(w)\\
\alpha(t,w), & \hat{x}_{M_{t}}(w)\leq x<\bar{x}_{M_{t}}(w)
\end{cases}.
\]

\begin{figure}[h]
\includegraphics[scale=0.5]{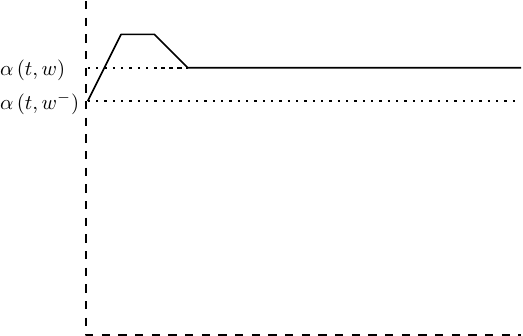}\caption{\label{fig:The-graph-of}{\small{}The graph of $H'_{M_{t}}$ restricted
to $C_{[w]_{1}^{M_{t}}}$ when $\alpha(t,w)>\alpha\left(t,w^{-}\right)$}}
\end{figure}

\begin{claim}
\label{claim: H_M_t rules}There exists $\delta_{t+1}$ such that
if $\epsilon_{t+1}<\delta_{t+1}$ then:

(a) for all $w\in\Sigma_{{\bf A}}$ and $M_{t}\leq n\leq N_{t+1}$
\[
H_{M_{t}}\left(C_{[w]_{1}^{M_{t}}}\right)=H_{M_{t}-1}\left(C_{[w]_{1}^{M_{t}}}\right)=H_{N_{t}}\left(C_{[w]_{1}^{M_{t}}}\right).
\]

(b) Equation (\ref{eq: the important equality of H_M_t}) holds. 
\end{claim}
\begin{proof}
Let $\delta_{t+1}$ be small enough so that the end points of 
\[
\left\{ H_{M_{t}-1}\left(C_{[w]_{1}^{N_{t+1}}}\right):\ [w]_{1}^{N_{t+1}}\in\Sigma_{{\bf A}}\left(N_{t+1}\right)\right\} 
\]
 are not in a $\delta_{t+1}$ neighborhood of $\left\{ H_{M_{t}-1}\left(\underline{x}_{M_{t}}\left(\tilde{w}\right)\right):\ \tilde{w}\in\Sigma_{{\bf A}}\left(M_{t}\right)\right\} $.
As a consequence $H_{M_{t}}'(x)=\alpha(t,w)$ for all $w\in\Sigma_{A}$
and $x\in\left\{ \underline{x}_{N_{t+1}}(w),\bar{x}_{N_{t+1}}(w)\right\} $.

Fix $w\in\Sigma_{{\bf A}}$ and $\epsilon_{t+1}\leq\delta_{t+1}$,
we first prove (a). Write for convenience $\underline{x}=\underline{x}_{M_{t}}(w)$,
$\bar{x}=\bar{x}_{M_{t}}(w)$ and $\hat{x}=\underline{x}+\epsilon_{t+1}\left(\bar{x}-\underline{x}\right)=\underline{x}_{M_{t}}(w)+\epsilon_{t+1}m_{{\rm \TT}}\left(C_{[w]_{1}^{M_{t}}}\right)$.
In this notation $\hat{x}-\underline{x}=\epsilon_{t+1}m_{\TT}\left(C_{[w]_{1}^{M_{t}}}\right)$
and \footnote{For an interval $I$ and a point $x$, $x+I=\left\{ x+y:\ y\in I\right\} .$}
we have 
\[
H_{M_{t}-1}\left(C_{[w]_{1}^{M_{t}}}\right)=\left(H_{M_{t}-1}\left(\underline{x}\right),H_{M_{t}-1}\left(\bar{x}\right)\right)=H_{M_{t}-1}\left(\underline{x}\right)+\left(0,H_{M_{t}-1}\left(\bar{x}\right)-H_{M_{t}-1}\left(\underline{x}\right)\right).
\]
In addition, since $h_{M_{t}}\circ H_{M_{t-1}}\left(\underline{x}\right)=H_{M_{t-1}}\left(\underline{x}\right)$,
then
\begin{align*}
H_{M_{t}}\left(C_{[w]_{1}^{M_{t}}}\right) & =H_{M_{t}}\left(\underline{x}\right)+\left(0,H_{M_{t}}\left(\bar{x}\right)-H_{M_{t}}\left(\underline{x}\right)\right)\\
 & =H_{M_{t}-1}\left(\underline{x}\right)+\left(0,H_{M_{t}}\left(\bar{x}\right)-H_{M_{t}}\left(\underline{x}\right)\right).
\end{align*}
This shows that (a) is equivalent to showing that 
\[
H_{M_{t}}\left(\bar{x}\right)-H_{M_{t}}\left(\underline{x}\right)=H_{M_{t}-1}\left(\bar{x}\right)-H_{M_{t}-1}\left(\underline{x}\right).
\]

Now 
\begin{align*}
H_{M_{t}}\left(\bar{x}\right)-H_{M_{t}}\left(\underline{x}\right) & =\int_{\underline{x}}^{\bar{x}}H_{M_{t}}'(s)ds\\
 & =\int_{\underline{x}}^{\hat{x}}G'_{\alpha\left(t,w^{-}\right),\alpha(t,w)}\left(\frac{s-\underline{x}}{\hat{x}-\underline{x}}\right)ds+\alpha(t,w)\left(\bar{x}-\hat{x}\right)\\
 & =\int_{0}^{\hat{x}-\underline{x}}G'_{\alpha\left(t,w^{-}\right),\alpha(t,w)}\left(\frac{s}{\hat{x}-\underline{x}}\right)ds+\alpha(t,w)\left(\bar{x}-\hat{x}\right).
\end{align*}

For all $\alpha_{1},\alpha_{2},\delta>0$, $\int_{0}^{\delta}G_{\alpha_{1},\alpha_{2}}'\left(\frac{x}{\delta}\right)dx=\delta\alpha_{2}$.
Whence 
\begin{align*}
H_{M_{t}}\left(\bar{x}\right)-H_{M_{t}}\left(\underline{x}\right) & =\alpha(t,w)\left(\hat{x}-\underline{x}\right)+\alpha(t,w)\left(\bar{x}-\hat{x}\right)\\
 & =\alpha(t,w)m_{\TT}\left(C_{[w]_{1}^{M_{t}}}\right)\\
 & =H_{M_{t}-1}\left(\bar{x}\right)-H_{M_{t}-1}\left(\underline{x}\right),
\end{align*}
we have finished the proof of part (a). 

To see part (b) notice if $\underline{x}\notin C_{[w]_{1}^{N_{t+1}}}$
then $H_{M_{t}}$ restricted to $C_{[w]_{1}^{N_{t+1}}}$ is linear
with slope $\alpha(t,w)$. This shows that 
\begin{align*}
m_{\TT}\left(H_{M_{t}}\left(C_{[w]_{1}^{M_{t}}}\right)\right) & =\alpha(t,w)m_{\TT}\left(C_{[w]_{1}^{N_{t+1}}}\right)\\
 & =m_{\TT}\left(H_{M_{t}}\left(C_{[w]_{1}^{M_{t}}}\right)\right)\frac{m_{\TT}\left(C_{[w]_{1}^{N_{t+1}}}\right)}{m_{\TT}\left(C_{[w]_{1}^{M_{t}}}\right)}\\
 & =m_{\TT}\left(H_{M_{t}}\left(C_{[w]_{1}^{M_{t}}}\right)\right)m_{\TT}\left(\left.C_{[w]_{M_{t}+1}^{N_{t+1}}}\right\lfloor C_{[w]_{1}^{M_{t}}}\right),
\end{align*}
as required. If $\underline{x}\in C_{[w]_{1}^{N_{t+1}}}$ then $C_{[w]_{1}^{N_{t+1}}}=\left[\underline{x},\overline{x}_{N_{t+1}}(w)\right)$
and thus as in the proof of part (a)
\begin{align*}
m_{\TT}\left(H_{M_{t}}\left(C_{[w]_{1}^{N_{t+1}}}\right)\right) & =H_{M_{t}}\left(\bar{x}_{N_{t+1}}(w)\right)-H_{M_{t}}\left(\underline{x}\right)\\
 & =\int_{0}^{\hat{x}-\underline{x}}G'_{\alpha\left(t,w^{-}\right),\alpha(t,w)}\left(\frac{s}{\hat{x}-\underline{x}}\right)ds+\alpha(t,w)\left(\bar{x}_{N_{t+1}}(w)-\hat{x}\right)\\
 & =\alpha(t,w)\left(\bar{x}_{N_{t+1}(w)}-\underline{x}\right)=\alpha(t,w)m_{\TT}\left(C_{[w]_{1}^{N_{t+1}}}\right).
\end{align*}
Continuing as in the case $\underline{x}\notin C_{[w]_{1}^{N_{t+1}}}$
one arrives at the conclusion. 
\end{proof}
\begin{rem}
\label{RK: An important feature of the perturbation}An important
feature of this construction that will be used in the extension to
two dimensions is that for any $1\leq l\leq\#\Sigma_{{\bf A}}(n)$,
\begin{equation}
h_{n}\left(H_{n-1}\left(C_{w(n,l)}\right)\right)=H_{n-1}\left(C_{w(n,l)}\right).\label{eq: h_n keeps the former Markov partition}
\end{equation}
This in turn implies that for every $n\in\mathbb{N}$, $\mathcal{H}_{n}(x,y):=(H_{n}(x),y)$
is a diffeomorphism of $\MM{}_{\sim}$ and the Markov partition $\left\{ R_{1},R_{2},R_{3}\right\} $
for $\tilde{f}$ defined by 
\[
R_{i}:=\begin{cases}
J_{i}\times\left[-\frac{\p}{\p+2},\frac{\p^{2}}{\p+2}\right], & i\in\{1,3\}\\
J_{2}\times\left[-\frac{\p}{\p+2},\frac{1}{\p+2}\right], & i=2
\end{cases}.
\]
is preserved by $\mathcal{H}_{n}$.
\end{rem}
\begin{thm}
\label{thm: the one dimensional perturbation}There exists a choice
of $\lambda_{k}\downarrow1$, $\left\{ n_{k},m_{k},N_{k},M_{k}\right\} _{k\in\NN}\subset\NN$
and $\ep=\left\{ \epsilon_{k}\right\} _{k\in\NN}$ so that:

(i). The Markov measure from the construction of Section \ref{sec:Type--Markov}
is a type ${\rm III}_{1}$ measure for the shift on $\Sigma_{{\bf A}}$. 

(ii). The function $\mathfrak{h}_{\ep}$ is a circle homeomorphism
and $m_{{\rm \TT}}\circ\mathfrak{h}_{\ep}\sim\mu^{+}$ where $\mu^{+}=\Theta_{*}\M{P_{k},\pi_{k}}_{k=1}^{\infty}$.

(iii) The function $\mathfrak{g}=\mathfrak{h_{\ep}}\circ S\circ\mathfrak{h_{\ep}^{-1}}$
is $C^{1}$, and for every $x\in\TT$, 
\[
1.6\leq\mathfrak{g'}(x)\leq1.7.
\]
\end{thm}
The proof of this Theorem is by showing that we can realise smoothly
a the inductive construction of Section \ref{sec:Type--Markov} (with
three extra conditions) and include a new sequence $\left\{ \epsilon_{k}\right\} $
in it so that the following properties hold:
\begin{enumerate}
\item $\h_{\epsilon}:=\lim_{n\to\infty}H_{n}$ is a homeomorphism of $\TT$. 
\item $\left\{ \g_{n}\right\} $ is a convergent subsequence in the $C^{1}$
topology, here $\mathfrak{g}_{n}:=H_{n}\circ S\circ H_{n}^{-1}$. 
\item The limit function $\mathfrak{g}=\lim_{n\to\infty}\g_{n}=\h_{\epsilon}\circ S\circ\h_{\epsilon}^{-1}$
satisfies $1.6\leq\mathfrak{g}'(x)\leq1.7$. 
\item $m_{\TT}\circ\h_{\epsilon}\sim\mu^{+}$. 
\end{enumerate}

\subsubsection{The inductive choice of $\left\{ \epsilon_{l}\right\} _{l=1}^{\infty}$. }

Before we continue we would like to set up some notation which will
be used. 
\begin{itemize}
\item Given $\ep=\left\{ \epsilon_{k}\right\} _{k=1}^{t}$ and $n\leq N_{t}$
we denote by $h_{\ep,n}$ the function in the construction with the
sequence $\ep$ at level $n$. 
\item For $j\leq N_{t}$, $H_{\ep,j}:=h_{\ep,j}\circ h_{\ep,j-1}\circ\cdots\circ h_{\ep,1}$.
The function $H_{\ep,j}$ only depends on $\left\{ \epsilon_{s}\right\} _{s=1}^{t}$
with $j\leq N_{t}$. 
\item $H_{0,j}$ will denote the function with $\underbar{\ensuremath{\epsilon}}=\underbar{0}$.
\end{itemize}
\begin{lem}
\label{lem: bound on h_M_t derivative}Assume that $\left\{ \epsilon_{s}\right\} _{s=1}^{t}$
were chosen so that for all $s<t$ and $x\in\TT$, $h_{\ep,M_{s}}^{'}(x)=e^{\pm2^{-N_{s}}}$.
If $M_{t}$ is sufficiently large with respect to $N_{t}$ and $\epsilon_{t+1}$
is small enough then the following two properties hold:

(i) For all $x\in\TT$, 
\[
h_{\ep,M_{t}}'(x)=e^{\pm2^{-N_{t}}}.
\]

(ii) Let \textup{$w\in\Sigma_{{\bf A}}$ and $M_{t}<n\leq N_{t+1}$.
Denote by $\xi(n,w)=\underline{x}_{n}(w)+\frac{1}{\p}\left(\overline{x}_{n}(w)-\underline{x}_{n}(w)\right)$
the point in $\frac{1}{\p}$ proportion in $C_{[w]_{1}^{n}}$. Then
\[
\frac{H_{\ep,n-1}\left(\xi_{n}\left(w\right)\right)-H_{\ep,n-1}\left(\underline{x}_{n}(w)\right)}{H_{\ep,n-1}\left(\bar{x}_{n}\left(w\right)\right)-H_{\ep,n-1}\left(\underline{x}_{n}(w)\right)}=\frac{1}{\p}.
\]
That is the (reference) point in $\frac{1}{\p}$ proportion in $C_{[w]_{1}^{n}}$
travels under $H_{\ep,n-1}$ to the reference point in $H_{\ep,n-1}\left(C_{[w]_{1}^{n}}\right)$. }
\end{lem}
\begin{proof}
In the course of the proof we write for $n\leq N_{t+1}$, $h_{n}=h_{\ep,n}$
and $H_{\ep,n}=H_{n}$. Let $\delta>0$. Since $h_{n}$ is the identity
for $N_{t}<n<M_{t}$, then $H_{M_{t}-1}=H_{N_{t}}$. The function
$H_{N_{t}}^{'}$ is a product of $N_{t}$ bounded Lipschitz functions
and $\inf_{t\in[0,1]}H'_{N_{t}}(x)>0$. Therefore there exists $K(t)>1$,
which depends only on $\left\{ \lambda_{s},N_{s},M_{s},\epsilon_{s}\right\} _{s=1}^{t-1}$
and $\left\{ N_{t},\lambda_{t},\epsilon_{t}\right\} $, such that
for every $x,y\in\TT$, 
\[
\left|H_{M_{t}-1}'(x)-H_{M_{t}-1}'(y)\right|=\left|H_{N_{t}}'(x)-H_{N_{t}}'(y)\right|\leq K(t)|x-y|
\]
and for every $x\in\TT$, 
\begin{equation}
K(t)^{-1}\leq\left|H'_{M_{t}-1}(x)\right|<K(t).\label{eq:bound on H__M_t '}
\end{equation}
By uniform expansion of $S$, if $M_{t}$ is sufficiently large then
\[
\sup_{w\in\Sigma_{{\bf A}}}m_{{\rm \TT}}\left(H_{M_{t}-1}\left(C_{[w]_{1}^{M_{t}}}\right)\right)\leq\p^{-\left(M_{t}-1\right)}K(t)<\frac{\delta}{K(t)^{2}}.
\]
This implies that for every $w\in\Sigma_{{\bf A}}$ and $x,y\in H_{M_{t}-1}\left(C_{[w]_{1}^{M_{t}}}\right)$,
\[
\left|H_{M_{t}-1}'(x)-H_{M_{t}-1}'(y)\right|\leq K(t)|x-y|\leq K(t)m_{{\rm \TT}}\left(H_{M_{t}-1}\left(C_{[w]_{1}^{M_{t}}}\right)\right)<\delta/K(t).
\]
Averaging this inequality over all $y\in H_{M_{t}-1}\left(C_{[w]_{1}^{M_{t}}}\right)$,
for every $x\in H_{M_{t}-1}\left(C_{[w]_{1}^{M_{t}}}\right)$, 
\begin{align*}
\left|H_{M_{t}-1}'(x)-\alpha(t,w)\right| & =\left|H_{M_{t}-1}'(x)-\frac{1}{m_{\TT}\left(C_{[w]_{1}^{M_{t}}}\right)}\int_{C_{[w]_{1}^{M_{t}}}}H_{M_{t}-1}'(y)dy\right|\\
 & \leq K(t)^{-1}\delta.
\end{align*}
 If follows from this and the lower bound in (\ref{eq:bound on H__M_t '})
that for every $x\in C_{[w]_{1}^{M_{t}}}$, 
\[
\left|\frac{\alpha(t,w)}{H'_{M_{t}-1}(x)}-1\right|<\delta,
\]

A consequence of the latter inequality which is proved by fixing ${\bf x}(w)=\underline{x}_{M_{t}}(w)=\overline{x}_{M_{t}}\left(w^{-}\right)$
once on $w$ and once on $w^{-},$ is that 
\[
\forall w\in\Sigma_{{\bf A}}\left(M_{t}\right),\ \left|\frac{\alpha(t,w)}{H'_{M_{t}-1}\left({\bf x}(w)\right)}-\frac{\alpha\left(t,w^{-}\right)}{H'_{M_{t}-1}\left({\bf x}(w)\right)}\right|<2\delta.
\]
Part $(i)$ follows by choosing an appropriate $\delta$ and the definition
of $h_{M_{t}}$. 

(ii) By the definition of $h_{M_{t}}$, if $\epsilon_{t+1}$ is small
enough then equation (\ref{eq: the important equality of H_M_t})
holds. Using property 3 of $\psi_{\epsilon_{t},\lambda_{t}}$, a proof
by induction shows that that for all $M_{t}<n<J\leq N_{t+1}$, 
\begin{equation}
m_{\TT}\left(H_{n}\left(C_{[w]_{1}^{J}}\right)\right)=m_{\TT}\left(H_{n}\left(C_{[w]_{1}^{n+1}}\right)\right)m_{{\rm \TT}}\left(\left.C_{\left[w\right]{}_{n+1}^{J}}\right|C_{[w]_{1}^{n+1}}\right).\label{eq: distribution of an interval inside an interval}
\end{equation}
The conclusion follows since if $w_{n+1}\in\{1,2\}$ then $C_{[w1]_{1}^{n+2}}=\left[\underline{x}_{n+1}(w),\xi_{n+1}(w)\right)$
\begin{eqnarray*}
\frac{H_{n}\left(\xi_{n+1}(w)\right)-H_{n}\left(\underline{x}_{n+1}(w)\right)}{H_{n}\left(\overline{x}_{n+1}(w)\right)-H_{n}\left(\underline{x}_{n+1}(w)\right)} & = & \frac{m_{{\rm \TT}}\left(H_{n}\left(C_{[w1]_{1}^{n+2}}\right)\right)}{m_{\TT}\left(H_{n}\left(C_{[w]_{1}^{n+1}}\right)\right)}\\
 & = & m_{\TT}\left(\left.C_{\left[w1\right]_{n+2}^{n+2}}\right|C_{[w]_{1}^{n+1}}\right)\\
 & = & m_{\TT}\left(\left.w_{2}=1\right|w_{1}=1\right)=\frac{1}{\p}.
\end{eqnarray*}
If $w_{n+1}=3$ then $C_{[w21]_{1}^{n+3}}=\left[\underline{x}_{n+1}(w),\xi_{n+1}(w)\right)$
and then 
\[
\frac{H_{n}\left(\xi_{n+1}(w)\right)-H_{n}\left(\underline{x}_{n+1}(w)\right)}{H_{n}\left(\overline{x}_{n+1}(w)\right)-H_{n}\left(\underline{x}_{n+1}(w)\right)}=m_{\TT}\left(\left.w_{2}=2,w_{3}=1\right|w_{1}=3\right)=\frac{1}{\varphi}.
\]
\end{proof}
By part (i) of the previous Lemma we can choose sequences $\left\{ \lambda_{t},n_{t},N_{t},M_{t},\epsilon_{t}\right\} _{t\in\mathbb{N}}$
so that $\sup_{x\in\TT}h'_{\ep,M_{t}}(x)\leq e^{2^{-N_{t}}}$ for
all $t\in\mathbb{N}$. 
\begin{prop}
\label{Cor: A consequence on the bound of eta and chi} Assume $\varphi/\lambda_{1}^{2}>1.6$,
assume that for all $t\in\mathbb{N}$, $\sup_{x\in\TT}h'_{\ep,M_{t}}(x)\leq e^{2^{-N_{t}}}$,
then
\[
\sup_{n\in\NN}\left|h_{\ep,n}(x)-x\right|\leq e(1.6)^{-n}
\]
and consequently $\lim_{n\to\infty}H_{\ep,n}(x)=\mathfrak{h}_{\ep}(x)$
is a homeomorphism of $\TT$. 
\end{prop}
\begin{proof}
If for some $\tau\leq T+1$, $M_{\tau}<k<M_{\tau+1}$ then, 
\begin{equation}
\sup_{x\in\TT}h_{\ep,k}^{'}(x)=\sup_{s\in\TT}\left|\psi_{k}(s)\right|\overset{\eqref{eq: bound on derivative}}{\leq}\lambda_{\tau}^{2}\leq\lambda_{1}^{2}.\label{eq: bound on derivative II}
\end{equation}
Therefore for every $n\leq M_{T+1}$, 
\begin{align*}
m_{\text{\ensuremath{\TT}}}\left(H_{\epsilon,n-1}\left(C_{[w]_{1}^{n}}\right)\right) & \leq\left(\prod_{k=1}^{T}\sup_{x\in\TT}\left|h'_{\ep,M_{k}}(x)\right|\right)\lambda_{1}^{2n}m_{\text{\ensuremath{\TT}}}\left(C_{[x]_{1}^{n}}\right)\\
 & \leq\exp\left(\sum_{k=1}^{T}2^{-N_{k}}\right)\left(\frac{\lambda_{1}}{\varphi}\right)^{n}\leq e(1.6)^{-n}.
\end{align*}
The invariance of $H_{\underline{\epsilon},n-1}\left(C_{[w]_{1}^{n}}\right)$
under $h_{\underline{\epsilon},n}$ implies that
\[
\sup_{x\in\TT}\left|h_{\underline{\epsilon},n}(x)-x\right|\leq\sup_{w\in\Sigma_{{\bf A}}}m_{\text{\ensuremath{\TT}}}\left(H_{\underline{\epsilon},n-1}\left(C_{[w]_{1}^{n}}\right)\right)\leq e(1.6)^{-n}.
\]

Consequently for every $n<m$, 
\begin{eqnarray*}
\left|H_{\ep,m}(z)-H_{\ep,n}(z)\right| & \leq & \sum_{k=n}^{m}\left|H_{\ep,k+1}(z)-H_{\ep,k}(z)\right|\\
 & = & \sum_{k=n}^{m}\left|h_{\ep,k+1}\left(H_{\ep,k}(z)\right)-H_{\ep,k}(z)\right|\\
 & \leq & e\sum_{k=n}^{m}\sup_{z\in\TT}\left|h_{\ep,k+1}(z)-z\right|\leq e\sum_{k=n}^{m}(1.6)^{-k}.
\end{eqnarray*}
This shows that $\left\{ H_{\ep,m}\right\} _{m=1}^{\infty}$ is a
Cauchy sequence in $C(\TT)$, it's limit being a continuous and strictly
increasing function is a homeomorphism of $\TT$. 
\end{proof}
\begin{lem}
\label{lem: Differentiablility of g}Assume $\left\{ \epsilon_{k}\right\} _{k=1}^{t}$
are already chosen so that for all $s<t$ and $x\in\TT$, $h_{\ep,M_{s}}^{'}(x)=e^{\pm2^{-N_{s}}}$.
If $M_{t}$ is large enough with respect to $N_{t}$ then there exists
$\tilde{\delta}_{t+1}>0$ so that for all $\epsilon_{t+1}<\tilde{\delta}_{t+1}$
\begin{equation}
\g_{N_{t+1}}'(x)=\lambda_{t+1}^{\pm M_{t}}e^{\pm2^{-N_{t}+2}}\g_{N_{t}}'(x)\label{eq: relation between the derivatives in beta shift}
\end{equation}

Here $\g_{N_{t}}=H_{\ep,N_{t}}\circ S\circ H_{\ep,N_{t}}^{-1}$. 
\end{lem}
\begin{proof}
Assume first that $\epsilon_{t+1}=0$ and since we are not going to
vary $\ep$ we write $H_{n}$ and $h_{n}$ to denote $H_{\ep,n}$
and $h_{\ep,n}$. Since $\epsilon_{t+1}=0$, by Lemma \ref{lem: bound on h_M_t derivative}
if $M_{t}$ is large enough then $h'_{M_{t}}(x)=e^{\pm2^{-N_{t}}}$
for all $x\in\TT$. We assume that $M_{t}$ is large enough for this
to hold. 

Let $z\in\TT$, there exists a unique $y=y(z)$ such that $z=H_{N_{t+1}}(y)$.
By the chain and differentiation of inverse functions, if $\mathfrak{g}_{N_{t+1}}$
is differentiable at $z$, 
\[
\g_{N_{t+1}}'(z)=\p\frac{H_{N_{t+1}}'\left(Sy\right)}{H'_{N_{t+1}}(y)}.
\]
Therefore since $h_{k}=id.$ for all $N_{t}<k<M_{t}$, $H_{N_{t}}=H_{M_{t}-1}$
and 
\begin{eqnarray*}
\frac{\g_{N_{t+1}}'(z)}{\g_{N_{t}}'(z)} & = & \frac{H_{N_{t+1}}'\left(Sy\right)}{H'_{N_{t}}(Sy)}\cdot\frac{H_{N_{t}}'\left(y\right)}{H'_{N_{t+1}}(y)}\\
 & = & \left(\prod_{k=M_{t}}^{N_{t+1}}h'_{k}\left(H_{k-1}\left(Sy\right)\right)\right)\left(\prod_{k=M_{t}}^{N_{t+1}}h'_{k}\left(H_{k-1}\left(y\right)\right)\right)^{-1}
\end{eqnarray*}

Fix $j\in\left[M_{t},N_{t+1}\right)$. Notice that $H_{j-1}(y)\in H_{j-1}\left(C_{[w]_{1}^{j}}\right)$
if and only if 
\[
H_{j-2}\left(Sy\right)\in H_{j-2}\left(C_{\left[w_{2}\cdots w_{j}\right]}\right),
\]
and that by Lemma \ref{lem: bound on h_M_t derivative}.(ii), $H_{j-1}(y)$
and $H_{j-2}\left(Sy\right)$ are to the right of $\xi\left(C_{\left[w\right]_{1}^{j}}\right)$
and $\xi\left(C_{\left[w_{2}\cdots w_{j}\right]}\right)$ respectively
if and only if $y$ is to the right of the reference point in $C_{[w]_{1}^{j}}$.
Thus under the assumption that $\epsilon_{t+1}=0$ for all $j\in\left(M_{t}+1,N_{t+1}\right]$,
\[
\frac{h_{j-1}'\left(H_{j-2}\left(Sy\right)\right)}{h_{j}'\left(H_{j-1}(y)\right)}=1.
\]
The last equality together with Lemma \ref{lem: bound on h_M_t derivative}(i)
implies that if $M_{t}$ is large enough then, 

\begin{eqnarray*}
\frac{\g_{N_{t+1}}'(z)}{\g_{N_{t}}'(z)} & = & \frac{h_{M_{t}}'\left(H_{M_{t}-1}\left(Sy\right)\right)}{h_{M_{t}}'\left(H_{M_{t}-1}(y)\right)}\frac{h_{N_{t+1}}'\left(H_{N_{t+1}-1}(Sy)\right)}{h_{M_{t}+1}'\left(H_{M_{t}}(y)\right)}\underset{=1}{\underbrace{\prod_{j=M_{t}+2}^{N_{t+1}}\frac{h_{j-1}'\left(H_{j-2}\left(Sy\right)\right)}{h_{j}'\left(H_{j-1}(y)\right)}}}\\
 & = & \left(\lambda_{t+1}^{8}e^{2^{-N_{t}+1}}\right)^{\pm1}
\end{eqnarray*}
 The last inequality makes use of the fact that for $l\in\left\{ M_{t}+1,N_{t+1}\right\} $
and $z\in\TT$, $\left|h'_{l}(z)\right|=\lambda_{t+1}^{\pm2}$. 

In \cite{Bruin Hawkins} they argue that the estimate on the derivative
is continuous (uniformly) with respect to $\epsilon_{t+1}$ since
$\psi_{\epsilon,\lambda_{t+1}}^{'}$ converges pointwise to $\psi_{0,\lambda_{t+1}}^{'}$
when $\epsilon\to0$. However this convergence is not uniform (and
it can't be as it converges to a step function) and therefore their
argument is not sufficient for convergence in the $C^{1}$ norm.

We proceed as follows. For $n\in\left(M_{t},N_{t+1}\right]$ and $w\in\Sigma_{{\bf A}}$
with $w_{n}=1$ denote by ${\rm BS}(n,w)$, the \textit{Bad Set} at
stage $n$ for $w$, to be the following set 
\[
\left\{ y\in C_{[w]_{1}^{n}}:\ \forall\delta>0,\ \exists z\in(y-\delta,y+\delta),\ h_{n}'\circ H_{n-1}(z)\notin\left\{ \frac{\lambda_{t+1}\p^{2}}{1+\lambda_{t+1}\p},\frac{\p^{2}}{1+\lambda_{t+1}\p}\right\} \right\} 
\]
  This set, which is a union of four small intervals, is the set of
all $y\in C_{[w]_{1}^{n}}$ where the derivative of $h_{n}^{'}\circ H_{n-1}$
is not constant on a neighborhood of $y$. 

First we demand that $\delta_{t+1}$ is small enough so that the conclusion
of Lemma \ref{lem: bound on h_M_t derivative} and equation (\ref{eq: distribution of an interval inside an interval})
hold for all $\epsilon_{t+1}<\delta_{t+1}$. 

Secondly we demand that $\delta_{t+1}$ is small enough so that for
$M_{t}<n<m\leq N_{t+1}$, if ${\rm BS}(m,w)\cap{\rm BS}\left(n,w\right)\neq\emptyset$
then one of the end points of $H_{m-1}\left(C_{[w]_{1}^{m}}\right)$
is either an end point of $H_{n-1}\left(C_{[w]_{1}^{n}}\right)$ or
the point in $\frac{1}{\p}$ proportion in $H_{n-1}\left(C_{[w]_{1}^{n}}\right)$. 

\begin{figure}[h]
\includegraphics[scale=0.55]{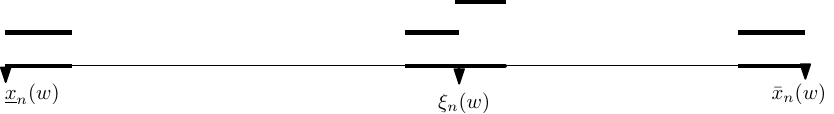}\label{fig: 1222}

\textit{\footnotesize{}The small intervals demonstrate the possibilities
of locations of $BS(m,w)$. }{\footnotesize \par}
\end{figure}

To understand why we choose these points, notice that in those marked
endpoints 
\[
h_{n}'(x)=1.
\]
This can be done if for example\footnote{Here notice that $\sup_{n\leq k\leq N_{t+1},x\in\TT}\frac{1}{h_{\ep,k}'(x)}\leq\lambda_{t+1}^{-2}$
irrespectible of the choice of $\ep$.} 
\[
\frac{m_{{\rm \TT}}\left(H_{\ep,m-1}\left(C_{[w]_{1}^{m}}\right)\right)}{m_{\TT}\left(H_{\ep,n-1}\left(C_{[w]_{1}^{n}}\right)\right)}\geq\left(\frac{1}{\p^{2}\lambda_{t+1}^{2}}\right)^{n-m}\geq\frac{1}{5^{-n_{t+1}}}\gg\delta_{t+1},
\]
Indeed, if $\epsilon_{t+1}\leq\delta_{t+1}$, then $BS(n,w)$ is union
4 subintervals of $H_{\ep,n-1}\left(C_{[w]_{1}^{n}}\right)$ considerably
smaller length then $H_{\ep,m-1}\left(C_{[w]_{1}^{m}}\right)$ and
thus their bad sets can only intersect in a unique interval if either
$\bar{x}_{m}(w)\in\left\{ \bar{x}_{n}(w),\xi_{n}(w)\right\} $ or
$\underline{x}_{m}(w)\in\left\{ \underline{x}_{m}(w),\xi_{n}(w)\right\} $. 

In fact with such a choice of $\delta_{t+1}$ one has that for all
$w\in\Sigma_{{\bf A}}$ and $M_{t}<n<m\leq N_{t+1}$, ${\rm BS}(n,w)\cap{\rm BS}(m,w)$
is always one interval for which one of its end points satisfies 
\begin{equation}
h_{n}'(H_{n-1}(x))=1.\label{eq: Bad set anal. good point}
\end{equation}
In addition, 
\begin{eqnarray}
\frac{m_{{\rm \TT}}\left({\rm BS}(n,w)\cap{\rm BS}(m,w)\right)}{m_{{\rm \TT}}\left({\rm BS}(n,w)\right)} & \leq & \sup_{x\in\TT}\left(h_{n}^{-1}\circ\cdots h_{m}^{-1}\right)'(x)\frac{m_{\TT}\left(C_{[w]_{1}^{m}}\right)}{m_{\TT}\left(C_{[w]_{1}^{n}}\right)}\label{eq: Size of Bsets inside eachother}\\
 & \leq & \left(\frac{\lambda_{1}^{2}}{\p}\right)^{n-m}\leq(1.6)^{n-m}.\nonumber 
\end{eqnarray}
By the definition of $h_{\ep,n}$, $h_{\ep,n}'\circ H_{n-1}$ is a
Lipschitz function with a Lipschitz constant of order $Const./m_{{\rm \TT}}\left({\rm BS}(n,w)\right)$. 

It follows from (\ref{eq: Bad set anal. good point}) and (\ref{eq: Size of Bsets inside eachother})
that there exists a constant $B>0$ such that for all $y\in{\rm BS}(n,w)\cap{\rm BS}(m,w)$,
\[
h_{n}'(H_{n}(y))=e^{\pm B(1.6)^{n-m}}.
\]

The final argument is as follows: Given $x\in\TT$ there is a unique
$y\in\TT$ such that $x=H_{N_{t+1}}(y)$. Let $w$ be such that $y\in C_{[w]_{1}^{N_{t+1}}}$.
If $y\notin\cup_{n=M_{t}+1}^{N_{t+1}}{\rm BS}(n,w)$ then a similar
analysis as in the case $\epsilon_{t+1}=0$ yields the conclusion.
Otherwise there exists a maximal $M_{t}<{\bf J}=J(y)\leq N_{t+1}$
such that $y\in{\rm BS}({\bf J},w)$. A similar argument as in the
case $\epsilon_{t+1}=0$ yields 
\begin{equation}
g_{N_{t+1}}'(z)=\prod_{k={\bf J}}^{N_{t+1}}\frac{h_{k}\left(H_{k-1}\left(Sy\right)\right)}{h_{k}\left(H_{k-1}(y)\right)}=\lambda_{t+1}^{\pm4}g_{{\bf J}}'\circ H_{{\bf J}-1}(y).\label{eq: derivative of g till BS is hit}
\end{equation}
For $M_{t}+2\leq n<{\bf J}-M_{t}/4$, either $y\notin{\rm BS}(k,w)$
for all $k\leq n$ and then we proceed as in the case $\epsilon_{t+1}=0$
or $y\in{\rm BS}(n,w)\cap{\rm BS}\left({\bf J},w\right)$ and then,
\[
h_{n}'\circ H_{n-1}(y)=e^{\pm B(1.6)^{n-{\bf J}}}.
\]
In addition, $S\left({\rm BS}(n,w)\cap{\rm BS}\left({\bf J},w\right)\right)$
is an interval of size $\p m_{{\rm \TT}}\left({\rm BS}(n,w)\cap{\rm BS}({\bf J},w)\right)$
with one point $x$ for which\footnote{$x$ is either an end point or the point in $\frac{1}{\p}$ proportion
in $C_{[w_{2},...,w_{n+1}]}$} $h_{n}'\circ H_{n-1}(x)=1$. Therefore as before, 
\[
h_{n}'\circ H_{n-1}(Sy)=e^{\pm\p B(1.6)^{m-{\bf J}}.}
\]
Thus, using that for all $\ensuremath{M_{t}\leq k\leq N_{t+1},}\frac{h_{k}'\left(H_{k-1}(Sy)\right)}{h_{j}'\left(H_{k-1}(y)\right)}\leq\lambda_{t+1}^{4}$
\begin{eqnarray*}
g_{{\bf J}}'\circ H_{J-1}(y) & \leq & \g_{N_{t}}'\left(H_{M_{t}-1}(y)\right)\frac{h_{M_{t}}'\left(H_{M_{t}-1}\left(Sy\right)\right)}{h_{M_{t}}'\left(H_{M_{t}-1}(y)\right)}\prod_{n=M_{t}+1}^{{\bf J}-M_{t}/4}e^{3B(1.6)^{n-{\bf J}}}\prod_{k=J-M_{t}/4}^{J}\frac{h_{k}'\left(H_{k-1}(Sy)\right)}{h_{j}'\left(H_{k-1}(y)\right)}\\
 & \leq & \left(\g_{N_{t}}'\left(z\right)e^{2^{-N_{t}+2}}\right)e^{C(1.6)^{-M_{t}/4}}\lambda_{t+1}^{M_{t}}.
\end{eqnarray*}
The upper bound follows from the last equation together with (\ref{eq: derivative of g till BS is hit})
since $M_{t}/4\gg N_{t}$. The lower bound is similar. 
\end{proof}
A consequence Lemma \ref{lem: Differentiablility of g} is that we
can choose $\ep=\left\{ \epsilon_{k}\right\} _{k=1}^{\infty}$ so
that $\g_{N_{t}}$ and $D_{\g_{N_{t}}}$ converge uniformly to a map
$\mathfrak{g}$ with 
\begin{equation}
D_{\g}(x)=\p\cdot\left(\prod_{t\in\NN}\lambda_{t}^{\pm M_{t-1}}\right)\cdot e^{\sum_{t=1}^{\infty}2^{-N_{t}+4},}.\label{eq: derivative of g with epsilon}
\end{equation}
By taking care that for each $t\in\NN$, $\lambda_{t}^{M_{t-1}}$
is small enough and the $N_{t}$ are large enough, 
\[
1.6\leq\p\cdot\left(\prod_{t\in\NN}\lambda_{t}^{\pm M_{t-1}}\right)\cdot\exp\left(\sum_{t=1}^{\infty}2^{-N_{t}+4}\right)\leq1.7,
\]
thus the limiting transformation $\mathfrak{g}$ is uniformly expanding.
What remains to be shown before we can explain the modified inductive
construction of $\left\{ \lambda_{k},M_{k},N_{k},\epsilon_{k}\right\} _{k=1}^{\infty}$
is that we can choose $\ep$ so that $\mbox{\ensuremath{m_{{\rm \TT}}\circ\mathfrak{h}_{\ep}\sim\mu^{+}}.}$ 
\begin{lem}
\label{Lem:  the pertubation of the beta shift}Assume that $\mu^{+}$
is a push forward via $\Theta$ of the Markovian type ${\rm III}_{1}$
measure for the shift defined by $\left\{ \lambda_{k},m_{k},n_{k},M_{k},N_{k}\right\} _{k=1}^{\infty}$
. Then there exists a sequence $\underline{\varepsilon}=\left\{ \varepsilon_{k}\right\} _{k=1}^{\infty}$
such that for every $\ep=\left\{ \epsilon_{k}\right\} _{k=1}^{\infty}$
which satisfies $\forall k\in\NN,\ \epsilon_{k}\leq\varepsilon_{k}$,
the function $\mathfrak{h}_{\ep}$ defined previously satisfies 
\[
m_{\text{\ensuremath{\TT}}}\circ\mathfrak{h}{}_{\ep}\sim\mu^{+}.
\]
\end{lem}
\begin{proof}
The proof of the Lemma will be done by applying the theory of local
absolute continuity of Shiryaev with $\cF_{t}$ the sigma algebra
generated by $\left\{ C_{[w]_{1}^{N_{t}}}:\ w\in\Sigma_{{\bf A}}\right\} $.
For $\ep=\left\{ \epsilon_{k}\right\} _{k=1}^{\infty}$, we will use
the notation $\r_{\ep,n}(x):=h_{\ep,n}'\left(H_{\ep,n-1}(x)\right)$. 

Given $\ep=\left\{ \epsilon_{k}\right\} _{k=1}^{t}$, 
\[
\left(m_{\text{\ensuremath{\TT}}}\circ\h_{\ep}\right)_{t}:=\left.m_{\text{\ensuremath{\TT}}}\circ\h_{\ep}\right|_{\cF_{t}}=m_{\TT}\circ H_{\ep,N_{t}},
\]
and 
\[
\left(\mu^{+}\right)_{t}:=\left.\mu\right|_{\cF_{t}}=m_{\TT}\circ H_{0,N_{t}}.
\]

A calculation shows that 
\[
z_{t}(x):=\frac{d\left(m_{\text{\ensuremath{\TT}}}\circ\h_{\ep}\right)_{t}}{d\left(\mu^{+}\right)_{t}}(x)=\frac{H_{\underbar{\ensuremath{\epsilon}},N_{t}}^{'}(x)}{H_{0,N_{t}}'(x)}.
\]

Writing $\tilde{H}_{\ep,k,t}$ for the function $H_{\delta(\ep,k),N_{t}}$
with 
\[
\delta\left(\ep,k\right)_{j}:=\begin{cases}
\epsilon_{j}, & 1\leq j\leq k\\
0, & j>k
\end{cases},
\]
and noticing that $\tilde{H}_{\delta\left(\ep,0\right),N_{t}}=H_{0,N_{t}}$
we get 
\begin{equation}
z_{t}(x)=\prod_{k=1}^{t}\frac{\tilde{H}_{\ep,k,t}'(x)}{\tilde{H}_{\ep,k-1,t}'(x)}.\label{eq: the first formula for z_t}
\end{equation}
By \cite[p. 527 remark 2]{Shi} it remains to show that we can choose
$\underline{\varepsilon}$ such that if for all $k\in\NN$, $\epsilon_{k}<\varepsilon_{k}$,
then $\left\{ z_{t}\right\} _{t=1}^{\infty}$ is uniformly integrable
with respect to $\mu$. We proceed to show how to choose $\underbar{\ensuremath{\varepsilon}}$.
Let $x\in\TT\backslash{\rm Bd}\left(S\right)$. 

Fix $k\in\NN$. By the chain rule and the fact that $H_{\delta\left(\ep,k\right),M_{k-1}}=H_{\delta\left(\ep,k-1\right),M_{k-1}}$
one sees that 
\[
\frac{\tilde{H}_{\ep,k,t}'(x)}{\tilde{H}_{\ep,k-1,t}'(x)}=\left(\prod_{l=M_{k-1}+1}^{N_{k}}\frac{\rho_{\delta\left(\ep,k\right),l}(x)}{\rho_{\delta\left(\ep,k-1\right),l}(x)}\right)\cdot\left(\prod_{l=N_{k}}^{N_{t}}\frac{\rho_{\delta\left(\ep,k\right),l}(x)}{\rho_{\delta\left(\ep,k-1\right),l}(x)}\right).
\]
First we will want to prove that if $\epsilon_{k}$ is small enough,
then 
\begin{equation}
\prod_{l=N_{k}}^{N_{t}}\frac{\rho_{\delta\left(\ep,k\right),l}(x)}{\rho_{\delta\left(\ep,k-1\right),l}(x)}\leq e^{3(1.6)^{-N_{k}}}.\label{eq: the almost markov property for the derivatives}
\end{equation}
To see (\ref{eq: the almost markov property for the derivatives}),
first notice that since for every $s\geq k$ and $N_{s}<n<M_{s}$,
\[
h_{\delta\left(\ep,k\right),n}=h_{\delta\left(\ep,k-1\right),n}=id,
\]
then for all $k\leq s\leq t-1$, 
\[
\prod_{l=N_{s}}^{M_{s}-1}\frac{\rho_{\delta\left(\ep,k\right),l}(x)}{\rho_{\delta\left(\ep,k-1\right),l}(x)}=1.
\]
Secondly, for $s\geq k$ and $M_{s}<n\leq N_{s+1}$ there exists $w\in\Sigma_{{\bf A}}$
such that $x\in C_{[w]_{1}^{n}}$. If $w_{n}\neq1$ then $\rho_{\delta\left(\ep,k\right),l}(x)=\rho_{\delta\left(\ep,k\right),l}(x)=1$.
Otherwise notice that for $\eta\in\{\delta\left(\ep,k\right),\delta\left(\ep,k-1\right)\}$,
$\eta_{s+1}=0$ and $H_{\eta,n-1}(x)$ is to the right of the point
in $\frac{1}{\p}$ proportion in $H_{\eta,n-1}\left(C_{[w]_{1}^{n}}\right)$
if and only if $x$ is to the right of the point in $\frac{1}{\p}$
in $C_{[w]_{1}^{n}}$. Therefore for all $s\geq k$ and $M_{s}<l\leq N_{s+1}$,
$\rho_{\delta\left(\ep,k\right),l}(x)=\rho_{\delta\left(\ep,k-1\right),l}(x)$
and by Lemma \ref{lem: bound on h_M_t derivative}.(i), 
\begin{eqnarray*}
\prod_{s=k}^{t-1}\prod_{l=M_{s}}^{N_{s+1}}\frac{\rho_{\delta\left(\ep,k\right),l}(x)}{\rho_{\delta\left(\ep,k-1\right),l}(x)} & = & \prod_{s=k}^{t-1}\frac{\rho_{\delta\left(\ep,k\right),M_{s}}(x)}{\rho_{\delta\left(\ep,k-1\right),M_{s}}(x)}\\
 & \leq & \prod_{s=k}^{t-1}\frac{e^{(1.6)^{-N_{t}}}}{e^{-(1.6)^{-N_{t}}}}\leq e^{3(1.6)^{-N_{k}}}\ \ \ \square_{\eqref{eq: the almost markov property for the derivatives}}
\end{eqnarray*}
We remark here that similarly one can get that 
\[
\prod_{s=k}^{t-1}\prod_{l=M_{s}}^{N_{s+1}}\frac{\rho_{\delta\left(\ep,k\right),l}(x)}{\rho_{\delta\left(\ep,k-1\right),l}(x)}\geq e^{-3(1.6)^{-N_{k}}},
\]
which in turn shows that there exists $c>1$ such that 
\begin{equation}
z_{t}(x)=c^{\pm1}\prod_{k=1}^{t}\prod_{l=M_{k-1}}^{N_{k}}\frac{\rho_{\delta\left(\ep,k\right),l}(x)}{\rho_{\delta\left(\ep,k-1\right),l}(x)},\label{eq: z_t is almost as for Markov chains}
\end{equation}
If for every $t<k$, we have chosen $M_{t}$ to be large enough so
that Lemma \ref{lem: bound on h_M_t derivative}.(i) holds then there
exists $c>0$ such that 
\[
z_{t}(x)=c^{\pm1}\prod_{k=1}^{t}\prod_{l=M_{k-1}+1}^{N_{k}}\frac{\rho_{\delta\left(\ep,k\right),l}(x)}{\rho_{\delta\left(\ep,k-1\right),l}(x)}.
\]

As $\epsilon_{k}\to0$,
\[
\psi'_{k}(x):=\psi'_{\epsilon_{k},\lambda_{k}}(x)\xrightarrow{}\psi'_{0,\lambda_{k}}(x)\ \mu\text{ a.e. }x,
\]
It follows that 
\[
\prod_{l=M_{k-1}+1}^{N_{k}}\frac{\rho_{\delta\left(\ep,k\right),l}(x)}{\rho_{\delta\left(\ep,k-1\right),l}(x)}\xrightarrow[\epsilon_{k}\to0]{}1\ \mu\ a.e.\ x.
\]
By Egorov's Theorem there exists $A_{k}\in\BB_{\TT}$, with $\mu\left(A_{k}\right)>1-\frac{1}{2^{k}\prod_{r=1}^{k}\left(\lambda_{r}\right)^{4n_{r}}}$
such that 
\[
\prod_{l=M_{k-1}+1}^{N_{k}}\frac{\rho_{\delta\left(\ep,k\right),l}(x)}{\rho_{\delta\left(\ep,k-1\right),l}(x)}\xrightarrow[\epsilon_{k}\to0]{}1,\ \text{uniformly in\ }x\in A_{k}.
\]
The lower bound on the measure of $A_{k}$ is chosen because for every
$\epsilon_{k}>0$ 
\begin{equation}
\max_{x,y\in\TT}\prod_{l=M_{k-1}+1}^{N_{k}}\frac{\rho_{\delta\left(\ep,k\right),l}(x)}{\rho_{\delta\left(\ep,k-1\right),l}(y)}\lessapprox\left(\max_{x,y\in\TT}\frac{\psi'_{\epsilon_{k},\lambda_{k}}(x)}{\psi'_{0,\lambda_{k}}(y)}\right)^{n_{k}}=\left(\lambda_{k}\right)^{4n_{k}}.\label{eq: the bound on the frac of rho's}
\end{equation}
Now we are finally in a position to define the sequence $\underbar{\ensuremath{\varepsilon}}$.
Let $\varepsilon_{k}$ be small enough so that for every $\ep$ with
$\epsilon_{k}<\varepsilon_{k}$ and $x\in A_{k}$, 
\[
1-\frac{1}{k^{2}}\leq\prod_{l=M_{k-1}+1}^{N_{k}}\frac{\rho_{\delta\left(\ep,k\right),l}(x)}{\rho_{\delta\left(\ep,k-1\right),l}(x)}\leq1+\frac{1}{k^{2}}.
\]
Let $\ep$ which satisfies for every $k\in\NN$, $\epsilon_{k}<\varepsilon_{k}$.
For large $M$, if for some $n\in\NN$ and $x\in\TT$, $z_{n}(x)>M$,
then there exists $q=q(M)\le n$ such that $x\in\cup_{r=q}^{n}A_{r}^{c}$.
Therefore by (\ref{eq: z_t is almost as for Markov chains}) and decomposing
the set $\left[z_{n}>M\right]$ by the last $r\leq n$ for which $x\in A_{r}^{c}$,
\begin{eqnarray*}
\int_{\left[z_{n}>M\right]}z_{n}(x)d\mu(x) & \leq & c\int_{\left[z_{n}>M\right]}\left(\prod_{k=1}^{n}\prod_{l=M_{k-1}+1}^{N_{k}}\frac{\rho_{\delta\left(\ep,k\right),l}(x)}{\rho_{\delta\left(\ep,k-1\right),l}(x)}\right)d\mu(x)\\
 & \leq & c\sum_{r=q(M)}^{n}\int_{\left[A_{r}^{c}\cap\left(\cap_{j=r+1}^{n}A_{j}\right)\right]}\left(\prod_{k=1}^{n}\prod_{l=M_{k-1}+1}^{N_{k}}\frac{\rho_{\delta\left(\ep,k\right),l}(x)}{\rho_{\delta\left(\ep,k-1\right),l}(x)}\right)d\mu(x).
\end{eqnarray*}
Since for $x\in A_{j}$, $\prod_{l=M_{j-1}+1}^{N_{j}}\frac{\rho_{\delta\left(\ep,k\right),l}(x)}{\rho_{\delta\left(\ep,k-1\right),l}(x)}\leq1+\frac{1}{j^{2}}$,
\begin{eqnarray*}
\sum_{r=q(M)}^{n}\int_{\left[A_{r}^{c}\cap\left(\cap_{j=r+1}^{n}A_{j}\right)\right]}\left(\prod_{k=1}^{n}\prod_{l=M_{k-1}+1}^{N_{k}}\frac{\rho_{\delta\left(\ep,k\right),l}(x)}{\rho_{\delta\left(\ep,k-1\right),l}(x)}\right)d\mu(x) & \leq\\
\sum_{r=q(M)}^{n}\prod_{j=r+1}^{n}\left(1+\frac{1}{j^{2}}\right)\int_{A_{r}^{c}}\left(\prod_{k=1}^{r}\prod_{l=M_{k-1}+1}^{N_{k}}\frac{\rho_{\delta\left(\ep,k\right),l}(x)}{\rho_{\delta\left(\ep,k-1\right),l}(x)}\right)d\mu\left(x\right) & \leq\\
\left[\prod_{j=1}^{\infty}\left(1+\frac{1}{j^{2}}\right)\right]\sum_{r=q(M)}^{n}\mu\left(A_{r}^{c}\right)\max_{x\in\TT}\left(\prod_{k=1}^{r}\prod_{l=M_{k-1}+1}^{N_{k}}\frac{\rho_{\delta\left(\ep,k\right),l}(x)}{\rho_{\delta\left(\ep,k-1\right),l}(x)}\right) & \lessapprox & \ \ \text{by \eqref{eq: the bound on the frac of rho's} }\\
\left[\prod_{j=1}^{\infty}\left(1+\frac{1}{j^{2}}\right)\right]\sum_{r=q(M)}^{n}\frac{1}{2^{r}} & \lessapprox & 2^{-q(M)}\prod_{j=1}^{\infty}\left(1+\frac{1}{j^{2}}\right).
\end{eqnarray*}
When $M\to\infty$ then $q(M)\to\infty$ and therefore 
\[
\sup_{n\in\NN}\int_{\left[z_{n}>M\right]}z_{n}(x)d\mu(x)\lessapprox2^{-q(M)}\prod_{j=1}^{\infty}\left(1+\frac{1}{j^{2}}\right)\to0\ \text{as}\ M\to\infty.
\]
 This shows that $\left\{ z_{n}\right\} $ is uniformly integrable
and hence $m_{{\rm \TT}}\circ\mathfrak{h}_{\ep}\sim\mu^{+}$.
\end{proof}

\subsubsection{The modified induction process for choosing $\left\{ \lambda_{k},N_{k},M_{k},n_{k},m_{k},\epsilon_{k}\right\} $
\label{subsec:The-modified-induction} and the proof of Theorem \ref{thm: the one dimensional perturbation}.}

In the course of the construction here we arrived at two conditions
on $\{\epsilon_{k}\}$ and two extra conditions on $\left\{ \lambda_{k},M_{k}\right\} $.
In order to show the existence of these sequences one has to modify
the induction process of Section \ref{sec:Type--Markov} as follows
and insert the choice of $\left\{ \epsilon_{t}\right\} $ in the induction. 

In the proof of the previous Lemmas we have an extra condition on
the size of $M_{t}$ (or $m_{t}=M_{t}-N_{t}$) which is determined
by $\left\{ N_{s},\lambda_{s},M_{s-1},\epsilon_{s}\right\} _{s=1}^{t}$. 

The choice of $\varepsilon_{t+1}$ in Lemma \ref{Lem:  the pertubation of the beta shift},
$\tilde{\delta}_{t+1}$ in Lemma \ref{lem: Differentiablility of g}
and $\epsilon_{t+1}$ in Proposition \ref{lem: bound on h_M_t derivative}
is determined by $\left\{ N_{s},\lambda_{s},M_{s-1},\epsilon_{s}\right\} _{s=1}^{t}$
and $\left\{ N_{t+1},M_{t}\right\} $. We also need to take care that
\[
1.6\leq\p\cdot\left(\prod_{t\in\NN}\lambda_{t}^{\pm2M_{t-1}}\right)\cdot\exp\left(\pm\sum_{t=1}^{\infty}2^{-N_{t}+4}\right)\leq1.7.
\]
 This shows now that the order of choice in the induction is as follows
\[
\left\{ \lambda_{s},n_{s},N_{s},m_{s},M_{s},\epsilon_{s}\right\} _{s=1}^{t}\Rightarrow\lambda_{t+1}\Rightarrow\left\{ n_{t+1},N_{t+1}\right\} \Rightarrow\epsilon_{t+1}\Rightarrow\left\{ m_{t+1},M_{t+1}\right\} .
\]
 The modifications needed to be done in the inductive construction
are: First change the condition (\ref{eq:condition on Lambda}) on
$\lambda_{t+1}$ with the condition 
\[
\lambda_{t+1}^{2M_{t}}\leq\exp\left(2^{-N_{t}}\right),
\]
as this involves making $\lambda_{t+1}$ smaller this choice is valid.
This gives that
\[
\p\cdot\left(\prod_{t\in\NN}\lambda_{t}^{\pm2M_{t-1}}\right)\cdot\exp\left(\pm\sum_{t=1}^{\infty}2^{-N_{t}+4}\right)=\p\cdot\exp\left(\pm\sum_{t=1}^{\infty}2^{-N_{t}+5}\right).
\]
By demanding now that $N_{1}>20$, we get 
\[
\p\cdot\left(\prod_{t\in\NN}\lambda_{t}^{\pm2M_{t-1}}\right)\cdot\exp\left(\pm\sum_{t=1}^{\infty}2^{-N_{t}+4}\right)=\p\cdot\exp\left(\pm\sum_{t=1}^{\infty}2^{-N_{t}+5}\right)\in(1.6,1.7)
\]
as we required. There is no further change in the inductive choice
of $\lambda_{t},n_{t},N_{t}$ as they will not depend on $\underbar{\ensuremath{\epsilon}}$. 

Given $\left\{ \lambda_{s},n_{s},N_{s},m_{s},M_{s},\epsilon_{s}\right\} _{s=1}^{t}$
and $N_{t+1}$ we choose $\epsilon_{t+1}$ to be small enough so that
the conclusions of Lemma \ref{lem: bound on h_M_t derivative}.(ii),
Lemma \ref{lem: Differentiablility of g} and Lemma \ref{Lem:  the pertubation of the beta shift}
hold true. 

Then we choose $m_{t+1}$ based on the original constraints from Section
\ref{sec:Type--Markov} together with the restriction that $M_{t+1}=m_{t+1}+N_{t+1}$
is large enough so that the conclusion of Lemma \ref{lem: bound on h_M_t derivative}.(i)
is true. Since this involves perhaps enlarging $m_{t+1}$ it is consistent
with the other constraints of the induction. 

\begin{proof}[Proof of Theorem \ref{thm: the one dimensional perturbation}]

Choose $\left\{ \lambda_{k},N_{k},M_{k},n_{k},m_{k},\epsilon_{k}\right\} _{k=1}^{\infty}$
as in the inductive construction. Build the Markovian measure $\eta=\M{\P_{k},\pi_{k}:k\in\ZZ}$
determined by $\left\{ \lambda_{k},N_{k},M_{k},n_{k},m_{k}\right\} _{k=1}^{\infty}$,
$\mu:=\Phi_{*}\left(\eta\right)$ and $\mu^{+}=\Theta_{*}\left(\M{\P_{k},\pi_{k}:\ k\in\NN}\right)$.

Part (i) follows from Theorem \ref{thm: lambdas are...} since $\left\{ \lambda_{k},N_{k},M_{k},n_{k},m_{k}\right\} _{k=1}^{\infty}$
satisfy the constraints of the inductive construction in Section \ref{sec:Type--Markov}
hence it is a type ${\rm III}_{1}$ measure for the shift. 

(ii) and (iii): Since we chose $\ep=\left\{ \epsilon_{k}\right\} $
so that the conclusion of Lemma \ref{Lem:  the pertubation of the beta shift}
holds, it follows that $m_{\TT}\circ\mathfrak{h}_{\ep}\sim\mu^{+}$.
As we chose the sequences so that the conditions of Lemma \ref{lem: Differentiablility of g}
hold, for all $t\in\NN$ and $x\in\TT$, 
\[
\mathfrak{g}_{N_{t+1}}^{'}(x)=\exp\left(\pm2^{-N_{t}+4}\right)\lambda_{t+1}^{\pm2M_{t}}\mathfrak{g}_{N_{t}}'(x).
\]
Therefore $\left\{ \mathfrak{g}_{N_{t}}'\right\} _{t=1}^{\infty}$
is a Cauchy sequence in $C\left(\TT\right)$, it's limit function
satisfies 
\[
1.6\leq\mathfrak{g}'(x)=\p\cdot\left(\prod_{t\in\NN}\lambda_{t}^{\pm2M_{t-1}}\right)\cdot\exp\left(\pm\sum_{t=1}^{\infty}2^{-N_{t}+4}\right)\leq1.7.
\]

\end{proof}

\section{Type ${\rm III}_{1}$ Anosov Diffeomorphisms\label{sec:Type III Anosov}}

Let $\left\{ \lambda_{k},m_{k},n_{k},M_{k},N_{k}\right\} _{k=1}^{\infty}$
and $\epsilon=\left\{ \epsilon_{k}\right\} _{k=1}^{\infty}$ as in
Theorem \ref{thm: the one dimensional perturbation} and let $\h_{\epsilon}$
be the resulting function. Set $\mathfrak{H}_{\epsilon}(x,y):=(\h_{\epsilon}(x),y)$
and 
\[
\mathfrak{G}(x,y):=\mathfrak{H}_{\epsilon}\circ\tilde{f}\circ\mathfrak{H}_{\epsilon}^{-1}(x,y)=\begin{cases}
\left(\mathfrak{g}\left(x\right),-\p^{-1}y\right), & 0\leq x\leq1/\p\\
\left(\mathfrak{g}\left(x\right),-\p^{-1}\left(y-\frac{\p^{2}}{\p+2}\right)\right), & 1/\p\leq x\leq1
\end{cases}
\]
In the construction of Section \ref{sec:Type--Markov}, $P_{k}={\bf Q}$
for all $k<0$. Writing ${\bf m}_{\MM}$ for the Lebesgue measure
on $\MM{}_{\sim}$ one then has 
\[
d{\bf m}_{\MM}\circ\mathfrak{H}_{\underbar{0}}(x,y)=d\mu^{+}(x)dy=d\mu(x,y),
\]
or in other words ${\bf m}_{\MM}\circ\mathfrak{H}_{\underbar{0}}=\Phi_{*}\M{\P_{k},\pi_{k}:\ k\in\ZZ}$.
Therefore since $m_{{\rm \TT}}\circ\h_{\ep}\sim\mu^{+}$, 
\[
\mu:=\Phi_{*}\M{\P_{k},\pi_{k}:\ k\in\ZZ}\sim{\bf m}_{\MM}\circ\mathfrak{H}_{\ep}.
\]
Consequently $\left(\MM_{\sim},\BB_{\MM_{\sim}},{\bf m}_{\MM},\mathfrak{G}\right)$
is a type ${\rm III}_{1}$ transformation. This is because $\left(\MM_{\sim},\BB_{\sim},{\bf m}_{\MM},\mathfrak{G}\right)$
is measure theoretically equivalent to $\left(\MM_{\sim},\BB_{\sim},{\bf m}_{\MM}\circ\mathfrak{H}_{\ep},\tilde{f}\right)$
which is orbit equivalent to $\left(\MM_{\sim},\BB_{\sim},\mu,\tilde{f}\right)$. 

By Remark \ref{RK: An important feature of the perturbation}, $\mathfrak{G}{}_{\epsilon}$
is one to one and onto. In addition, for every $\left(x,y\right)\notin\partial\MM$,
$\mathfrak{G}$ is differentiable in a neighborhood of $(x,y)$ as
all the partial derivatives are continuous in $\MM\backslash\partial\MM$,
and 
\[
D_{\mathfrak{G}}(x,y)=\left(\begin{array}{cc}
\mathfrak{g}'\left(x\right) & 0\\
0 & -\p^{-1}
\end{array}\right).
\]
The problem is that $\mathfrak{G}$ when viewed as a transformation
of $\MM_{\sim}$ is not even continuous on the horizontal lines of
$\partial\MM$. 

We define a sequence of functions $\rr_{n}(x,y):\TT\times\left[-\p/(\p+2),\p^{2}/(\p+2)\right]\to\TT,n\in\NN$
using the construction of the previous section. This defines a sequence
$h_{n,y}(\cdot):=\rr_{n}\left(\cdot,y\right):(\TT\ {\rm or}\ [0,1/\p])\to\TT$
and 
\[
\mathfrak{h}_{y}(x):=\lim_{n\to\infty}h_{n,y}\circ h_{n-1,y}\circ\cdots h_{1,y}(x),
\]
where we will take care that the limit exists. The new examples will
then be of the form 
\[
\mathfrak{Z}(x,y):=\begin{cases}
\left(\mathfrak{h}_{-y/\varphi}\circ S\circ\mathfrak{h}_{y}^{-1}(x),-y/\p\right), & x\leq1/\p\\
\left(\mathfrak{h}_{-y/\varphi+\frac{\varphi}{\varphi+2}}\circ S\circ\mathfrak{h}_{y}^{-1}(x),-y/\p+\frac{\p}{\p+2}\right), & 1/\p\leq x\leq1
\end{cases}:\ \MM_{\sim}\to\MM_{\sim}.
\]

Particular care in the definition of $\mathfrak{h}_{y}$ is taken
in order to ensure that if $(x,y)\sim\left(\hat{x},\hat{y}\right)$
then $\mathfrak{h}_{y}(x)=\mathfrak{h}_{\hat{y}}\left(\hat{x}\right)$
as this is needed for the continuity of $\mathfrak{Z}$ on $\partial\mathbb{M}$. 

\subsection{Definition of the coupling time on the horizontal boundary of $\protect\MM$}

Denote by 
\begin{eqnarray*}
U_{1}: & = & \left([0,1/\p]\times\left(\frac{\p^{2}}{\p+2}-\frac{1}{\p^{10}},\frac{\p^{2}}{\p+2}\right]\right)\bigcup\left([1/\p,1]\times\left[-\frac{\p}{\p+2},-\frac{\p}{\p+2}+\frac{1}{\p^{10}}\right)\right),\\
U_{2} & := & \left((1/\p,1)\times\left(\frac{1}{\p+2}-\frac{1}{\p^{10}},\frac{1}{\p+2}\right]\right)\bigcup\left([1/\p,1]\times\left[-\frac{\p}{\p+2},-\frac{\p}{\p+2}+\frac{1}{\p^{10}}\right)\right)
\end{eqnarray*}
and $\mathbb{M}\backslash U:=U_{1}\cup U_{2}$. Then $\MM\backslash U$
is a neighborhood of the horizontal lines of $\partial\MM$. 

\begin{figure}[h]
\includegraphics[width=9cm,height=4cm]{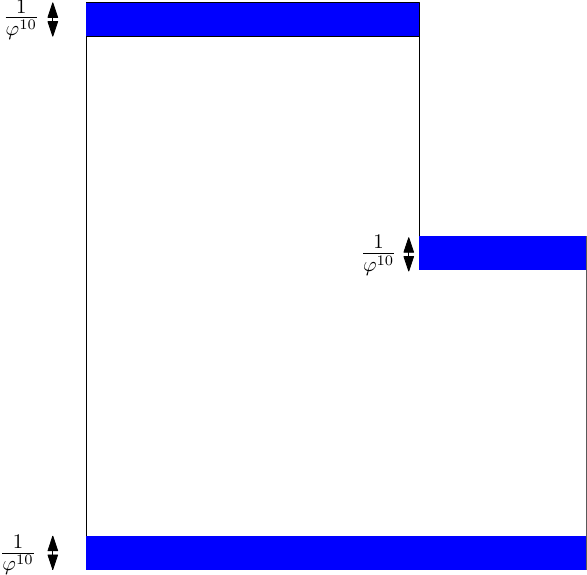}

\caption{The bands are $\protect\MM\backslash U$}
\end{figure}

In our construction for any $(x,y)\in U$, 
\[
\rr_{n}(x,y)=h_{n}(x),
\]
with $h_{n}$ the functions in the one dimensional example in Section
\ref{sec:Type III perturbations of the Golden Mean Shift}. This means
that for any $(x,y)\in U,$
\[
\mathfrak{h}_{y}(x)=\mathfrak{h}_{\ep}(x).
\]
We now will proceed to specify the construction of $\mathfrak{h}_{y}(x)$
for $(x,y)\in\MM\backslash U$.

The first step is to do a coupling on the horizontal lines which roughly
tells us where is the break that the problem that although $\left(x,y\right)$
is the same point in $\MM_{\sim}$ as $\left(\hat{x},\hat{y}\right)$,
\[
\left(\mathfrak{h}_{\ep}(x),y\right)\neq\left(\mathfrak{h}_{\ep}\left(\hat{x}\right),\hat{y}\right).
\]
For example consider the case $x=1/\p^{3}$, $y=\frac{\p^{2}}{\p+2}$
and $\hat{x}=1/\p$, $\hat{y}=-\frac{\p}{\p+2}$. The point $\frac{1}{\p}$
is a fixed point for $\mathfrak{h}_{\ep}$ meaning $\mathfrak{h}_{\ep}\left(1/\p\right)=1/\p$.
Since $\mathfrak{h}_{\ep}\left(1/\p^{3}\right)=\frac{\lambda_{1}}{\p\left(1+\lambda_{1}\p\right)}\neq\frac{1}{\p^{3}}$
we get 
\[
\left(\mathfrak{h}_{\ep}(x),y\right)\neq\left(\mathfrak{h}_{\ep}\left(\hat{x}\right),\hat{y}\right)=\left(\hat{x},\hat{y}\right).
\]
However if we took care that $\frac{1}{\p^{3}}$ is a fixed point
for $\mathfrak{h}_{y}$ then we will have the desired equality. It
turns out that the correct way to do this will be by setting $h_{1,y}|_{\left[0,1/\p^{3}\right)}=h_{2,y}|_{\left[0,1/\p^{3}\right)}=Id|_{\left[0,1/\p^{3}\right)}$
and to start perturbing (similarly as in the definition of $h_{n}$
from the previous Section) from $n\geq3$. In general we will have
a decomposition of the horizontal lines of $\partial\MM$ to $\left\{ V_{i}\right\} _{i=1}^{\infty}$
and we will start perturbing at $V_{i}$ from $n\geq i+1$. 

To be more precise the Horizontal boundary consists of the lines $[0,1/\p)\times\left\{ \p^{2}/\left(\p+2\right)\right\} $,
$[1/\p,1)\times\left\{ 1/\left(\p+2\right)\right\} $ and $\TT\times\left\{ -\p/\left(\p+2\right)\right\} $.
We look at a countable partition of the horizontal lines $\partial\MM$
which are identified by $\sim$ and couple them in a time $T\in\NN$
such that in the symbolic space on $\TT$, the move $w(T)_{T}\to1$
is possible for both pieces identified. 

\subsubsection{The horizontal subsegments of $\partial\mathbb{M}$ and their coupling
time: }
\begin{enumerate}
\item $V_{1}:=\left[0,1/\p^{2}\right)\times\left\{ -\frac{\p}{\p+2}\right\} \sim\left[1/\p,1\right)\times\left\{ \frac{1}{\p+2}\right\} $.
In this case $[0,1/\p^{2})=C_{[1]_{1}^{1}}$ and $\left[1/\p,1\right)=C_{[2]_{1}^{1}}$
and $T\left(V_{1}\right)=2$. 
\item $V_{2}:=\left[0,1/\p^{3}\right)\times\left\{ \frac{\p^{2}}{\p+2}\right\} \sim\left[1/\p^{2},1/\p\right)\times\left\{ -\frac{\p}{\p+2}\right\} $.
Here $[0,1/\p^{3})=C_{[11]_{1}^{2}}$ and $\left[1/\p^{2},1/\p\right)=C_{[32]_{1}^{2}}$
and $T\left(V_{2}\right)=3$. 
\item $V_{3}:=\left[1/\p^{3},1/\p^{2}\right)\times\left\{ \frac{\p^{2}}{\p+2}\right\} \sim\left[1/\p,1/\p+1/\p^{4}\right)\times\left\{ -\frac{\p}{\p+2}\right\} $.
Here $[1/\p,1/\p^{2})=C_{[132]_{1}^{3}}$ and $\left[1/\p,1/\p+1/\p^{4}\right)=C_{[211]_{1}^{3}}$
and $T\left(V_{3}\right)=4$.
\item $V_{4}:=\left[1/\p^{2},1/\p^{2}+1/\p^{5}\right)\times\left\{ \frac{\p^{2}}{\p+2}\right\} \sim\left[1/\p+1/\p^{4},1/\p+1/\p^{3}\right)\times\left\{ -\frac{\p}{\p+2}\right\} $.
Here $C_{[3211]}=\left[1/\p^{2},1/\p^{2}+1/\p^{5}\right)$ and $C_{[2132]}=\left[1/\p+1/\p^{4},1/\p+1/\p^{3}\right)$. 
\item For general $j>4$, $V_{j}:=C_{\left[{\bf w}(j)\right]_{1}^{j}}\times\left\{ \frac{\p^{2}}{\p+2}\right\} \sim C_{\left[\tilde{{\bf w}}(j)\right]_{1}^{j}}\times\left\{ -\frac{\p}{\p+2}\right\} $
where ${\bf w}(j),\tilde{{\bf w}}(j)$ are the following words of
length $j$, 
\begin{eqnarray*}
{\bf w}(j) & = & \begin{cases}
32\cdots32132 & j\ \text{odd}\\
32\cdots3211, & j\ {\rm \text{even}}
\end{cases}\\
\tilde{{\bf w}}(j) & = & \begin{cases}
23\cdots23211, & j\ \text{odd}\\
23\cdots232132 & j\ \text{even}.
\end{cases}
\end{eqnarray*}
As is expected for all $j\geq2$, $T\left(V_{j}\right)=j-1$. The
following is immediate from the definition.
\end{enumerate}
\begin{claim}
\label{Cl: how the boundary decomposition is moved by f} For any
$j\geq2$, 
\[
\tilde{f}\left(V_{j}\right)=V_{j-1},
\]
and 
\[
\tilde{f}\left(V_{1}\right)=[0,1/\p)\times\{\frac{1}{\varphi+2}\}\subset U.
\]
\end{claim}

\subsubsection{Definition of the perturbation maps $h_{n,y}$:}

For $w\in\Sigma_{{\bf A}}$ and $n\in\NN$, we write again $C_{[w]_{1}^{n}}:=\left[\underline{x}_{n}(w),\bar{x}_{n}(w)\right)$.
Let 
\[
{\bf u}(x,y):=\begin{cases}
\min\left\{ \frac{\p^{2}}{\p+2}-y,y+\frac{\p}{\p+2}\right\} , & 0\leq x\leq\frac{1}{\p}\\
\min\left\{ \frac{1}{\p+2}-y,y+\frac{\p}{\p+2}\right\}  & \frac{1}{\p}\leq x\leq1
\end{cases}
\]
be the minimal distance of $(x,y)$ to the horizontal lines of $\partial\MM$.
In addition we will write ${\bf {\bf y}}(x,y):\MM\to\left\{ -\frac{\p}{\p+2},\frac{1}{\p+2},\frac{\p^{2}}{\p+2}\right\} $
to be the value so that 
\[
{\bf u}(x,y)=\left|{\bf y}(x,y)-y\right|.
\]
Under that notation $(x,{\bf y}(x,y))$ is the closest point to $(x,y)$
in the horizontal boundary. Let $(x,y)\in\MM_{\sim}$.

\textbf{Case 1} $(x,y)\in U$: we do the regular construction as in
Section \ref{sec:Type III perturbations of the Golden Mean Shift}.
That is for any $N_{t}<n<M_{t}$, $h_{n}$ is the identity. For any
$M_{t}<n\leq N_{t}$, if $w_{n}=1$ then $h_{n}|_{H_{n-1}\left(C_{[w]_{1}^{n}}\right)}$
is a rescaling of $\psi_{t}$ to the interval $H_{n-1}\left(C_{[w]_{1}^{n}}\right)$
and if $w_{n}\neq1$ then $h_{n}|_{H_{n-1}\left(C_{[w]_{1}^{n}}\right)}$
is the identity. If for some $t$, $n=M_{t}$ then $\left.h_{M_{t}}\circ H_{M_{t}-1}\right|_{C_{[w]_{1}^{M_{t}}}}$
is the distribution correction function in the construction. Finally
we set 
\[
\rr_{n}(x,y)=h_{n,y}(x):=h_{n}(x)
\]
and 
\[
K_{n,y}(x):=h_{n,y}\circ h_{n-1,y}\circ\cdots\circ h_{1,y}(x)=H_{n}(x).
\]

\textbf{Case 2}, $(x,y)\notin U$: In this case ${\bf u}\left(x,y\right)<\frac{1}{\p^{10}}$.
Let $(x,{\bf y}(x,y))\in\partial\MM$ be the closest point on the
horizontal lines of $\partial\MM$ to $(x,y)$. Let ${\bf j}(x,y)\in\NN$
be the integer so that 
\[
\left(x,{\bf y}(x,y)\right)\in V_{{\bf j}(x,y)}.
\]
 This means that either $x\in C_{\left[{\bf w}(j)\right]_{1}^{j}}$
(if $x\leq1/\p)$ or $x\in C_{\left[\tilde{{\bf w}}(j)\right]_{1}^{j}}$
($1/\p\leq x<1$). We will define the construction for $x\in C_{\left[{\bf w}(j)\right]_{1}^{j}}$,
the other case being similar. First we define for any $x\in C_{\left[{\bf w}(j)\right]_{1}^{j}}$,
\[
K_{{\bf j}(x,y),{\bf y}(x,y)}(x)=x.
\]

Then for any $(x,y)\in\MM\backslash U$ such that $x\in C_{\left[{\bf w}({\bf j})\right]_{1}^{{\bf j}}}$
we set 
\begin{eqnarray*}
K_{{\bf j},y}(x) & := & \left(H_{{\bf J}}(x)-x\right)\left[3\p^{20}\left({\bf u}(x,y)\right)^{2}-2\varphi^{30}{\bf u}(x,y)^{3}\right]+x
\end{eqnarray*}
For $n>{\bf j}(x,y)$, assume we have defined for all ${\bf j}<k<n$,
$h_{k,y}:=\rr_{k}\left(\cdot,y\right)$ and $x\in C_{[w]_{1}^{n}}\subset C_{\left[{\bf w}({\bf j})\right]_{1}^{{\bf j}}}$.
We set 
\[
K_{n-1,y}|_{C_{[w]_{1}^{n}}}=h_{n-1,y}\circ\cdots h_{{\bf j}+1,y}\circ K_{{\bf j},y}(x),
\]
and 
\[
\mathfrak{l}_{n}(y,w):=m_{\TT}\left(K_{n-1,y}\left(C_{[w]_{1}^{n}}\right)\right)=K_{n-1,y}\left(\bar{x}_{n}\left(w)\right)\right)-K_{n-1,y}\left(\underline{x}_{n}\left(w)\right)\right).
\]
If $w_{n}\neq1$ or $N_{t}<n<M_{t}$ for some $t\in\mathbb{N}$, then
for all $x\in K_{n-1,y}\left(C_{[w]_{1}^{n}}\right)$, $\rr_{n}(x,y):=x$.
If $w_{n}=1$ and $M_{t}<n\leq N_{t}$ and ${\bf j}<n$ then 
\[
\rr_{n}(x,y):=K_{n-1,y}\left(\underline{x}_{n}(w)\right)+\mathfrak{l}_{n}(y,w)\psi_{t}\left(\frac{x-K_{n-1,y}\left(\underline{x}_{n}(w)\right)}{\mathfrak{l}_{n}(y,w)}\right).
\]
Finally if $j\leq M_{t}=n$ then $r_{M_{t},y}$ is the distribution
correction function with $H_{M_{t}-1}$ replaced by $K_{M_{t}-1,y}$. 
\begin{rem}
\label{Rk: properties of q_n}The 2 variable function 
\[
\mathtt{q}_{{\bf J}}(x,{\bf u}):=\left(H_{{\bf J}}(x)-x\right)\left[3\p^{20}{\bf u}^{2}-2\p^{30}{\bf u}^{3}\right]+x
\]
 was chosen because of it's following properties:
\end{rem}
\begin{enumerate}
\item $\mathtt{q}_{{\bf J}}\left(x,\p^{-10}\right)=H_{{\bf J}}\left(x\right)$
and $\qq_{{\bf J}}(x,0)=x$. This means that for $y\in\partial\MM$,
$\rr_{n}(x,y)=x$ and therefore $K_{{\bf j},y}$ interpolates between
the identity map and $H_{{\bf J}}|_{C_{\left[{\bf w}({\bf j)}\right]_{1}^{{\bf j}}}}$. 
\item A consequence of the previous property is that $\frac{\partial\qq_{{\bf j}}}{\partial x}\left(x,0\right)=1$
and $\frac{\partial\qq_{{\bf J}}}{\partial x}\left(x,\p^{-10}\right)=\frac{\partial H_{{\bf J}}}{\partial x}\left(x\right)$.
This is needed in order that $\frac{\partial K_{{\bf J},y}}{\partial x}$
will be continuous in $y$. 
\item $\frac{\partial\qq_{{\bf J}}}{\partial u}\left(x,0\right)=\frac{\partial\qq_{{\bf J}}}{\partial u}\left(x,\p^{-10}\right)=0$
which is necessary for continuity of $\frac{\partial K_{{\bf j},y}}{\partial y}$.
\item $\sup_{0\leq z\leq\p^{-10}}\left|\frac{\partial\qq_{{\bf j}}}{\partial u}(x,z)\right|=\frac{3}{2}\p^{10}\left|H_{{\bf j}}(x)-x\right|$.
We will show that the right hand side is uniformly exponentially small
when ${\bf j\to\infty}$. The control of the derivatives in the $y$
direction is to our opinion the hardest part in this section. 
\end{enumerate}
The idea behind this construction can be summarized as follows: For
a fixed $(x,y)$ which is close enough to the horizontal segment on
the boundary we first look at the coupling time of the interval which
contains the point closest to $(x,y)$ on the boundary. On the boundary
we start to apply the rescaling after the coupling time to ensure
that the resulting map will be a map of $\MM_{\sim}$ (respects the
equivalence relation). Inside $U$ we just start perturbing from the
start and in what remains we do an interpolation using $q_{{\bf J}}$,
of the map on the boundary and the map on $U$. 

\subsubsection{Definition of $\mathfrak{Z}_{N}$ and the new examples of Anosov
diffeomorphisms:}

Define $\mathfrak{Z}_{N}:\MM_{\sim}\to\MM_{\sim}$, 
\[
\mathfrak{Z}_{N}(x,y)=\begin{cases}
\left(K_{N,-y/\p}\circ S\circ K_{N,y}^{-1}(x),-y/\p\right), & (x,y)\in R_{1}\cup R_{3}\\
\left(K_{N,-y/\p+\frac{\p}{\p+2}}\circ S\circ K_{N,y}^{-1}(x),-y/\p+\frac{\p}{\p+2}\right), & (x,y)\in R_{2}
\end{cases}.
\]

\begin{rem}
In the construction of the previous subsection for every $n\in\NN$,
$x\in\left\{ 0,1/\p^{2},1/\p\right\} $ are fixed points for $h_{\ep,n}$
(Remark \ref{RK: An important feature of the perturbation}). This
remains true for $\mathtt{r}_{n}$ in the sense that for all $y$
and $x\in\left\{ 0,1/\p^{2},1/\p\right\} $ , $\mathtt{r}_{n}\left(x,y\right)=\hat{r}_{n}(x,y)=x.$
This shows that $\mathfrak{Z}_{N}$ is continuous. In addition if
$x$ is an endpoint of the segment $K_{n,y}\left(C_{[w]_{1}^{n}}\right)$
for some $w\in\Sigma_{{\bf A}}$ and $y$, then $\frac{\partial\rr_{n}}{\partial x}(x,y)=1.$
This gives that $\mathfrak{Z}_{N}$ is $C^{1}$. The invariance of
the Markov partition $\left\{ J_{1},J_{2},J_{3}\right\} $ of $S$
under $K_{y,n}$ gives that $\mathfrak{Z}_{N}$ is one to one and
onto and 
\[
\mathfrak{Z}_{n}^{-1}(x,y)=\begin{cases}
\left(K_{n,-\p y}\circ\left(\left.S\right|_{J_{1}\cup J_{3}}\right)^{-1}\circ K_{n,y}^{-1}(x),-\p y\right), & -\frac{\p}{\p+2}\leq y\leq\frac{1}{\p+2}\\
\left(K_{n,-\p y+\p^{2}/(\p+2)}\circ\left(\left.S\right|_{J_{2}}\right)^{-1}\circ K_{n,y}^{-1}(x),-\p y+\frac{\p^{2}}{\p+2}\right), & \frac{1}{\p+2}<y\leq\frac{\p^{2}}{\p+2}
\end{cases}.
\]
Here $\left(S|_{J_{1}\cup J_{3}}\right)^{-1}(x):=\frac{x}{\p}:[0,1]\to\left[1,1/\p\right]$
is the inverse branch of $S$ to the segment $[0,1/\p]$ and $\left(\left.S\right|_{J_{2}}\right)^{-1}(x)=\frac{x+1}{\p}:[0,1/\p]\to[1/\p,1]$
is the inverse branch of $S$ to the segment $[1/\p,1]$. Since $\mathfrak{Z}_{n}^{-1}$
is $C^{1}$, $\mathfrak{Z}_{n}$ is a diffeomorphism. 
\end{rem}
\begin{thm}
\label{thm: The proof of convergence of y derivative} The sequence
$\mathfrak{Z}_{N_{t}}$ converges in the $C^{1}$ topology to a type
${\rm III}_{1}$ Anosov diffeomorphism. 
\end{thm}
The proof of this Theorem of this paper is by a series of Lemma's.
The first step is to show that $K_{N_{t},y}(x)$ converges uniformly
in $\MM$ as $t\to\infty$. 
\begin{lem}
\label{lem: P_n is exponentionally close to x in intersection points}
If $x\in C_{\left[{\bf w}(J)\right]_{1}^{J}}\cup C_{\left[\tilde{{\bf w}}(J)\right]_{1}^{J}}$
and $J\in\left[M_{t-1},N_{t}\right)$, 
\[
\left|H_{J}(x)-x\right|\leq m_{\TT}\left(C_{[w(x)]_{1}^{J-3}}\right)\leq\p^{-(J-3)},
\]
where $w(x)\in\left\{ {\bf w}(J),\tilde{{\bf w}}(J)\right\} $ is
such that $x\in C_{\left[w(x)\right]_{1}^{J}}$. 
\end{lem}
\begin{proof}
By the form of ${\bf w}(J)\ \text{and }\tilde{{\bf w}}(J)$ one has
that for all $l\leq J-3$, $w(x)_{l}\in\{2,3\}$, hence 
\[
H_{J-3}(x)|_{C_{[w(x)]_{1}^{J-3}}}=x.
\]
Since $\underline{x}_{J-3}(w(x)),\ \bar{x}_{J-3}(w(x))$ are fixed
points of $h_{j-2},h_{j-1}$ and $h_{j}$ , $H_{J}\left(C_{[w(x)]_{1}^{J-3}}\right)=C_{[w(x)]_{1}^{J-3}}$,
the lemma follows.
\end{proof}
\begin{cor}
The limit 
\[
\lim_{n\to\infty}K_{n,y}(x)=:\mathfrak{h}_{y}(x)
\]
 exists uniformly in $\MM$ and is a continuous function and the function
$\mathcal{H}(x,y)=\left(\h_{y}(x),y\right)$ is a homeomorphism of
$\MM_{\sim}$. 
\end{cor}
\begin{proof}
The proof is similar to the proof of Lemma \ref{Cor: A consequence on the bound of eta and chi}.
First we claim that for every $n\in\NN$, 
\begin{equation}
\sup_{(x,y)\in\MM_{\sim}}\left|\rr_{n}(x,y)-x\right|\leq\left(1.5\right)^{-n}.\label{eq: r_n is uniformly close to x}
\end{equation}
This is true since for every $w\in\Sigma_{{\bf A}}$, 
\[
\rr_{n}\left(K_{n-1,y}\left(C_{[w]_{1}^{n}}\right),y\right)=K_{n-1,y}\left(C_{[w]_{1}^{n}}\right),
\]
and consequently
\[
\left|\rr_{n}(x,y)-x\right|\leq m_{{\rm \TT}}\left(K_{n-1,y}\left(C_{[w]_{1}^{n}}\right)\right)\leq\left(\frac{\lambda_{1}^{2}}{\p}\right)^{-n}.\ \ \ \boxtimes\eqref{eq: r_n is uniformly close to x}
\]
The last inequality follows since $\left|\frac{\partial\psi_{l}}{\partial x}\left(x\right)\right|\leq\lambda_{1}^{2}$
for every $l\in\NN$ and thus $\left|\frac{\partial\rr_{k}}{\partial x}\right|_{\infty}\leq\lambda_{1}^{2}.$
Proceeding as in Lemma \ref{Cor: A consequence on the bound of eta and chi},
it follows that for every $y$, $\left\{ K_{n,y}(x)\right\} _{n=1}^{\infty}$
is a Cauchy sequence in the uniform topology. Thus, $\mathfrak{h}_{y}(x)$
is a continuous function in $\MM$ as it is a uniform limit of continuous
functions. Notice that $\mathfrak{h}_{y}$ is a homeomorphism of the
circle for $y\in\left[-\p/\left(\p+2\right),1/(\p+2)\right]$ or of
$\left[0,1/\p\right]$ if $y\in\left[1/\left(\p+2\right),\p^{2}/\left(\p+2\right)\right]$. 

It remains to show that if $\left(x,y\right)\sim\left(\hat{x},\hat{y}\right)$
(for points on $\partial\MM$) then $\left(\h_{y}(x),y\right)\sim\left(\h_{\hat{y}}\left(\hat{x}\right),\hat{y}\right)$.
Let $\left(x,y\right),\left(\hat{x},\hat{y}\right)\in\partial\MM$
with $\left(x,y\right)\sim\left(\hat{x},\hat{y}\right)$. There exists
${\bf j}(x,y)\in\NN$ such that $\left(x,y\right),\left(\hat{x},\hat{y}\right)\in V_{{\bf j}}$.
Since $\left(x,y\right)\sim\left(\hat{x},\hat{y}\right)$, it follows
that for every $n>{\bf j}$ and a word $w\in\Sigma_{{\bf A}}(n)$,
\[
x\in C_{\left[{\bf w}({\bf j)}w\right]_{1}^{n+{\bf j}}}\ \Leftrightarrow\ \hat{x}\in C_{\left[\tilde{{\bf w}}\left({\bf j}\right)w\right]_{1}^{n+{\bf j}}}.
\]
Since for all $n\leq{\bf j}$, $\rr_{n}|_{V_{{\bf J}}}=Id_{\TT}$,
this property and the definition of $r_{n}(\cdot,\cdot)$ yields that
for all $n\in\NN$, 
\[
\left(K_{n,y}(x),y\right)\sim\left(K_{n,\hat{y}}\left(\hat{x}\right),\hat{y}\right).
\]
The Lemma follows by taking $n\to\infty$. 
\end{proof}
Denote the function of the first coordinate by $\mathfrak{z}_{n}(x,y)$.
Our goal is to prove that the limit 
\[
\mathfrak{z}(x,y):=\lim_{t\to\infty}\mathfrak{z}_{N_{t}}(x,y)
\]
exists for all $(x,y)$ and $\mathfrak{z}$ is a $C^{1}\left(\MM_{\sim}\right)$
function with 
\[
1.6\leq\frac{\partial\mathfrak{z}}{\partial x}(x,y)\leq1.7.
\]
The conclusion of hyperbolicity of $\mathfrak{Z}$ will follow from
a standard Lemma in the theory of Lyapunov exponents. 
\begin{lem}
\label{lem: the x-derivative of the two dimensional }If in addition
\[
1.6\leq\p\cdot\lambda_{1}^{6}\left(\prod_{t\in\NN}\lambda_{t}^{\pm2M_{t-1}}\right)\cdot\exp\left(\pm\sum_{t=1}^{\infty}2^{-N_{t}+4}\right)\leq1.7
\]
 then $\frac{\partial\mathfrak{z}}{\partial x}$ is a continuous function
in $\MM_{\sim}$ and 
\[
1.6\leq\frac{\partial\mathfrak{z}}{\partial x}(x,y)\leq1.7.
\]
\end{lem}
\begin{rem*}
The extra condition in this Lemma can easily be inserted into the
inductive construction of the sequence $\left\{ \lambda_{k},M_{k},N_{k},m_{k},n_{k},\epsilon_{k}\right\} _{k=1}^{\infty}$. 
\end{rem*}
\begin{proof}
Let $(x,y)\in\MM_{\sim}$. For convenience to the reader, we will
first show that $\frac{\partial\mathfrak{z}_{N_{t}}}{\partial x}(x,y)$
converges pointwise and $1.6\leq\frac{\partial\mathfrak{z}}{\partial x}(x,y)\leq1.7$
and then argue that the convergence is in fact uniform. 

Let $t\in\NN$ and $(x,y)\in\MM_{\sim}$ be fixed. There exists a
$w=w(x)\in\Sigma_{{\bf A}}$ such that for all $t\in\mathbb{N},$
$x\in K_{N_{t},y}\left(C_{[w]_{1}^{N_{t}}}\right)$. As in the proof
of Lemma \ref{lem: Differentiablility of g}, we write $z_{y}(x)$
to be the unique point in $C_{[w]_{1}^{N_{t}}}$ such that $x=K_{N_{t},y}\left(z_{y}(x)\right)$.
Recall that 
\begin{eqnarray*}
\mathfrak{z}_{N_{t}}(x,y) & = & \begin{cases}
K_{N_{t},-y/\p}\left(\p K_{N_{t},y}^{-1}(x)\right), & 0\leq x\leq1/\p\\
K_{N_{t},-y/\p+\p/\p+2}\left(\p K_{N_{t},y}^{-1}(x)-1\right), & 1/\p\leq x\leq1
\end{cases}\\
 & = & \begin{cases}
K_{N_{t},-y/\p}\left(Sz_{y}(x)\right), & 0\leq x\leq1/\p\\
K_{N_{t},-y/\p+\p/\p+2}\left(Sz_{y}(x)\right), & 1/\p\leq x\leq1
\end{cases}.
\end{eqnarray*}
By the chain rule, the Lemma will follow once we show that uniformly
in $\left(x,y\right)\in\MM_{\sim}$ with $0\leq x\leq1/\p$, 
\begin{eqnarray}
\lim_{t\to\infty}\left(\frac{\partial K_{N_{t},-y/\p}}{\partial x}\left(Sz_{y}(x)\right)\right)\cdot\left(\frac{\partial K_{N_{t},y}}{\partial x}\left(z_{y}(x)\right)\right)^{-1} & \in & \left(\frac{1.6}{\p},\frac{1.7}{\p}\right),\label{eq:needed limit for x-derivative}
\end{eqnarray}
and for every $\left(x,y\right)\in\MM_{\sim}$ with $1/\p\leq x\leq1$,
\[
\lim_{t\to\infty}\left(\frac{\partial K_{N_{t},-y/\p+\p/\p+2}}{\partial x}\left(Sz_{y}(x)\right)\right)\cdot\left(\frac{\partial K_{N_{t},y}}{\partial x}\left(z_{y}(x)\right)\right)\in\left(\frac{1.6}{\p},\frac{1.7}{\p}\right).
\]
We will separate the proof for three cases: We assume that $(x,y)\in R_{1}\cup R_{3}$,
equivalently $0\leq x\leq1/\p$, the proof when $(x,y)\in R_{2}$
is similar and just involves changing the appearance of $-y/\p$ by
$-y/\p+\p/(\p+2)$. 

\textbf{Case 1: }$(x,y)\in U\cap\mathfrak{Z}^{-1}U$. In this case
\[
\mathfrak{z}_{N_{t}}(x,y)=H_{N_{t}}\circ S\circ H_{N_{t}}^{-1}(x),
\]
and the conclusion is true by Lemma \ref{lem: Differentiablility of g}. 

\textbf{Case 2:} $(x,y)\in U\cap\mathfrak{Z}^{-1}U^{c}$. Firstly
since $(x,y)\in U$ then $K_{n,y}(x)=H_{n}(x)$. In addition, because
$\mathfrak{Z}(x,y)\notin U$ there exists $\mathbf{J\in\mathbb{N}}$
such that 
\[
\left[w_{2},...,w_{{\bf J}+1}\right]_{1}^{{\bf J}}=\left[{\bf w}({\bf J})\right]_{1}^{{\bf J}}\ \text{or }\left[\tilde{{\bf w}}({\bf J})\right]_{1}^{{\bf J}},
\]
and consequently $w_{n}\neq1$ for all $2\leq l<\mathbf{J}-2$. This
shows that $\left.K_{n,y}\right|_{C_{[w]_{1}^{n}}}=h_{n}\circ\cdots h_{{\bf J}-2}\circ h_{1}$
and writing $z_{y}(x)\in\TT$ for the point such that $K_{N_{t},y}\left(z_{y}(x)\right)=x$,
\[
\frac{\partial K_{N_{t},y}}{\partial x}\left(z_{y}(x)\right)=\frac{\partial h_{1}}{\partial x}\left(z_{y}(x)\right)\cdot\prod_{k={\bf J}-2}^{N_{t}}\frac{\partial h_{l}}{\partial x}\left(H_{l-1}\left(z_{y}(x)\right)\right).
\]
 By the definition of the construction
\[
\frac{\partial K_{N_{t},-y/\p}}{\partial x}\left(Sz_{y}(x)\right)=\prod_{l={\bf J}}^{N_{t}}\frac{\partial h_{l,-y/\p}}{\partial x}\left(K_{l-1,-y/\p}(Sz_{y}(x))\right),
\]
and, here $\lambda_{t\left({\bf J}\right)}=\lambda_{k}$ if $N_{k-1}<{\bf J}\leq N_{k}$
or $1$ otherwise, 
\begin{eqnarray*}
\left(\frac{\partial K_{N_{t},-y/\p}}{\partial x}\left(Sz_{y}(x)\right)\right)\cdot\left(\frac{\partial K_{N_{t},y}}{\partial x}\left(z_{y}(x)\right)\right)^{-1} & = & \lambda_{t\left({\bf J}\right)}^{\pm6}\left(\frac{\partial h_{1}}{\partial x}\left(z_{y}(x)\right)\right)^{-1}{\bf \cdot I},
\end{eqnarray*}
Where 
\[
{\bf I}:=\left(\prod_{l={\bf J}}^{N_{t}}\frac{\partial h_{l,-y/\p}}{\partial x}\left(K_{l-1,-y/\p}\left(Sz_{y}(x)\right)\right)\right)\cdot\left(\prod_{l={\bf J}}^{N_{t}}\frac{\partial h_{l,y}}{\partial x}\left(H_{l}\left(z_{y}(x)\right)\right)\right)^{-1}.
\]

As in the proof of Lemma \ref{lem: Differentiablility of g}, assuming
that $\ep\equiv0$, one has that for $l\geq{\bf J}+1$, $K_{l-1,-y/\p}\left(Sz_{y}(x)\right)$
is to the right of the point in $1/\p$ proportion in $K_{l-1,-y/\p}\left(C_{[w_{2},...,w_{l+1}]}\right)$
if and only if $K_{l-1,y}\left(z_{y}(x)\right)$ is to the right of
the point in $1/\p$ proportion in $K_{l,-y/\p}\left(C_{[w]_{1}^{l+1}}\right)$.
This means that in the case $\ep=0$, 
\[
\left(\frac{\partial h_{l-1,-y/\p}}{\partial x}\left(K_{l-2,-y/\p}(Sz_{y}(x))\right)\right)\cdot\left(\frac{\partial h_{l,y}}{\partial x}\left(K_{l-1,y}(z_{y}(x))\right)\right)^{-1}=1.
\]
By proceeding with the analysis of the the bad sets as in Lemma \ref{lem: Differentiablility of g}
one proves that 
\[
{\bf I}=\left(\prod_{k=t\left({\bf J}\right)}^{t}\lambda_{k}^{\pm2M_{k-1}}\right)\exp\left(\pm\sum_{k=t\left({\bf J}\right)}^{t}2^{-N_{k}+4}\right),
\]
and thus
\begin{eqnarray*}
\left(\frac{\partial K_{N_{t},-y/\p}}{\partial x}\left(Sz_{y}(x)\right)\right)\cdot\left(\frac{\partial K_{N_{t},y}}{\partial x}\left(z_{y}(x)\right)\right)^{-1} & = & \left[\left(\frac{\partial h_{1}}{\partial x}\left(z_{y}(x))\right)\right)^{-1}\left(\prod_{k=t\left({\bf J}\right)}^{t}\lambda_{k}^{\pm2M_{k-1}}\right)\right.\\
 &  & \left.\cdot\exp\left(\pm\sum_{k=t\left({\bf J}\right)}^{t}2^{-N_{k}+4}\right)\right].
\end{eqnarray*}
 This shows (\ref{eq:needed limit for x-derivative}). In fact, because
\[
\lim_{s\to\infty}\left(\prod_{k=s}^{\infty}\lambda_{k}^{\pm2M_{k-1}}\right)\exp\left(\pm\sum_{k=s}^{\infty}2^{-N_{k}+4}\right)=1,
\]
 the convergence is uniform as $t\to\infty$. 

\textbf{Case 3:} $(x,y)\in U^{c}$. In this case let ${\bf \tilde{J}}\in\NN$
be such that the closest point to $(x,y)$ on the Horizontal segments
of $\partial\MM$ is in $V_{\tilde{{\bf J}}}$. If $\tilde{{\bf J}}=1$
then $\left(Sz_{y}(x),-y/\p\right)\in U$. Otherwise $\left(Sz_{y}(x),-y/\p\right)$
is in $U^{c}$ and the closest point to it on the Horizontal segments
of $\partial\MM$ is in $V_{\tilde{{\bf J}}-1}$. Consequently 
\begin{eqnarray*}
\left(\frac{\partial K_{N_{t},-y/\p}}{\partial x}\left(Sz_{y}(x)\right)\right)\cdot\left(\frac{\partial K_{N_{t},y}}{\partial x}\left(z_{y}(x)\right)\right)^{-1} & = & \left(\prod_{l={\bf \tilde{J}}-1}^{N_{t}}\frac{\partial h_{l,-y/\p}}{\partial x}\left(K_{l-1,-y/\p}\left(Sz_{y}(x)\right)\right)\right)\\
 &  & \cdot\left(\prod_{l={\bf \tilde{J}}}^{N_{t}}\frac{\partial h_{l,y}}{\partial x}\left(K_{l-1,y}\left(z_{y}(x)\right)\right)\right)^{-1}.
\end{eqnarray*}
Similarly as in case 2, one has 
\[
\left(\frac{\partial K_{N_{t},-y/\p}}{\partial x}\left(Sz_{y}(x)\right)\right)\cdot\left(\frac{\partial K_{N_{t},y}}{\partial x}\left(z_{y}(x)\right)\right)^{-1}=\lambda_{t\left({\bf \tilde{J}}\right)}^{\pm2}\left(\prod_{k=t\left(\tilde{{\bf J}}\right)}^{t}\lambda_{k}^{\pm2M_{k-1}}\right)\exp\left(\pm\sum_{k=t\left(\tilde{{\bf J}}\right)}^{t}2^{-N_{k}+4}\right)
\]
and the convergence is uniform. 
\end{proof}

\subsubsection{Proving differentiability in the $y$-direction. }

Again we will prove differentiability in the $y$ direction for $(x,y)\in R_{1}\cup R_{3}$.
The idea of the proof here is as follows. If $(x,y)\in U$ then $K_{n,\tilde{y}}(x)=H_{\ep,n}(x)$
for all $\tilde{y}$ in a neighborhood of $(x,y)$, hence $\frac{\partial K_{n,y}}{\partial y}(\cdot)\equiv0$.
Otherwise, for $(x,y)\in\MM\backslash U$, $K_{{\bf j}(x,y)-2,\tilde{y}}(x)=x$
and the first (major) change between $K_{n,y}(x)$ and $K_{n,\tilde{y}}(x)$
appears at time $n={\bf j}(x,y)$. We will show that for our construction
the $y$ derivative of $K_{n,y}(x)$ can be bounded above by a (bounded)
constant times $\frac{\partial K_{{\bf j}(x,y),y}(x)}{\partial y}$,
the uniform convergence of $\partial\mathfrak{z}_{n}/\partial y$
will follow from the chain rule and simple arithmetic. 

The following notation will be used in this subsection. Usually we
will consider $x\in[0,1/\p${]} and work constantly with a fixed $w\in\Sigma_{{\bf A}}$
such that $x\in C_{[w]_{1}^{n}}$ for all $n\in\NN$. If that is the
case we will write $\left[\underline{x}_{n},\bar{x}_{n}\right)$ to
denote $C_{[w]_{1}^{n}}$. 

For $-\frac{\p}{\p+2}\leq y\leq\frac{\p^{2}}{\p+2}$ and $n\geq{\bf N}(y)$,
let ${\rm BS}(n,w,y)\subset C_{[w]_{1}^{n}}$ to be the bad set as
in the proof of Lemma \ref{lem: Differentiablility of g} with $h_{n}'\circ H_{n}$
replaced by $\left(\frac{\partial r_{n,y}}{\partial x}\right)\circ K_{n-1,y}$. 

For an $w\in\Sigma_{{\bf A}}$ we denote by $w_{1}^{n}=w_{1}w_{2}\cdots w_{n}$
the finite word derived by $w$ up to time $n$. Given a finite word
$w_{1}^{n}$, $[w_{1}^{n}]$ denotes the $n$-periodic word defined
by $w_{1}^{n}$. Finally given two words $w$ and $\tilde{w}$ (in
which case $w$ is a finite word), the word $w\tilde{w}$ denotes
the concatenation of $w$ and $\tilde{w}$. 

Recall the definition of $K_{{\bf j}(x,y),y}(x)=\rr_{{\bf j}(x,y)}\left(x,y\right)$
which is defined by 
\[
\rr_{{\bf j}(x,y)}(x,y):=\left(H_{{\bf j}(x,y)}(x)-x\right)\mathscr{P}(x,y)+x,
\]
where 
\[
\mathscr{P}(x,y)=\left[3\p^{20}\left({\bf u}(x,y)\right)^{2}-2\p^{30}{\bf u}(x,y)^{3}\right].
\]
 In the following proof if ${\bf j}(x,y)={\bf J}$ we will need a
different definition of the bad set for $n={\bf J}$. Let 
\[
\mathbb{BS}\left({\bf J}\right):=\begin{cases}
{\rm BS}\left({\bf J}-2,{\bf w}\left({\bf J}\right)\right), & {\bf J}\ odd\\
{\rm BS}\left({\bf J}-1,{\bf w}\left({\bf J}\right)\right)\cup{\rm BS}\left({\bf J},{\bf w}\left({\bf J}\right)\right) & {\bf J}\ even
\end{cases}
\]
if $(x,y)\in R_{1}\cup R_{3}$ (For $(x,y)\in R_{2}$ change the odd
to even and even to odd). 
\begin{lem}
\label{lem: induction down when new endpoint is similar to old endpopint}Assume
that $M_{t}\leq{\bf j}\leq N_{t+1}$, and $x\in C_{[w]_{1}^{N_{t+1}}}$.
If $x\notin\mathbb{BS}\left({\bf J}\right)$, for every ${\bf j<}n\leq N_{t+1}$,
there exists $0\leq\beta_{n}(x)\leq\p$ such that for every $(x,y)$
with ${\bf j}\left(x,y\right)={\bf j}$, 
\[
K_{n,y}(x)=K_{n,y}\left(\underline{x}_{n}\right)+\beta_{n}(x)l_{n}\left(y,w\right),
\]
in addition $\beta_{N_{t+1}}(x)$ is continuous in $x$. 
\end{lem}
\begin{proof}
Let $(x,y)\in\left(\MM\backslash U\right)\cup\left(R_{1}\cup R_{3}\right)$
so that ${\bf j}(x,y)={\bf j}$ (the case $(x,y)\in R_{2}$ is similar).
The proof is by induction on $n$. Since $x\notin\mathbb{BS}\left({\bf J}\right)$,
$\underline{x}_{{\bf J}+1}\in\left\{ \underline{x}_{{\bf J}},\underline{x}_{{\bf J}}+\p^{-1}\left(\bar{x}_{{\bf J}}-\underline{x}_{{\bf J}}\right)\right\} $
and 
\[
{\bf w}\left({\bf J}\right)=\begin{cases}
32\cdots32132, & {\bf J}\ \text{odd}\\
32\cdots3211, & {\bf J}\ \text{even}
\end{cases},
\]
it follows that if ${\bf J}$ is even then by property (3) of $\psi_{t}$,
\begin{eqnarray*}
H_{{\bf J}}\left(x\right)-H_{{\bf J}}\left(\underline{x}_{{\bf J}+1}(w)\right) & = & \begin{cases}
\left(\p\psi_{t+1}(1/\p)\right)^{2}\left(x-\underline{x}_{{\bf {\bf J}}+1}\right), & w_{{\bf {\bf J}}+1}=1\\
\left(\p\psi_{t+1}(1/\p)\right)\left(\p^{2}\left(1-\psi_{t+1}\left(1/\p\right)\right)\right)\left(x-\underline{x}_{{\bf J}+1}\right), & w_{{\bf N+1}}=3
\end{cases}\\
 & = & \frac{P_{{\bf J}-1}\left(w_{{\bf J}-1},w_{{\bf J}}\right)}{{\rm Q}\left(w_{{\bf J}},w_{{\bf J}+1}\right)}\frac{P_{{\bf J}}\left(w_{{\bf J}},w_{{\bf J}+1}\right)}{{\rm Q}\left(w_{{\bf J}},w_{{\bf J}+1}\right)}\left(x-\underline{x}_{{\bf J}+1}\right)\\
 & := & b\left(x-\underline{x}_{{\bf J+1}}\right).
\end{eqnarray*}
and if ${\bf J}$ is odd then 
\begin{eqnarray*}
H_{{\bf J}}\left(x\right)-H_{{\bf J}}\left(\underline{x}_{{\bf J}+1}\right) & = & \p^{2}\left(1-\psi_{t+1}\left(1/\p\right)\right)\left(x-\underline{x}_{{\bf J}+1}\right)\\
 & = & \frac{P_{{\bf J}-2}\left(w_{{\bf J}-2},w_{{\bf J}-1}\right)}{{\rm Q}\left(w_{{\bf J}-2},w_{{\bf J}-1}\right)}\left(x-\underline{x}_{{\bf J}+1}\right)\\
 & := & b\left(x-\underline{x}_{{\bf J+1}}\right).
\end{eqnarray*}
It then follows that 
\begin{eqnarray*}
K_{{\bf J},y}(x)-K_{{\bf J},y}\left(\underline{x}_{{\bf J}+1}\right): & = & \rr_{{\bf J}}\left(x,y\right)-\rr_{{\bf J}}\left(\underline{x}_{{\bf J}+1},y\right)\\
 & = & \left(x-\underline{x}_{{\bf J}+1}\right)\left[\left(b-1\right)\mathscr{P}(x,y)+1\right],
\end{eqnarray*}
and 
\begin{eqnarray*}
\mathfrak{l}_{{\bf j+1}}(y,w) & := & K_{{\bf j},y}\left(\bar{x}_{{\bf J}+1}\right)-K_{{\bf j},y}\left(\underline{x}_{{\bf j+1}}\right)\\
 & = & \left(\bar{x}_{{\bf j}+1}-\underline{x}_{{\bf J}+1}\right)\left[\left(b-1\right)\mathscr{P}(x,y)+1\right].
\end{eqnarray*}
This implies that 
\begin{equation}
\frac{K_{{\bf J},y}\left(x\right)-K_{{\bf J},y}\left(\underline{x}_{{\bf J}+1}\right)}{\mathfrak{l}_{{\bf J}+1}(y,w)}=\frac{x-\underline{x}_{{\bf J}+1}}{\bar{x}_{{\bf J}+1}-\underline{x}_{{\bf J+1}}}.\label{eq: equality of the bad set}
\end{equation}
Therefore 
\[
K_{{\bf J}+1,y}(x)=K_{{\bf J+1},y}\left(\underline{x}_{{\bf J}+1}\right)+\underset{:=\beta_{{\bf j}}(x)}{\underbrace{\psi_{t+1}\left(\frac{x-\underline{x}_{{\bf j+1}}}{\bar{x}_{{\bf J}+1}-\underline{x}_{{\bf j+1}}}\right)}}\mathfrak{l}_{{\bf J}+1}\left(y,w\right)
\]
and the base of induction is proved. 

For the inductive step notice that if the conclusion of the Lemma
is true for $n\in\NN$, then 
\[
\frac{K_{n,y}(x)-K_{n,y}\left(\underline{x}_{n+1}\right)}{\mathfrak{l}_{n+1}(y,w)}=\frac{\beta_{n}(x)-\beta_{n}\left(\underline{x}_{n+1}\right)}{\beta_{n}\left(\bar{x}_{n+1}\right)-\beta_{n}\left(\underline{x}_{n+1}\right)}
\]
does not depend on $y.$ The conclusion then follows for $n+1$ with
\[
\beta_{n+1}(x):=\psi_{t+1}\left(\frac{\beta_{n}(x)-\beta_{n}\left(\underline{x}_{n+1}\right)}{\beta_{n}\left(\bar{x}_{n+1}\right)-\beta_{n}\left(\underline{x}_{n+1}\right)}\right)
\]
and the continuity of $\beta_{n+1}$ follows from the continuity of
$\beta_{n}$ and $\psi_{t+1}$. 
\end{proof}
The last lemma shows the importance of knowing how $\frac{\partial\mathfrak{l}_{n}}{\partial y}$
decays when ${\bf N}(y)<n\leq N_{t+1}$. We will now show that it
is exponential in $n$. 
\begin{lem}
\label{lem: derivatives of scaling functions}Let $M_{t}\leq{\bf j}<n\leq N_{t+1}$,
a $w\in\Sigma_{{\bf A}}$ with $w_{1}^{{\bf j}}={\bf w(j)}$ and $-\frac{\varphi}{\varphi+2}\leq y\leq\frac{\varphi^{2}}{\varphi+2}$,
then 
\[
\left|\frac{\partial\mathfrak{l}_{n}}{\partial y}(y,w)\right|\leq(1.6)^{{\bf j}-n}\left|\frac{\partial l_{{\bf j}}}{\partial y}(y,w)\right|
\]
and 
\[
\left|\frac{\partial\mathfrak{l}_{{\bf j}+1}}{\partial y}(y,w)\right|\lesssim(1.6)^{-{\bf j}}.
\]
\end{lem}
\begin{proof}
We assume $\left(\underline{x}_{N_{t+1}},y\right)\notin\partial U$,
the proof for the case $\left(\underline{x}_{N_{t+1}},y\right)\in\partial U$
is similar. In this case for small $|h|$, $\left(\underline{x}_{N_{t+1}},y+h\right)\in U^{c}$. 

Since $\bar{x}_{N_{t+1}}(w)$ is not in the bad set $\mathbb{BS}\left({\bf j}(x,y)\right)$,
it follows from (\ref{eq: equality of the bad set}) that for small
$|h|$, 
\begin{eqnarray*}
\frac{m_{{\rm \TT}}\left(K_{{\bf j},y+h}\left(C_{[w]_{1}^{N_{t+1}}}\right)\right)}{m_{\TT}\left(K_{{\bf j},y+h}\left(C_{[w]_{1}^{{\bf j}+1}}\right)\right)}: & = & \frac{K_{{\bf j},y+h}\left(\bar{x}_{N_{t+1}}\right)-K_{{\bf j},y+h}\left(\underline{x}_{N_{t+1}}\right)}{K_{{\bf j},y+h}\left(\bar{x}_{{\bf j}+1}\right)-K_{{\bf j},y+h}\left(\underline{x}_{{\bf j}+1}\right)}\\
 & = & \frac{\bar{x}_{N_{t+1}}-\underline{x}_{N_{t+1}}}{\bar{x}_{{\bf j}+1}-\underline{x}_{{\bf j}+1}}\\
 & = & \frac{m_{{\rm \TT}}\left(C_{[w]_{1}^{N_{t}+1}}\right)}{m_{\TT}\left(C_{[w]_{1}^{{\bf j}+1}}\right)}.
\end{eqnarray*}
It then follows by definition of $h_{n,y}$ for $n>{\bf j}$ that
for $|h|$ small, 
\[
\mathfrak{l}_{N_{t+1}}(y+h,w)=l_{{\bf j}+1}(y+h,w)\prod_{k={\bf j}+1}^{N_{t+1}}P_{k}\left(w_{k},w_{k+1}\right),
\]
hence 
\[
\frac{\mathfrak{l}_{N_{t+1}}(y+h,w)}{\mathfrak{l}_{N_{t+1}}(y,w)}=\frac{\mathfrak{l}_{{\bf j}+1}(y+h,w)}{\mathfrak{l}_{{\bf j}+1}(y,w)}.
\]
This yields that 
\[
\mathfrak{l}_{N_{t+1}}(y+h,w)-\mathfrak{l}_{N_{t+1}}(y,w)=\frac{\mathfrak{l}_{N_{t+1}}(y,w)}{\mathfrak{l}_{{\bf j}(x,y)+1}(y,w)}\left[\mathfrak{l}_{{\bf j}(x,y)+1}(y+h,w)-\mathfrak{l}_{{\bf j}(x,y)+1}(y,w)\right],
\]
dividing by $h$ and taking limit $h\to0$ we get 
\[
\left|\frac{\partial\mathfrak{l}_{N_{t+1}}(y,w)}{\partial y}\right|=\frac{\mathfrak{l}_{N_{t+1}}(y,w)}{\mathfrak{l}_{{\bf j}+1}(y,w)}\left|\frac{\partial\mathfrak{l}_{{\bf j}+1}(y,w)}{\partial y}\right|\leq(1.6)^{{\bf j}-N_{t+1}}\left|\frac{\partial\mathfrak{l}_{{\bf j}+1}(y,w)}{\partial y}\right|.
\]
The last inequality follows from, for all $N_{t}\leq{\bf j}<n\leq N_{t+1}$,
\[
\prod_{k={\bf j}}^{n-1}P_{k}\left(w_{k},w_{k+1}\right)\leq\left(\frac{\lambda_{t+1}\p}{1+\lambda_{t+1}\p}\right)^{{\bf j}-n}\leq\left(\frac{\lambda_{t+1}}{\p}\right)^{{\bf j}-n}<(1.6)^{{\bf j}-n}.
\]

For the proof of the second part notice that for $x\in\left\{ \underline{x}_{{\bf j}},\bar{x}_{{\bf j}}\right\} $,
\begin{eqnarray}
\left|\frac{\partial K{}_{{\bf j},y}}{\partial y}(x)\right| & = & \left|\frac{\partial\rr{}_{{\bf j}}}{\partial y}(x,y)\right|\label{eq: decay of r_j}\\
 & \leq & \left|\frac{\mathtt{\partial}\mathscr{P}}{\partial y}(x,y)\right|\left|H_{\ep,{\bf j}}(x)-x\right|\nonumber \\
 & \leq & \frac{3}{2}\p^{10}\left|m\left(C_{[w]_{1}^{{\bf j}-2}}\right)\right|\leq\frac{3}{2}\p^{12}\p^{-{\bf j}}\leq\frac{1}{2}(1.62)^{-{\bf j}},\nonumber 
\end{eqnarray}
for all large ${\bf j}$. Thus (recall $\mathfrak{l}_{{\bf j}+1}(y,w)=\rr{}_{{\bf j}(x,y)}\left(\bar{x}_{{\bf j}+1}\right)-\rr{}_{{\bf j}}\left(\underline{x}_{{\bf j}+1}\right)$)
\begin{eqnarray*}
\left|\frac{\partial\mathfrak{l}_{{\bf j}+1}}{\partial y}(y,w)\right| & \leq & (1.6)^{-{\bf j}},\ \ \text{for all large}\ {\bf j.}
\end{eqnarray*}
\end{proof}
Lemma \ref{lem: induction down when new endpoint is similar to old endpopint}
shows that if $x\in C_{[{\bf w}\left({\bf j}\right)]_{1}^{n}}$ is
not in $x\in\mathbb{BS}\left({\bf j}\right)$ then the $y$- derivative
of $K_{N_{t+1},y}(x)$ (here $t$ is the number such that $N_{t}<{\bf j}<N_{t+1}$)
is controlled by the derivative on a finite collection of points plus
the evolution of the lengths of the intervals. We would like to point
out that there is actually no bad set if $N_{t}<{\bf j}\leq M_{t}$
because then $K_{{\bf j},y}|_{C_{[{\bf w}\left({\bf j}\right)]_{1}^{{\bf j}}}}=id_{\TT}$.
This idea will be reiterated with a slight modification for the derivatives
$\partial K_{n,y}/\partial y$ for $M_{s}<n<N_{s}$ where $M_{s}>{\bf j}(x,y)$. 

For points in the bad set we will apply a correction point procedure
which we call the $x$-delta method. Assume that $x\in\mathbb{BS}\left({\bf j}\right)$.
For $\Delta$-small (so that $(x,y\pm\Delta)\in U^{c})$) there exists
a unique ${\bf x}\left(\Delta\right)$ such that 
\begin{equation}
\frac{K_{{\bf j}(x,y),y+\Delta}({\bf x}\left(\Delta\right))-K_{{\bf j}(x,y),y+\Delta}\left(\underline{x}_{{\bf j}(x,y)+1}\right)}{\mathfrak{l}_{{\bf j}(x,y)+1}(y+\Delta,w)}=\frac{K_{{\bf j}(x,y),y}(x)-K_{{\bf j}(x,y),y}\left(\underline{x}_{{\bf j}(x,y)+1}\right)}{\mathfrak{l}_{{\bf j}(x,y)+1}(y,w)}.\label{eq: the x-delta equality}
\end{equation}
We will use Lemma \ref{lem: derivatives of scaling functions} to
obtain a first order approximation for ${\bf x}\left(\Delta\right)$
when $\Delta$ is small. 

In the next Lemma, if $M_{t}<{\bf j}+1<N_{t+1}$ $\beta_{{\bf j}+1}(x):=\psi_{t+1}\left(\frac{K_{{\bf j},y}(x)-K_{{\bf j},y}\left(\underline{x}_{{\bf j}+1}\right)}{\mathfrak{l}_{{\bf j}+1}(y,w)}\right)$
and for $M_{t}<{\bf j}+1<n\leq N_{t+1}$ 
\[
\beta_{n+1}(x):=\psi_{t+1}\left(\frac{\beta_{n}(x)-\beta_{n}\left(\underline{x}_{n+1}\right)}{\beta_{n}\left(\bar{x}_{n+1}\right)-\beta_{n}\left(\underline{x}_{n+1}\right)}\right).
\]

\begin{lem}
\label{lem:first use of the x-delta method}Assume that $M_{t}\leq{\bf j}\leq N_{t+1}$,
and $x\in C_{[w]_{1}^{N_{t+1}}}$ with $w_{1}^{{\bf j}}={\bf w}\left({\bf j}\right)$.
The following holds:

(i) For every $y$ so that ${\bf j}\left(x,y\right)={\bf j}$ and
$\Delta$ so that $(x,y\pm\Delta)\in\MM\backslash U$, 
\[
K_{N_{t+1},y+\Delta}\left({\bf x}\left(\Delta\right)\right)=K_{N_{t+1},y+\Delta}\left(\underline{x}_{N_{t+1}}\right)+\beta_{N_{t+1}}(x)l_{N_{t+1}}\left(y+\Delta,w\right),
\]

(ii) $\left|{\bf x}\left(\Delta\right)-x\right|\leq4(1.6)^{-N_{t+1}}\Delta+o\left(\Delta\right)$
as $\Delta\to0$.
\end{lem}
\begin{proof}
(i) This is the same as the proof of Lemma \ref{lem: induction down when new endpoint is similar to old endpopint}
by using (\ref{eq: the x-delta equality}) as the starting point. 

(ii) If $x\notin\mathbb{BS}\left({\bf j}\right)\cap C_{[w]_{1}^{N_{t}}}$
then by Lemma \ref{lem: induction down when new endpoint is similar to old endpopint},
${\bf x}\left(\Delta\right)=x$. Since $\underline{x}_{N_{t+1}}\notin\mathbb{BS}\left({\bf j}\right)$,
\begin{eqnarray*}
\frac{K_{{\bf j},y+\Delta}\left(\underline{x}_{N_{t+1}}\right)-K_{{\bf j},y+\Delta}\left(\underline{x}_{{\bf j}(x,y)+1}\right)}{\mathfrak{l}_{{\bf j}+1}(y+\Delta,w)} & = & \frac{\underline{x}_{N_{t+1}}-\underline{x}_{{\bf j}(x,y)+1}}{\overline{x}_{{\bf j}+1}-\underline{x}_{{\bf j}+1}}\\
 & = & \frac{K_{{\bf j},y}\left(\underline{x}_{N_{t+1}}\right)-K_{{\bf j},y}\left(\underline{x}_{{\bf j}+1}\right)}{\mathfrak{l}_{{\bf j}+1}(y,w)}.
\end{eqnarray*}
Therefore by adding and subtracting $K_{{\bf j},y+\Delta}\left(\underline{x}_{N_{t+1}}\right)/\mathfrak{l}_{{\bf j}+1}(y+\Delta,w)$
on the right hand side and $K_{{\bf j},y}\left(\underline{x}_{N_{t+1}}\right)/\mathfrak{l}_{{\bf j}+1}(y,w)$
in the left hand side of equation \ref{eq: the x-delta equality},
if follows that equation (\ref{eq: the x-delta equality}) is equivalent
to 
\begin{equation}
K_{{\bf j},y+\Delta}({\bf x}\left(\Delta\right))-K_{{\bf j},y+\Delta}\left(\underline{x}_{N_{t+1}}\right)=\frac{\mathfrak{l}_{{\bf j}+1}(y+\Delta,w)}{\mathfrak{l}_{{\bf j}+1}(y,w)}\left(K_{{\bf j},y}(x)-K_{{\bf j},y}\left(\underline{x}_{N_{t+1}}\right)\right).\label{eq: modified xbeta equality}
\end{equation}
For the ease of notation we will write $\underline{X}:=\underline{x}_{N_{t+1}}$
and $H_{{\bf j}}(z):=H_{\ep.{\bf j}}\left(z\right)$. Since by Lemma
\ref{lem: derivatives of scaling functions}, 
\[
\mathfrak{l}_{{\bf j}+1}(y+\Delta,w)=\mathfrak{l}_{{\bf j}+1}(y,w)\pm(1.6)^{-{\bf J}}\Delta+o(\Delta)
\]
we have
\[
\frac{\mathfrak{l}_{{\bf j}+1}(y+\Delta,w)}{\mathfrak{l}_{{\bf j}+1}(y,w)}\left(K_{{\bf j},y}(x)-K_{{\bf j},y}\left(\underline{x}\right)\right)=\left[1\pm\frac{(1.6)^{-{\bf j}}\Delta}{\mathfrak{l}_{{\bf j}+1}(y,w)}\right]\left(K_{{\bf j},y}(x)-K_{{\bf j},y}\left(\underline{X}\right)\right)+o\left(\Delta\right).
\]
In addition for all $\left|\Delta\right|$ small , 
\[
K_{{\bf j},y+\Delta}({\bf x}\left(\Delta\right))-K_{{\bf j},y+\Delta}\left(\underline{X}\right)=\left(H_{{\bf j}}({\bf x}(\Delta))-H_{{\bf {\bf j}}}\left(\underline{X}\right)-\left({\bf x}(\Delta)-\underline{X}\right)\right)\mathscr{P}(x,y+\Delta)+\left({\bf x}(\Delta)-\underline{X}\right)
\]
By Taylor expansion 
\[
\mathscr{P}(x,y+\Delta)=\mathscr{P}(x,y)+\Delta\frac{\partial\mathscr{P}}{\partial y}(x,y)+o\left(\Delta\right).
\]
Using this one can show that (\ref{eq: modified xbeta equality})
yields,
\[
\left({\bf x}\left(\Delta\right)-x\right)+\left(H_{{\bf j}}({\bf x}(\Delta))-H_{{\bf j}}\left(x\right)-\left({\bf x}(\Delta)-x\right)\right)\mathscr{P}(x,y)=\Delta\left({\bf I}+{\bf II}\right)+o\left(\Delta\right)
\]
where 
\[
\left|{\bf I}\right|:=\left|\frac{(1.6)^{-{\bf j}}}{\mathfrak{l}_{{\bf j}+1}(y,w)}\left(K_{{\bf j},y}(x)-K_{{\bf j},y}\left(\underline{X}\right)\right)\right|\leq\left(1.6\right)^{-{\bf j}}\p^{-N_{t+1}+{\bf j}}
\]
and 
\begin{eqnarray*}
{\bf \left|II\right|}: & = & \left|\left(H_{{\bf j}}({\bf x}(\Delta))-H_{{\bf j}}\left(\underline{X}\right)-\left({\bf x}(\Delta)-\underline{X}\right)\right)\frac{\partial\mathscr{P}}{\partial y}(x,y)\right|\\
 & \leq & \left|\lambda_{t+1}^{4}-1\right|\max_{(x,y)\in U^{c}}\left|\frac{\partial\mathscr{P}}{\partial y}(x,y)\right|\left|{\bf x}(\Delta)-\underline{X}\right|\\
 & \leq & \left|\lambda_{t+1}^{4}-1\right|\left(\frac{3}{2}\p^{10}\right)\p^{-N_{t+1}}\\
 & \leq & \p^{-N_{t+1}}.
\end{eqnarray*}
For both inequalities we used the fact that 
\[
\max\left\{ \left|{\bf x}\left(\Delta\right)-\underline{x}_{N_{t+1}}\right|,\left|x-\underline{x}_{N_{t+1}}\right|\right\} \leq\p^{-N_{t+1}}.
\]
Since 
\[
\left|\left(H_{{\bf j}}({\bf x}(\Delta))-H_{{\bf j}}\left(x\right)-\left({\bf x}(\Delta)-x\right)\right)\mathscr{P}(x,y)\right|\leq\left|\lambda_{t+1}^{4}-1\right|\left|{\bf x}\left(\Delta\right)-x\right|
\]
we get by the triangle inequality that 
\begin{eqnarray*}
\left|{\bf x}\left(\Delta\right)-x\right|\left(1-\left|\lambda_{t+1}^{4}-1\right|\right) & \leq & \left|{\bf x}\left(\Delta\right)-x+\left(H_{{\bf j}}({\bf x}(\Delta))-H_{{\bf j}}\left(x\right)-\left({\bf x}(\Delta)-x\right)\right)\mathscr{P}(x,y)\right|\\
 & \leq & \Delta\left(\left|{\bf I}\right|+\left|{\bf II}\right|\right)+o\left(\Delta\right)\\
 & \leq & 2\Delta(1.6)^{-N_{t+1}}+o\left(\Delta\right).
\end{eqnarray*}
As $1-\left|\lambda_{t+1}-1\right|\geq\frac{1}{2}$ the conclusion
of part (ii) follows. 
\end{proof}
From now on we work under the assumption that $(1.62)/\lambda_{1}^{2}\geq1.6$.
As $\lambda_{1}$ can be made arbitrarily small this is compatible
with the inductive procedure. 
\begin{cor}
\label{cor: The decay after the first break } For every $(x,y)\in R_{1}\cup R_{3}$,
if $N_{t}<{\bf j}(x,y)\leq N_{t+1}$ then 
\[
(i)\ \ \ \ \left|\frac{\partial K_{N_{t+1},y}(x)}{\partial y}\right|\leq6(1.6)^{-{\bf j}(x,y)}.
\]
In addition if $(x,y)\in\partial U^{c}\cup\partial\MM$ then 
\[
\frac{\partial K_{N_{t+1},y}(x)}{\partial y}=0.
\]

(ii) Assume that $\left\{ N_{k},M_{k-1},\epsilon_{k},\lambda_{k}\right\} _{k=1}^{s}$
are chosen, there exists a choice of $M_{s},\lambda_{s+1},N_{s+1}\ \text{and }\epsilon_{s+1}$
(compatible with the inductive procedure) such that 
\[
\left|\frac{\partial\mathfrak{l}_{M_{t+1}}}{\partial y}(y,w)\right|\le3(1.6)^{-N_{t+1}}.
\]
\end{cor}
\begin{proof}
(i) First we claim that for all $w\in\Sigma_{{\bf A}}$ such that
$[w]_{1}^{{\bf j}(x,y)}={\bf w}\left({\bf j}(x,y)\right)$, 
\begin{equation}
\frac{\partial K_{N_{t+1},y}\left(\underline{x}_{N_{t+1}}\right)}{\partial y}=\pm4(1.6)^{-{\bf j}(x,y)}.\label{eq: derivative at special points}
\end{equation}
This is true because of the following argument. For each ${\bf j}(x,y)<n\leq N_{t+1}$,
either $\underline{x}_{n-1}=\underline{x}_{n}$ and then 
\[
K_{n,y}\left(\underline{x}_{n}\right)=K_{n-1,y}\left(\underline{x}_{n-1}\right)
\]
or $\underline{x}_{n}=\underline{x}_{n-1}+\p^{-1}\left(\bar{x}_{n-1}-\underline{x}_{n-1}\right)$
and then 
\[
K_{n,y}\left(\underline{x}_{n}\right)=K_{n-1,y}\left(\underline{x}_{n}\right)=K_{n-1,y}\left(\underline{x}_{n-1}\right)+\frac{\lambda_{t+1}\p}{1+\lambda_{t+1}\p}\mathfrak{l}_{n}(y,w).
\]
This equality remains true in a neighborhood of $y$. Therefore for
all ${\bf j}(x,y)<n\leq N_{t+1}$, 
\begin{eqnarray*}
\left|\frac{\partial K_{n,y}\left(\underline{x}_{n}(w)\right)}{\partial y}\right| & \leq & \left|\frac{\partial K_{n-1,y}\left(\underline{x}_{n-1}(w)\right)}{\partial y}\right|+\frac{2}{3}\max_{[w]_{1}^{N_{t+1}}\subset[{\bf w}\left({\bf j}(x,y)\right)]}\left|\frac{\partial\mathfrak{l}_{n}}{\partial y}(y,w)\right|\\
 & \overset{\text{Lem.}\ \ref{lem: derivatives of scaling functions}}{\leq} & \left|\frac{\partial K_{n-1,y}\left(\underline{x}_{n-1}(w)\right)}{\partial y}\right|+\frac{2}{3}(1.6)^{-n}
\end{eqnarray*}
and so 
\begin{eqnarray*}
\left|\frac{\partial K_{N_{t+1},y}\left(\underline{x}_{N_{t+1}}(w)\right)}{\partial y}\right| & \leq & \left|\frac{\partial K_{{\bf j}(x,y),y}\left(\underline{x}_{{\bf j}(x,y)}\right)}{\partial y}\right|+\frac{2}{3}\sum_{n={\bf j}(x,y)+1}^{N_{t+1}}(1.6)^{-n}\\
 & \overset{\eqref{eq: decay of r_j}}{\leq} & 4(1.6)^{-{\bf j}(x,y)}.\ \ \ \Square_{\eqref{eq: derivative at special points}}
\end{eqnarray*}

Now for a general $0\leq x\leq1/\p$, 
\begin{eqnarray*}
\frac{\partial K_{N_{t+1},y}(x)}{\partial y} & = & \lim_{\Delta\to0}\frac{K_{N_{t+1},y+\Delta}(x)-K_{N_{t+1},y}(x)}{\Delta}\\
 & = & \lim_{\Delta\to0}\frac{K_{N_{t+1},y+\Delta}\left({\bf x}\left(\Delta\right)\right)-K_{N_{t+1},y}(x)}{\Delta}+\lim_{\Delta\to0}\frac{K_{N_{t+1},y+\Delta}\left({\bf x}\left(\Delta\right)\right)-K_{N_{t+1},y+\Delta}(x)}{\Delta}.
\end{eqnarray*}
As
\[
\left|\frac{\partial K_{N_{t+1},y+\Delta}(x)}{\partial x}\right|\leq\lambda_{t+1}^{2\left(N_{t+1}-{\bf j}(x,y)\right)},
\]
it follows that 
\begin{eqnarray}
\lim_{\Delta\to0}\left|\frac{K_{N_{t+1},y+\Delta}\left({\bf x}\left(\Delta\right)\right)-K_{N_{t+1},y+\Delta}(x)}{\Delta}\right| & \leq & \lambda_{t+1}^{2N_{t+1}}\lim_{\Delta\to0}\left|\frac{{\bf x}\left(\Delta\right)-x}{\Delta}\right|\label{eq: xdelta -x derivative relation}\\
 & \overset{_{{\rm Lem\ \ref{lem:first use of the x-delta method}}}}{\leq} & \left(\frac{1.62}{\lambda_{1}^{2}}\right)^{-N_{t+1}}.\nonumber 
\end{eqnarray}
By Lemma \ref{lem:first use of the x-delta method}.(i) if $x\in C_{[w]_{1}^{N_{t+1}}}$
then,
\begin{eqnarray*}
\lim_{\Delta\to0}\left|\frac{K_{N_{t+1},y+\Delta}\left({\bf x}\left(\Delta\right)\right)-K_{N_{t+1},y}(x)}{\Delta}\right| & \leq & \lim_{\Delta\to0}\left|\frac{K_{N_{t+1},y+\Delta}\left(\underline{x}_{N_{t+1}}(w)\right)-K_{N_{t+1},y}\left(\underline{x}_{N_{t+1}}(w)\right)}{\Delta}\right|\\
 &  & +\beta_{n}(x)\lim_{\Delta\to0}\left|\frac{\mathfrak{l}_{N_{t+1}}(y+\Delta,w)-\mathfrak{l}_{N_{t+1}}(y,w)}{\Delta}\right|\\
 & \leq & \left|\frac{\partial K_{N_{t+1},y}\left(\underline{x}_{N_{t+1}}\right)}{\partial y}\right|+\left|\frac{\partial\mathfrak{l}_{N_{t+1}}}{\partial y}(y,w)\right|\\
 & \leq & 4(1.6)^{-{\bf j}(x,y)}+\left(1.6\right)^{-N_{t+1}}.
\end{eqnarray*}
and the conclusion follows. 

The second of part (i) in the Corollary is true since if $(x,y)\in\partial U^{c}\cup\partial\MM$,
then $\frac{\partial\mathscr{P}}{\partial y}(x,y)=0$. Therefore $\frac{\partial\mathfrak{l}_{{\bf j}(x,y)+1}}{\partial y}(y,w)=0$
and ${\bf x}\left(\Delta\right)=x+o(\Delta)$.

(ii) Let $w\in\Sigma_{{\bf A}}\left(M_{t+1}\right)$. As for all $y$,
$\left\{ K_{N_{t+1},y}\left(\underline{x}_{M_{t+1}}\right),K_{N_{t+1},y}\left(\bar{x}_{M_{t+1}}\right):\ w\in\Sigma_{{\bf A}}\left(M_{t+1}\right)\right\} $
are fixed points for $h_{M_{t+1},y}$ it follows that for all $\Delta$,
\begin{eqnarray*}
\mathfrak{l}_{M_{t+1}}(y+\Delta,w) & = & K_{M_{t+1}-1,y+\Delta}\left(\bar{x}_{M_{t+1}}\right)-K_{M_{t+1}-1,y+\Delta}\left(\underline{x}_{M_{t+1}}\right)\\
 & = & K_{N_{t+1},y+\Delta}\left(\bar{x}_{M_{t+1}}\right)-K_{N_{t+1},y+\Delta}\left(\underline{x}_{M_{t+1}}\right).
\end{eqnarray*}
The last line follows from $h_{n,y+\Delta}=id_{\TT}$ for $N_{t+1}<n<M_{t+1}$.
Writing $\underline{{\bf x}}(\Delta)$ (respectively $\bar{{\bf x}}\left(\Delta\right)$)
for the x-delta point of $\underline{x}_{M_{t+1}}(w)$ (respectively
$\bar{x}_{M_{t+1}}(w)$). By Lemma \ref{lem:first use of the x-delta method},
for $\left|\Delta\right|$ small, 
\[
K_{N_{t+1},y+\Delta}\left(\underline{{\bf x}}\left(\Delta\right)\right)=K_{N_{t+1},y+\Delta}\left(\underline{x}_{N_{t+1}}\right)+\beta_{N_{t+1}}\left(\underline{x}_{M_{t+1}}\right)\mathfrak{l}_{N_{t+1}}(y+\Delta,w),
\]
and 
\[
K_{N_{t+1},y+\Delta}\left(\bar{{\bf x}}\left(\Delta\right)\right)=K_{N_{t+1},y+\Delta}\left(\underline{x}_{N_{t+1}}\right)+\beta_{N_{t+1}}\left(\bar{x}_{M_{t+1}}\right)\mathfrak{l}_{N_{t+1}}(y+\Delta,w).
\]
It then follows that for $|\Delta|$ small, 
\[
\mathfrak{l}_{M_{t+1}}(y+\Delta,w)=\left\{ \beta_{N_{t+1}}\left(\bar{x}_{M_{t+1}}\right)-\beta_{N_{t+1}}\left(\underline{x}_{M_{t+1}}\right)\right\} \mathfrak{l}_{N_{t+1}}(y+\Delta,w)+\overline{I}_{\Delta}+\underline{I}_{\Delta}.
\]
where by (\ref{eq: xdelta -x derivative relation}), 
\[
\left|\overline{I}_{\Delta}\right|:=\left|K_{N_{t+1},y+\Delta}\left(\bar{{\bf x}}\left(\Delta\right)\right)-K_{N_{t+1},y+\Delta}\left(\bar{x}_{M_{t+1}}\right)\right|\leq\left(1.1\right)\Delta(1.59)^{-N_{t+1}}
\]
and 
\[
\left|\underline{I}_{\Delta}\right|:=\left|K_{N_{t+1},y+\Delta}\left(\underline{{\bf x}}\left(\Delta\right)\right)-K_{N_{t+1},y+\Delta}\left(\underline{x}_{M_{t+1}}\right)\right|\leq\left(1.1\right)\Delta(1.59)^{-N_{t+1}}.
\]
It then follows that 
\begin{eqnarray*}
\left|\frac{\partial\mathfrak{l}_{M_{t+1}}}{\partial y}(y,w)\right| & \leq & \underset{\ll1}{\underbrace{\left\{ \beta_{N_{t+1}}\left(\bar{x}_{M_{t+1}}\right)-\beta_{N_{t+1}}\left(\underline{x}_{M_{t+1}}\right)\right\} }}\left|\frac{\partial\mathfrak{l}_{N_{t+1}}}{\partial y}(y,w)\right|+\lim_{\Delta\to0}\frac{\left|\overline{I}_{\Delta}\right|+\left|\underline{I}_{\Delta}\right|}{\Delta}\\
 & \lesssim & 3(1.6)^{-N_{t+1}}\ \text{as}\ M_{t+1}\to\infty.
\end{eqnarray*}
 
\end{proof}
So far we have managed to show to control $\frac{\partial K_{N_{t+1},y}(x)}{\partial y}$
by a constant times the derivative at level ${\bf j}(x,y)$ where
$t=t(y)=\min\left\{ t\in\NN:N_{t+1}\geq{\bf j}(x,y)\right\} $. The
next step is for $s>t(y)$, to obtain a relation between $\frac{\partial K_{s+1,y}(x)}{\partial y}$
and $\frac{\partial K_{N_{s,y}}}{\partial y}$.
\begin{defn}
For $M_{s}<n\leq N_{s+1}$, $x\in C_{[w]_{1}^{N_{s+1}}}$ and $\left|\Delta\right|$
small we define ${\bf x}_{s}\left(\Delta\right)$ to be the unique
point such that 
\[
\frac{K_{M_{s},y+\Delta}\left({\bf x}_{s}\left(\Delta\right)\right)-K_{M_{s},y+\Delta}\left(\underline{x}_{M_{s}+1}\right)}{\mathfrak{l}_{M_{s}+1}\left(y+\Delta,w\right)}=\frac{K_{M_{s},y}\left(x\right)-K_{M_{s},y}\left(\underline{x}_{M_{s}+1}\right)}{\mathfrak{l}_{M_{s}+1}\left(y,w\right)}.
\]
Setting similarly to before for $0\leq x\leq1/\p$ and $y$ such that
${\bf j}(x,y)<M_{s}$, 
\[
\beta_{M_{s}+1}(x):=\psi_{s+1}\left(\frac{K_{M_{s},y}\left(x\right)-K_{M_{s},y}\left(\underline{x}_{M_{s}+1}\right)}{\mathfrak{l}_{M_{s}+1}\left(y,w\right)}\right)
\]
and for $M_{s}+1<n\leq N_{s+1}$, 
\[
\beta_{n}(x):=\psi_{s+1}\left(\frac{\beta_{n-1}(x)-\beta_{n-1}\left(\underline{x}_{n}\right)}{\beta_{n-1}\left(\bar{x}_{n}\right)-\beta_{n-1}\left(\underline{x}_{n}\right)}\right).
\]
\end{defn}
\begin{lem}
\label{lem: derivatives, scaling function long after the breaking point}For
all $(x,y)\in R_{1}\cup R_{3}$ with ${\bf j}(x,y)<N_{s}$ the following
holds: 

1. for every $|\Delta|$ small, 
\[
K_{N_{s+1},y+\Delta}\left({\bf x}_{s}\left(\Delta\right)\right)=K_{N_{s+1},y+\Delta}\left(\underline{x}_{N_{s+1}}\right)+\beta_{N_{s+1}}(x)l_{N_{s+1}}\left(y+\Delta,w\right).
\]

2. (i) For every $M_{s}<n\leq N_{s+1}$, 
\[
\left|\frac{\partial\mathfrak{l}_{n}}{\partial y}(y,w)\right|\leq(1.6)^{M_{s}-n}\left|\frac{\partial\mathfrak{l}_{M_{s}}}{\partial y}(y,w)\right|.
\]

(ii) $\left|\frac{\partial\mathfrak{l}_{M_{s}}}{\partial y}(y,w)\right|\le(1.6)^{-n_{s}}.$ 

(iii) Assume that $\left\{ N_{k},M_{k-1},\epsilon_{k},\lambda_{k}\right\} _{k=1}^{s}$
are chosen, there exists a choice of $M_{s},\lambda_{s+1},N_{s+1}\ \text{and }\epsilon_{s+1}$
(compatible with the inductive procedure) such that 
\[
\left|{\bf x}_{s}\left(\Delta\right)-x\right|\leq\p^{-N_{s+1}}\Delta+o\left(\Delta\right)
\]
 as $\Delta\to0$.
\end{lem}
\begin{proof}
This is done by induction on $s$. The base of induction is the first
$s\in\NN$ such that $N_{s}>{\bf j}(x,y)$. 

1. This is similar to the proof of Lemma \ref{lem: induction down when new endpoint is similar to old endpopint}
and Lemma \ref{lem:first use of the x-delta method}.(i). 

2. (i) Let $w\in\Sigma_{{\bf A}}$. The starting point is that by
the definition of $h_{M_{t},y}$ (as a distribution correcting function),
equation \ref{eq: the important equality of H_M_t} holds for $K_{M_{t},y}$.
Therefore for all $w\in\Sigma_{{\bf A}}$ and $h>0$ small,
\[
m_{{\rm \TT}}\left(K_{M_{s},y+h}\left(C_{[w]_{1}^{N_{s+1}}}\right)\right)=\mathfrak{l}_{M_{s}}(y+h,w)\frac{m_{\TT}\left(C_{[w]_{1}^{N_{s+1}}}\right)}{m_{\TT}\left(C_{[w]_{1}^{M_{s}}}\right)}.
\]
 The rest is similar to the proof of the first part of Lemma \ref{lem: derivatives of scaling functions}
with ${\bf j}(x,y)$ replaced by $M_{s}$. 

2.(ii) Since for all $y$, $\left\{ K_{N_{s},y}\left(\underline{x}_{M_{s}}(w)\right),K_{N_{s},y}\left(\bar{x}_{M_{s}}(w)\right):\ w\in\Sigma_{{\bf A}}\left(M_{s}\right)\right\} $
are fixed points for $h_{M_{s},y}$ it follows that (here $\underline{x}_{M_{s}}=\underline{x}_{M_{s}}(w)$)
\[
\mathfrak{l}_{M_{s}}(y+\Delta,w)=K_{N_{s},y+\Delta}\left(\bar{x}_{M_{s}}\right)-K_{N_{s},y+\Delta}\left(\underline{x}_{M_{s}}\right).
\]
The base of induction is Corollary \ref{cor: The decay after the first break }.(ii). 

The proof of the inductive step is the same as the proof of the base
of induction where we use the induction hypothesis that
\[
\left|\frac{\partial\mathfrak{l}_{N_{s}}}{\partial y}(y,w)\right|\le(1.6)^{M_{s-1}-N_{s}}\left|\frac{\partial\mathfrak{l}_{M_{s-1}}}{\partial y}(y,w)\right|\leq(1.6)^{-N_{s-1}-n_{s}},
\]
and 
\[
\left|{\bf x}_{s}\left(\Delta\right)-x\right|\leq\p^{-N_{s+1}}\Delta+o\left(\Delta\right).
\]
Therefore, 
\[
\left|\overline{I}_{\Delta}(s)\right|:=\left|K_{N_{t+1},y+\Delta}\left(\bar{{\bf x}}_{s}\left(\Delta\right)\right)-K_{N_{t+1},y+\Delta}\left(\bar{x}_{M_{s}}\right)\right|\leq\Delta(1.6)^{-N_{s+1}}
\]
and
\[
\left|\underline{I}_{\Delta}\right|:=\left|K_{N_{t+1},y+\Delta}\left(\underline{{\bf x}}_{s}\left(\Delta\right)\right)-K_{N_{t+1},y+\Delta}\left(\underline{x}_{M_{s}}\right)\right|\leq\Delta(1.6)^{-N_{s+1}}.
\]
It then follows that 
\begin{eqnarray*}
\left|\frac{\partial\mathfrak{l}_{M_{s}}}{\partial y}(y,w)\right| & \leq & \underset{\ll1}{\underbrace{\left\{ \beta_{N_{s}}\left(\bar{x}_{M_{s}}\right)-\beta_{N_{s}}\left(\underline{x}_{M_{s}}\right)\right\} }}\left|\frac{\partial\mathfrak{l}_{N_{s}}}{\partial y}(y,w)\right|+\lim_{\Delta\to0}\frac{\left|\overline{I}_{\Delta}(s)\right|+\left|\underline{I}_{\Delta}(s)\right|}{\Delta}\\
 & \lesssim & (1.6)^{-N_{s-1}}(1.6)^{-n_{s}}+2(1.6)^{-N_{s+1}}\\
 & \leq & (1.6)^{-n_{s}}.
\end{eqnarray*}
(iii) We first recall the definition of $K_{M_{s},y}$. Define $\alpha:\NN\times\Sigma_{{\bf A}}\times\left[-\p/\left(\p+2\right),\p^{2}/\left(\p+2\right)\right]$
by 
\[
\alpha(s,w,y):=\frac{\int_{C_{[w]_{1}^{M_{s}}}}\frac{\partial K_{M_{s}-1,y}}{\partial x}(x)dx}{m_{\TT}\left(C_{[w]_{1}^{M_{s}}}\right)}=\frac{\mathfrak{l}_{M_{s}}\left(y,w\right)}{m_{\TT}\left(C_{[w]_{1}^{M_{s}}}\right)}.
\]
It follows that for $s\in\NN$ such that $N_{s}>{\bf j}(x,y)$, the
definition of $K_{M_{s},y}$ restricted to $C_{[w]_{1}^{M_{s}}}$
is the function defined by $K_{M_{s},y}\left(\underline{x}_{M_{s}}\right)=K_{N_{s},y}\left(\underline{x}_{M_{s}}\right)$
and $x$-derivative
\[
\frac{\partial K_{M_{s},y}}{\partial x}(x):=\begin{cases}
\frac{\partial G_{\alpha\left(s,w^{-},y\right),\alpha\left(s,w,y\right)}}{\partial x}\left(\frac{x-\underline{x}_{M_{s}}}{\epsilon_{t+1}m_{\TT}\left(C_{[w]_{1}^{M_{s}}}\right)}\right) & 0\leq x-\underline{x}_{M_{s}}\leq\epsilon_{s+1}m_{\TT}\left(C_{[w]_{1}^{M_{s}}}\right)\\
\alpha(s,w,y) & x-\underline{x}_{M_{s}}\geq\epsilon_{s+1}m_{\TT}\left(C_{[w]_{1}^{M_{s}}}\right)
\end{cases},
\]
where $w^{-}$ is the predecessor of $w$ in $\Sigma_{{\bf A}}\left(M_{t}\right)$
and $G_{\alpha_{1},\alpha_{2}}:[0,1]\to[0,1]$ is the function defined
by (\ref{eq: Definition of G_alpha_beta}). 

Therefore the function 
\[
\mathscr{K}_{t,y}(x):=\frac{K_{M_{s},y}(x)-K_{M_{s},y}\left(\underline{x}_{M_{s}}\right)}{\mathfrak{l}_{M_{s}}\left(y,w\right)}=\left(m_{\TT}\left(C_{[w]_{1}^{M_{s}}}\right)\right)^{-1}\left[\frac{K_{M_{s},y}(x)-K_{M_{s},y}\left(\underline{x}_{M_{s}}\right)}{\mathfrak{\alpha}\left(s,y,w\right)}\right]
\]
satisfies $\mathscr{K}_{t,y}\left(\underline{x}_{M_{t}}(w)\right)=0$
and 
\[
\frac{\partial\mathscr{K}_{t,y}}{\partial x}(x):=m_{\TT}\left(C_{[w]_{1}^{M_{s}}}\right)^{-1}\cdot\begin{cases}
\frac{\partial G_{\Upsilon(s,w,y),1}}{\partial x}\left(\frac{x-\underline{x}_{M_{s}}}{\epsilon_{s+1}m_{\TT}\left(C_{[w]_{1}^{M_{s}}}\right)}\right), & 0\leq x-\underline{x}_{M_{s}}\leq\epsilon_{s+1}m_{\TT}\left(C_{[w]_{1}^{M_{s}}}\right)\\
1 & x-\underline{x}_{M_{s}}\geq\epsilon_{s+1}m_{\TT}\left(C_{[w]_{1}^{M_{s}}}\right)
\end{cases},
\]
where $\Upsilon(s,w,y)=\frac{\alpha\left(s,w^{-},y\right)}{\alpha\left(s,w,y\right)}.$
The function $\mathscr{K}_{s,y}$ is important since the definition
of ${\bf x}_{s}(\Delta)$ is as the unique point so that 
\[
\mathscr{K}_{s,y+\Delta}\left({\bf x}_{t}(\Delta)\right)=\mathscr{K}_{t,y}\left(x\right).
\]
Since for all $0<a$, $\int_{0}^{1}G'_{a,1}(x)dx=1$, it follows that
if $x-\underline{x}_{M_{s}}\geq\epsilon_{s+1}m_{\TT}\left(C_{[w]_{1}^{M_{s}}}\right)$
then writing $\hat{x}=\underline{x}_{M_{s}}+\epsilon m_{\TT}\left(C_{[w]_{1}^{M_{s}}}\right)$

\begin{align*}
\mathscr{K}_{s,y+\Delta}(x) & =\int_{\underline{x}_{M_{s}}}^{x}\frac{\partial\mathscr{K}_{s,y+\Delta}}{\partial x}(x)\\
 & =\left(m_{\TT}\left(C_{[w]_{1}^{M_{s}}}\right)\right)^{-1}\left(\int_{\underline{x}_{M_{s}}}^{\hat{x}}\frac{\partial G_{\Upsilon(s,w,y+\Delta),1}}{\partial x}\left(\frac{x-\underline{x}_{M_{s}}}{\epsilon_{s+1}m_{\TT}\left(C_{[w]_{1}^{M_{s}}}\right)}\right)dx+(x-\hat{x})\right)\\
 & =\frac{x-\underline{x}_{M_{s}}}{m_{\TT}\left(C_{[w]_{1}^{M_{s}}}\right)}=\mathscr{K}_{s,y}(x)\ \ \ \ \text{by a similar reasonning. }
\end{align*}
If $x-\underline{x}_{M_{s}}\leq\epsilon_{s+1}m\left(C_{[w]_{1}^{M_{s}}}\right)$
then using the fact that for all $0\leq x\leq1$, $\left|G_{a,1}(x)-G_{b,1}(x)\right|\leq x\left|b-a\right|$,
\begin{eqnarray*}
\mathscr{K}_{s,y+\Delta}(x) & = & \mathscr{K}_{s,y}(x)\pm\left|\Upsilon(s,w,y+\Delta)-\Upsilon(s,w,y)\right|\cdot\left(x-\underline{x}_{M_{s}}(w)\right)\\
 & = & \mathscr{K}_{s,y}(x)\pm\epsilon_{s+1}\left|\Upsilon(s,w,y+\Delta)-\Upsilon(s,w,y)\right|.
\end{eqnarray*}
\begin{eqnarray*}
\left|\Upsilon(s,w,y+\Delta)-\Upsilon(s,w,y)\right| & = & \frac{m\left(C_{[w]_{1}^{M_{s}}}\right)}{m\left(C_{\left[w^{-}\right]_{1}^{M_{s}}}\right)}\left|\frac{\mathfrak{l}_{M_{s}}\left(y+\Delta,w^{-}\right)}{\mathfrak{l}_{M_{s}}\left(y+\Delta,w\right)}-\frac{\mathfrak{l}_{M_{s}}\left(y,w^{-}\right)}{\mathfrak{l}_{M_{s}}\left(y,w\right)}\right|\\
 & \leq & \Delta\frac{\p\left(\left|\frac{\partial\mathfrak{l}_{M_{s}}}{\partial y}(y,w)\right|+\left|\frac{\partial\mathfrak{l}_{M_{s}}}{\partial y}\left(y,w^{-}\right)\right|\right)}{\min\left(\mathfrak{l}_{M_{s}}(y,w),\mathfrak{l}_{M_{s}}\left(y,w^{-}\right)\right)}+o(\Delta).
\end{eqnarray*}
Thus if $\epsilon_{s+1}$ is sufficiently small (this choice depends
on $\{N_{t},M_{t},\lambda_{t}:t\leq t$ \} and $N_{s+1}$) then 
\[
\mathscr{K}_{s,y+\Delta}(x)=\mathscr{K}_{s,y}(x)\pm\frac{1}{2}\Delta\p^{-N_{s+1}-M_{s}}+o(\Delta),
\]
and for all $\Delta$ sufficiently small 
\begin{eqnarray*}
\mathscr{K}_{s,y}\left(x\right) & = & \mathscr{K}_{s,y+\Delta}\left({\bf x}_{s}(\Delta)\right)\\
 & = & \mathscr{K}_{s,y}\left({\bf x}_{s}(\Delta)\right)\pm\frac{1}{2}\Delta\p^{-N_{s+1}}+o(\Delta).
\end{eqnarray*}
Since 
\[
\frac{\partial\mathscr{K}_{s,y}}{\partial x}(x)\geq\left(2m_{\TT}\left(C_{[w]_{1}^{M_{t}}}\right)\right)^{-1}\geq\frac{\varphi^{-M_{s}}}{2},
\]
it follows that 
\begin{eqnarray*}
\left|x-{\bf x}_{s}(\Delta)\right| & \leq & 2\varphi^{M_{s}}\left|\mathscr{K}_{s,y}\left(x\right)-\mathscr{K}_{s,y}\left({\bf x}_{t}(\Delta)\right)\right|\\
 & \leq & \Delta\p^{-N_{s+1}}+o(\Delta)
\end{eqnarray*}
as required. 
\end{proof}
The next corollary is the final ingredient for the proof of Theorem
\ref{thm: The proof of convergence of y derivative}. 
\begin{cor}
\label{cor: convergence of y drivative of K_ny} There exists a choice
of $\left\{ \lambda_{t},n_{t},N_{t},m_{t},M_{t},\epsilon_{t}\right\} _{t\in\NN}$
such that:

(i) For all $(x,y)\in\MM_{\sim}$ and for all $t\in\NN$ such that
${\bf j}(x,y)<N_{t}$, 
\[
\left|\frac{\partial K_{N_{t+1},y}(x)}{\partial y}\right|\leq\left|\frac{\partial K_{N_{t},y}\left(\underline{x}_{M_{t}}\left(w\right)\right)}{\partial y}\right|+2(1.6)^{-n_{t}}.
\]

(ii) $\frac{\partial K_{N_{t+1},y}(x)}{\partial y}$ converges uniformly
in $R_{1}\cup R_{3}$ as $t\to\infty$. For every $(x,y)\in\partial\left(R_{1}\cup R_{3}\right)$
or $(x,y)\in U$, 
\[
\lim_{t\to\infty}\left|\frac{\partial K_{N_{t+1},y}(x)}{\partial y}\right|=0
\]
\end{cor}
\begin{proof}
We assume that $\left\{ \lambda_{t},n_{t},N_{t},m_{t},M_{t},\epsilon_{t}\right\} $
are chosen so that Lemma \ref{lem: derivatives, scaling function long after the breaking point}
holds.

(i) Similarly as in the proof of Corollary \ref{cor: The decay after the first break }.(i)
one can use the facts that 
\[
\left|\frac{\partial K_{N_{t+1},y}\left(\underline{x}_{N_{t+1}}\left(w\right)\right)}{\partial y}\right|\leq\left|\frac{\partial K_{M_{t},y}\left(\underline{x}_{M_{t}}\left(w\right)\right)}{\partial y}\right|+\sum_{n=M_{t}+1}^{N_{t+1}}\left|\frac{\partial\mathfrak{l}_{n}}{\partial y}(y,w)\right|,
\]
and 
\[
K_{M_{t},y}\left(\underline{x}_{M_{t}}\left(w\right)\right)=K_{N_{t},y}\left(\underline{x}_{M_{t}}\left(w\right)\right)
\]
 to show that 
\begin{eqnarray*}
\left|\frac{\partial K_{N_{t+1},y}\left(\underline{x}_{N_{t+1}}\right)}{\partial y}\right| & \overset{{\rm Lem.}\ref{lem: derivatives, scaling function long after the breaking point}.2.i}{\leq} & \left|\frac{\partial K_{N_{t},y}\left(\underline{x}_{M_{t}}\right)}{\partial y}\right|+\sum_{n=M_{t}+1}^{N_{t+1}}(1.6)^{M_{t}-n}\left|\frac{\partial\mathfrak{l}_{M_{t}}}{\partial y}(y,w)\right|\\
 & \overset{{\rm Lem.}\ref{lem: derivatives, scaling function long after the breaking point}.2.ii}{\leq} & \left|\frac{\partial K_{N_{t},y}\left(\underline{x}_{M_{t}}\right)}{\partial y}\right|+(1.6)^{-n_{t}}\sum_{n=M_{t}+1}^{N_{t}}(1.6)^{M_{t}-n}\\
 & \leq & \left|\frac{\partial K_{N_{t},y}\left(\underline{x}_{M_{t}}\right)}{\partial y}\right|+\frac{5}{3}(1.6)^{-n_{t}}.
\end{eqnarray*}
Therefore by Lemma \ref{lem: derivatives, scaling function long after the breaking point},
\begin{align*}
\lim_{\Delta\to0}\left|\frac{K_{N_{t+1},y+\Delta}\left({\bf x}_{t}\left(\Delta\right)\right)-K_{N_{t+1},y}\left(x\right)}{\Delta}\right| & \leq\left(\left|\frac{\partial K_{N_{t},y}\left(\underline{x}_{M_{t}}\right)}{\partial y}\right|+\frac{5}{3}(1.6)^{-n_{t}}\right)+\beta_{n}(x)\frac{\partial l_{N_{t+1}}}{\partial y}(y,w)\\
 & \leq\left(\left|\frac{\partial K_{N_{t},y}\left(\underline{x}_{M_{t}}\right)}{\partial y}\right|+\frac{5}{3}(1.6)^{-n_{t}}\right)+\beta_{n}(x)\left(1.6\right)^{-n_{t+1}}
\end{align*}
and by Lemma \ref{lem: derivatives, scaling function long after the breaking point}.2.(iii),
\begin{align*}
\left|\frac{K_{N_{t+1},y+\Delta}\left({\bf x}_{t}\left(\Delta\right)\right)-K_{N_{t+1},y+\Delta}\left(x\right)}{\Delta}\right| & \leq\left(\sup_{\left(x,y\right)\in\MM_{\sim}}\left|\frac{\partial K_{N_{t+1},y+\Delta}\left(x\right)}{\partial x}\right|\right)\left(\frac{\left|{\bf x}_{t}\left(\Delta\right)-x\right|}{\Delta}\right)\\
 & \leq\lambda_{1}^{2N_{t+1}}\p^{-N_{t+1}}+o(\Delta)\leq(1.6)^{-N_{t}}+o(\Delta).
\end{align*}
A combination of the previous two inequalities and $n_{t}=o(N_{t})$,
$N_{t}=o\left(n_{t+1}\right)$ shows that 
\[
\left|\frac{\partial K_{N_{t+1},y}\left(x)\right)}{\partial y}\right|\leq\left|\frac{\partial K_{N_{t},y}\left(\underline{x}_{M_{t}}\left(w\right)\right)}{\partial y}\right|+2(1.6)^{-n_{t}}.
\]

(ii) Let $t\in\NN$, and $(x,y)\in R_{1}\cup R_{3}$ By applying part
(i) of this corollary repeatedly one has that with $t(x,y)=\min\left\{ t:\ N_{t}>{\bf j}(x,y)\right\} $.
For $t>t(x,y)+1$ and $w$ such that $x\in C_{[w]_{1}^{N_{t+1}}}$,
by the first part of the corollary 
\[
\left|\frac{\partial K_{N_{t},y}\left(\underline{x}_{M_{t}}\right)}{\partial y}\right|\leq\left|\frac{\partial K_{N_{t-1},y}\left(\underline{x}_{M_{t-1}}\right)}{\partial y}\right|+2\left(1.6\right)^{-n_{t}}
\]
For $t=t(x,y)$ by Corollary \ref{cor: The decay after the first break }
\[
\left|\frac{\partial K_{N_{t(x,y)},y}\left(\underline{x}_{M_{t(x,y)}}\right)}{\partial y}\right|\leq(1.6)^{-{\bf j}(x,y)}.
\]
A combination of these two observations shows that for $t>t(x,y)$
\begin{eqnarray*}
\left|\frac{\partial K_{N_{t},y}(x)}{\partial y}\right| & \leq & \left|\frac{\partial K_{N_{t(x,y)},y}\left(\underline{x}_{M_{t(x,y)}}\right)}{\partial y}\right|+2\sum_{k=t(x,y)}^{t}(1.6)^{-n_{k}}.\\
 & \overset{}{\leq} & (1.6)^{-{\bf j}(x,y)}+2(1.6)^{-n_{t(x,y)}}.
\end{eqnarray*}
This is enough to show that $\left\{ \frac{\partial K_{N_{s+1},y}(x)}{\partial y}\right\} _{s=1}^{\infty}$
is a Cauchy sequence in the uniform topology. Indeed if $s,t>t(x,y)$,
then
\[
\left|\frac{\partial K_{N_{t},y}(x)}{\partial y}-\frac{\partial K_{N_{s},y}(x)}{\partial y}\right|\leq\sum_{k=t}^{s}(1.6)^{-n_{k}}\leq2(1.6)^{-n_{t}}.
\]
If $N_{t}<{\bf j}(x,y)-3\leq N_{s+1}$ then $K_{N_{t+1},y}(x)=x$
in a neighborhood of $y$ and hence 
\begin{eqnarray*}
\left|\frac{\partial K_{N_{t+1},y}(x)}{\partial y}-\frac{\partial K_{N_{s+1},y}(x)}{\partial y}\right| & = & \left|\frac{\partial K_{N_{s+1},y}(x)}{\partial y}\right|\\
 & \leq & (1.6)^{-{\bf j}(x,y)}+2(1.6)^{-n_{t}}\\
 & \leq & 3(1.6)^{-n_{t}}
\end{eqnarray*}
 We leave the bound on the easier cases $t=t(x,y)-1<s$, $s,t<t(x,y)-1$
to the reader. 
\end{proof}
\begin{rem}
The latter corollary shows that $\lim_{t\to\infty}\frac{\partial K_{N_{t},y}}{\partial y}$
is uniformly continuous in $x$ as a uniform limit of continuous functions.
As a consequence, since for all $x\in[0,1/\varphi]$, the sequence
$K_{N_{t},y}^{-1}(x)$ converges uniformly to $\mathfrak{h}_{y}^{-1}(x)$
and $\mathfrak{h}_{y}^{-1}$ is a homeomorphism of $[0,1/\varphi]$
then for all $x\in[0,1/\varphi]$, 
\[
\frac{\partial K_{N_{t},y}\left(K_{N_{t},y}^{-1}(x)\right)}{\partial y}\xrightarrow[t\to\infty]{}\lim_{t\to\infty}\frac{\partial K_{N_{t},y}\left(\mathfrak{h}_{y}^{-1}(x)\right)}{\partial y}
\]
and the convergence is uniform in $t\in\NN$. 
\end{rem}
\begin{proof}[Proof of Theorem \ref{thm: The proof of convergence of y derivative}] 

Lemma \ref{lem: the x-derivative of the two dimensional } shows that
$\frac{\partial\mathfrak{z}}{\partial x}=\lim_{t\to\infty}\frac{\partial\mathfrak{z}_{N_{t}}}{\partial x}(x,y)$
exists and is a continuous function of $\MM_{\sim}$. It remains to
show that $\frac{\partial\mathfrak{z}}{\partial y}=\lim_{t\to\infty}\frac{\partial\mathfrak{z}_{N_{t}}}{\partial y}(x,y)$
exists and is a continuous function of $\MM_{\sim}$. To this end,
write
\[
B_{N}(x,y)=\left(\begin{array}{cc}
\frac{\partial K_{N,y}(x)}{\partial x} & \frac{\partial K_{N,y}(x)}{\partial y}\\
0 & 1
\end{array}\right)
\]
for the differential of the map $(x,y)\mapsto\left(K_{N,y}(x),y\right)$.
By the chain rule 
\[
D_{\mathfrak{Z}_{N_{t}}}(x,y)=B_{N_{t}}\left(SK_{N_{t},y}^{-1}(x),-y/\p)\right)\left(\begin{array}{cc}
\p & 0\\
0 & -1/\p
\end{array}\right)B_{N_{t}}^{-1}\left(K_{N_{t},y}^{-1}(x),y\right).
\]
This yields that 
\begin{eqnarray*}
\frac{\partial\mathfrak{z}_{N_{t}}}{\partial y}(x,y) & = & \p\frac{\partial K_{N_{t},-y/\p}(SK_{N_{t},y}^{-1}(x))}{\partial x}\cdot\left(-\left(\frac{\partial K_{N_{t},-y/\p}\left(K_{N_{t},y}^{-1}(x)\right)}{\partial x}\right)^{-1}\left(\frac{\partial K_{N_{t},y}\left(K_{N_{t},y}^{-1}(x)\right)}{\partial y}\right)\right)\\
 &  & -\frac{1}{\p}\frac{\partial K_{N_{t},-y/\p}\left(SK_{N_{t},y}^{-1}(x)\right)}{\partial y}\\
 & = & -\frac{\partial\mathfrak{z}_{N_{t}}}{\partial x}(x,y)\cdot\frac{\partial K_{N_{t},y}\left(K_{N_{t},y}^{-1}(x)\right)}{\partial y}-\frac{1}{\p}\frac{\partial K_{N_{t},-y/\p}\left(SK_{N_{t},y}^{-1}(x)\right)}{\partial y}.
\end{eqnarray*}
Since all the terms on the right hand side converge uniformly as $t\to\infty$,
the Theorem is proved. 

\end{proof}

\subsection{Proof of the Anosov property for $\mathfrak{Z}$. }

So far we have shown that $\mathfrak{Z}_{N_{t}}$ converges uniformly
to $\mathfrak{Z}$ and estimated the derivatives. We are going to
use the following well known Lemma, it's proof can be found in \cite{Viana}.
A function $\mathcal{A}:\MM_{\sim}\times\ZZ\to SL(2,\RR)$ is \textit{linear
cocycle} over a homeomorphism $\mathfrak{f:}\MM_{\sim}\to M_{\sim}$
if for any $m.n\in\ZZ$ and $x\in\TT^{2}$, 
\[
\mathcal{A}_{m+n}(x)=\mathcal{A}_{m}\circ\mathfrak{f}^{n}(x)\mathcal{A}_{n}(x).
\]
We say that the cocycle is Hyperbolic if there are $\sigma>1$ and
$C>0$ so that for every $x\in\MM_{\sim}$ there exists transverse
lines $E_{x}^{s}$ and $E_{x}^{u}$ in $\RR^{2}$ such that 

1. $\mathscr{A}(x)E_{x}^{s}=E_{\mathfrak{f}(x)}^{s}$ and $\mathscr{A}(x)E_{x}^{u}=E_{\mathfrak{f}(x)}^{u}$. 

2. $\left|\mathcal{A}_{n}(x)v^{s}\right|\leq C\sigma^{n}\left|v^{s}\right|$
and $\left|\mathcal{A}_{-n}(x)v^{u}\right|\leq C\sigma^{n}\left|v^{u}\right|$
for every $v^{s}\in E_{x}^{s}$, $v^{u}\in E_{x}^{u}$ and $n\geq1$. 
\begin{prop*}
\cite[Prop. 2.1 ]{Viana}Let $A:\MM_{\sim}\times\ZZ\to SL(2,\RR)$
be a linear cocycle over a homeomorphism $\mathfrak{f:}M_{\sim}\to M_{\sim}$.
If there exists $v\in\RR^{2}$, constants $c>0$ and $\sigma>1$ such
that $\left|A_{n}(x)v\right|\geq c\sigma^{n}$ then $\mathcal{A}$
is hyperbolic. The transverse lines $E_{x}^{s},E_{x,}^{u}$ in $\RR^{2}$
satisfy that for any $\sigma_{0}<\sigma$, there exists $C>0$ so
that for any $v^{s}\in E_{x}^{s}$, $v^{u}\in E_{x}^{u}$ and $n\geq1$,
\[
\left|\mathcal{A}_{n}(x)v^{s}\right|\leq C\sigma_{0}^{n}\left|v^{s}\right|\text{ and }\left|\mathcal{A}_{-n}(x)v^{u}\right|\leq C\sigma_{0}^{n}\left|v^{u}\right|.
\]
\end{prop*}
\begin{proof}[Proof that $\mathfrak{Z}$ is Anosov] 

Define $\mathcal{A}:\MM_{\sim}\times\mathbb{Z}\to SL\left(2,\RR\right)$
by
\[
\mathcal{A}(x)=\frac{1}{\det\left(D_{\mathfrak{Z}}(x,y)\right)}D_{\mathfrak{Z}}(x,y).
\]
Since $D_{\mathfrak{Z}}$ is of the form 
\[
\left(\begin{array}{cc}
\frac{\partial\mathfrak{z}}{\partial x}(x,y) & *\\
0 & -1/\p
\end{array}\right)
\]
and 
\[
1.6\leq\frac{\partial\mathfrak{z}}{\partial y}(x,y)\leq1.7
\]
one has that for all $(x,y)\in M_{\sim,}$ 
\[
\frac{1.6}{\p}\leq\left|\det\left(D_{\mathfrak{Z}}(x,y)\right)\right|\leq\frac{1.7}{\p}.
\]
and 
\[
\left|D_{\mathfrak{Z}^{n}}(x,y)\left(\begin{array}{c}
1\\
0
\end{array}\right)\right|=\left|\left(\begin{array}{c}
\prod_{k=0}^{n-1}\frac{\partial\mathfrak{z}}{\partial x}\circ\mathfrak{Z}^{k}(x,y)\\
0
\end{array}\right)\right|\geq(1.6)^{n}\left|\left(\begin{array}{c}
1\\
0
\end{array}\right)\right|.
\]
 It then follows that with $v=(1,0)^{tr}$ 
\[
\left|\mathcal{A}_{n}(x,y)v\right|=\left(\det\left(D_{\mathfrak{Z}^{n}}(x,y)\right)\right)^{-1}\left|D_{\mathfrak{Z}^{n}}(x,y)v\right|\geq\left(\frac{\varphi\cdot1.6}{1.7}\right)^{n}|v|\geq(1.52)^{n}|v|
\]
there exists transverse lines $E_{x}^{s}$ and $E_{x}^{u}$ in $\RR^{2}\simeq T_{x}M_{\sim}$
and $C>0$ so that for any $v^{u}\in E_{x}^{u}$, 
\[
\left|\mathcal{A}_{n}(x)v^{u}\right|\geq(1.5)^{n}\left|v^{u}\right|.
\]
It follows that 
\begin{eqnarray*}
\left|D_{\mathfrak{Z}^{n}}(x,y)v^{u}\right| & = & \left|\mathcal{A}_{n}(x)v^{u}\right|\prod_{k=0}^{n-1}\left|\det\left(D_{\mathfrak{Z}}\left(\mathfrak{Z}^{k}(x,y)\right)\right)\right|\\
 & \geq & C(1.5)^{n}\left(\frac{1.6}{\p}\right)^{n}\left|v^{u}\right|\\
 & \geq & C(1.48)^{n}\left|v^{u}\right|.
\end{eqnarray*}
Similarly one has for every $v^{s}\in E_{x}^{s}$, 
\begin{eqnarray*}
\left|D_{\mathfrak{Z}^{n}}(x,y)v^{s}\right| & \leq & C(1.5)^{-n}\left(\frac{1.7}{\p}\right)^{n}\left|v^{s}\right|\\
 & \leq & C\left(0.7\right)^{n}
\end{eqnarray*}
and so $\mathfrak{Z:}\MM_{\sim}\to\MM_{\sim}$ is Anosov. 

\end{proof}

\subsection{Proof of the type ${\rm III}_{1}$ property for $\mathcal{\mathfrak{Z}}$. }

For $-\frac{\p}{\p+2}<y<\frac{1}{\p+2}$ let, 
\[
\mathfrak{K}_{y}(x):=\lim_{n\to\infty}K_{n,y}(x):\TT\to\TT
\]
and for $\frac{1}{\p+2}\leq y\leq\frac{\p^{2}}{\p+2}$, 
\[
\mathfrak{K}_{y}(x):=\lim_{n\to\infty}K_{n,y}(x):[0,1/\p]\to[0,1/\p].
\]
In both cases it is an orientation preserving homeomorphism.

We will show that the measures $m_{{\rm Leb}(\TT)}\circ\mathfrak{K}_{y}$
and $m_{{\rm Leb(\left[0,1/\p\right])}}\circ\mathfrak{K}_{y}$ are
equivalent measures to $\mu^{+}$, the measure on $\TT$ arising from
$\left\{ \lambda_{k},m_{k}.M_{k},n_{k}.N_{k}\right\} $ in the previous
section.

In addition the Radon Nykodym derivatives 
\[
\frac{d\eta_{y}}{d\mu^{+}}(x):\mathbb{M}\to[0,\infty)
\]
defined by 
\[
\frac{d\eta_{y}}{d\mu^{+}}(x):=\frac{dm_{\TT}\circ\mathfrak{K}_{y}}{d\mu^{+}}(x),
\]
is a $\left(\MM_{\sim},\BB\left(\MM_{\sim}\right),\mu\right)$ measurable
function. This means that the measure $\eta$ on $\MM_{\sim}$ defined
by 
\begin{eqnarray*}
\int_{M_{\sim}}u(x,y)d\eta & = & \int_{\MM_{\sim}}u(x,y)d\eta_{y}(x)dy\\
 & = & \int_{\MM_{\sim}}u(x,y)\frac{d\eta_{y}}{d\mu^{+}}(x)d\mu^{+}(x)dy
\end{eqnarray*}
is equivalent to $\mu={\bf m}_{\mathbb{M}}\circ\mathfrak{H}_{\underline{0}}=\mu^{+}\otimes dy$. 

Since $\left(\mathbb{M_{\sim}},\BB_{\MM},\mu,\tilde{f}\right)$ is
a type ${\rm III}_{1}$ transformation and $\mu\sim\eta$, $\left(\mathbb{M},\BB_{\MM},\eta,\tilde{f}\right)$
is a type ${\rm III}_{1}$ transformations. Thus $\left(\MM,\BB_{\MM},{\rm m}_{\MM},\mathfrak{Z}\right)$
is a type ${\rm III}_{1}$ transformation since $\pi(x,y):=\left(\mathfrak{K}_{y}(x),y\right):\left(\MM,\BB_{\MM},{\rm m}_{\MM},\mathfrak{Z}\right)\to\left(\MM,\BB_{\MM},\eta,\tilde{f}\right)$
is an isomorphism. Therefore what is left to prove is the following. 
\begin{lem}
(i) For all $-\frac{\p}{\p+2}<y<\frac{1}{\p+2}$, $\left(\mathfrak{K}_{y}^{-1}\right)_{*}m_{\TT}$
is an equivalent measure to $\mu^{+}$ (the measure on $\TT$ arising
from $\left\{ \lambda_{k},m_{k}.M_{k},n_{k}.N_{k},\epsilon_{k}\right\} $
in the previous section).

(ii) For all $\frac{1}{\p+2}<y<\frac{\p^{2}}{\p+2}$, $\left(\mathfrak{K}_{y}\right)_{*}\left.m_{\TT}\right|_{[0,1/\p]}$
is an equivalent measure to $\mu^{+}|_{[0,1/\p]}$.

(iii) The Radon Nykodym derivatives $\frac{d\eta_{y}}{d\mu^{+}}(x):\mathbb{M}\to[0,\infty)$
are measurable in $\left(\MM_{\sim},\BB_{\MM_{\sim}},\mu\right)$. 
\end{lem}
\begin{proof}
Fix $-\frac{\p}{\p+2}<y<\frac{1}{\p+2}$. The proof is the same as
in Lemma \ref{Lem:  the pertubation of the beta shift} by using the
theory of local absolute continuity of Shiryaev with $\mathcal{F}_{t}:=\left\{ C_{[w]_{1}^{N_{t}}}:\ w\in\Sigma_{{\bf A}}\right\} $.
By the construction 
\[
\left(m_{\TT}\circ\mathfrak{K}_{y}\right)_{t}:=\left.m_{\TT}\circ\mathfrak{K}_{y}\right|_{\mathcal{F}_{t}}=m_{\TT}\circ K_{N_{t},y},
\]
and 
\[
\left(\mu^{+}\right)=m_{\TT}\circ H_{\underline{0},N_{t}}.
\]
Therefore, 
\[
Z_{t,y}(x):=\frac{d\left(m_{\TT}\circ\mathfrak{K}_{y}\right)_{t}}{d\left(\mu^{+}\right)_{t}}(x)=\frac{\left(\frac{\partial K_{N_{t},y}}{\partial x}\right)}{\left(\frac{\partial H_{\underbar{0},N_{t}}}{\partial x}\right)}(x).
\]
The rest of the proof that $Z_{t,y}(x)$ is uniformly integrable and
hence converges a.s. as $t\to\infty$ is the same as in Lemma \ref{Lem:  the pertubation of the beta shift}.
This proves (i) and $(ii)$.

To see $(iii)$, notice that the function 
\[
\left(x,y\right)\mapsto\frac{d\eta_{y}}{d\mu^{+}}(x)=\lim_{t\to\infty}Z_{t,y}(x)
\]
 is almost surely a limit of continuous functions, hence measurable. 
\end{proof}

\specialsection{Appendix: Proof of Theorem \ref{thm: lambdas are...}}

Assume that $\left\{ \lambda_{k},n_{k},N_{k},m_{k},M_{k}\right\} _{k\geq1}$
and $M_{0}$ are chosen via the inductive construction, $\left\{ \pi_{k},P_{k}\right\} _{k\in\ZZ}$
are defined by (\ref{eq: def of P's}) and $\mu=\M{\pi_{k},P_{k}:\ k\in\ZZ}.$Again
$T$ denotes the shift on $\Sigma_{{\bf A}}$. The proof of non-singularity,
the $K$-property of the shift with respect to $\mu$ and that 
\[
T'(x)=\frac{d\mu\circ T}{d\mu}(x)=\prod_{k=1}^{\infty}\frac{P_{k-1}\left(x_{k},x_{k+1}\right)}{P_{k}\left(x_{k},x_{k+1}\right)}
\]
appears in \cite[Thm. 6]{Kosloff TMS}. In order to show the other
properties of the Markov Shift, we will need a more concrete expression
of the Radon Nykodym derivatives. The measure $\mu$, or more concretely
it's transition matrices, differs from the stationary $\left\{ \pi_{\mathbf{Q}},\mathbf{Q}\right\} $
measure only when one moves inside state $1$ in the segments $\left[M_{j},N_{j+1}\right)$.
Denote by 
\[
L_{j}(x):=\#\left\{ k\in\left[M_{j-1},N_{j}\right):\ x_{k}=1\right\} 
\]
and 
\[
V_{j}(x)=\#\left\{ k\in\left[M_{j-1},N_{j}\right):\ x_{k}=x_{k+1}=1\right\} .
\]

\begin{lem}
\label{lem: R_N derivatives}For every $\epsilon>0$, there exists
$t_{0}\in\NN$ s.t for every $t>t_{0}$, $N_{t}\leq n<m_{t}$ and
$x\in\{1,2,3\}^{\ZZ}$, 
\[
\left(T^{n}\right)'(x)=(1\pm\epsilon)\prod_{k=1}^{t}\left[\left(\frac{1+\varphi}{1+\varphi\lambda_{k}}\right)^{L_{k}\circ T^{n}(x)-L_{k}(x)}\cdot\lambda_{k}^{V_{k}\circ T^{n}(x)-V_{k}(x)}\right].
\]
 
\end{lem}
\begin{proof}
Let $\epsilon>0$, $t\in\NN$ and $N_{t}\leq n<m_{t}$. Canceling
out all the $k's$ such that $P_{k-n}=P_{k}$ one can see that 
\begin{eqnarray*}
\left(T^{n}\right)'(x) & = & I_{t}\cdot\tilde{I_{t}}
\end{eqnarray*}
 where (notice in the definition of $I_{t}$ that $n>N_{t}$) 
\[
I_{t}=\prod_{u=1}^{t}\left[\left(\prod_{k=M_{u-1}}^{N_{u}}\frac{P_{k-n}\left(x_{k},x_{k+1}\right)}{P_{k}\left(x_{k},x_{k+1}\right)}\right)\cdot\left(\prod_{k=M_{u-1}+n}^{N_{u}+n-1}\frac{P_{k-n}\left(x_{k},x_{k+1}\right)}{P_{k}\left(x_{k},x_{k+1}\right)}\right)\right]
\]
and (here notice that $M_{t}>N_{t}+m_{t}>N_{t}+n)$ 
\[
\tilde{I}_{t}=\prod_{u=t+1}^{\infty}\left[\left(\prod_{k=M_{u-1}}^{M_{u-1}+n-1}\frac{P_{k-n}\left(x_{k},x_{k+1}\right)}{P_{k}\left(x_{k},x_{k+1}\right)}\right)\cdot\left(\prod_{k=N_{u}}^{N_{u}+n-1}\frac{P_{k-n}\left(x_{k},x_{k+1}\right)}{P_{k}\left(x_{k},x_{k+1}\right)}\right)\right].
\]
We will analyze the two terms separately. Since for every $M_{u-1}\leq k<M_{u-1}+n$,
$P_{k}=\mathbf{Q_{u}}$ and $P_{k-n}=\mathbf{Q}$, 
\[
\frac{P_{k-n}\left(x_{k},x_{k+1}\right)}{P_{k}\left(x_{k},x_{k+1}\right)}\leq\frac{{\rm \mathbf{Q}_{1,3}}}{{\rm \left(\mathbf{Q_{u}}\right)}_{1,3}}=\frac{1+\varphi\lambda_{u}}{1+\varphi}\leq\lambda_{u}.
\]
Similarly for $N_{u}\leq k<N_{u}+n$, $P^{(k)}=\mathbf{Q}$ and $P_{k-n}={\rm \mathbf{Q_{u}}}$
. Therefore 
\[
\frac{P_{k-n}\left(x_{k},x_{k+1}\right)}{P_{k}\left(x_{k},x_{k+1}\right)}\leq\frac{{\rm \left(\mathbf{Q_{u}}\right)_{1,1}}}{\mathbf{Q}_{1,1}}\leq\lambda_{u}.
\]
and 
\[
\lambda_{u}^{-2n}\leq\left(\prod_{k=M_{u-1}}^{M_{u-1}+n-1}\frac{P_{k-n}\left(x_{k},x_{k+1}\right)}{P_{k}\left(x_{k},x_{k+1}\right)}\right)\cdot\left(\prod_{k=N_{u}}^{N_{u}+n-1}\frac{P_{k-n}\left(x_{k},x_{k+1}\right)}{P_{k}\left(x_{k},x_{k+1}\right)}\right)\leq\lambda_{u}^{2n},
\]
here the lower bound is achieved by a similar analysis. This gives
\begin{align*}
\tilde{I}_{t} & =\prod_{u=t+1}^{\infty}\left[\lambda_{u}^{\pm2n}\right]=\prod_{u=t+1}^{\infty}\left[\lambda_{u}^{\pm2m_{u-1}}\right]\quad\text{(since\ \ensuremath{\forall u>t},\ \ensuremath{n<m_{t}<m_{u})}}\\
 & \overset{\eqref{eq:condition on Lambda}}{=}e^{\pm\sum_{n=t+1}^{\infty}\frac{1}{2^{n}}}\xrightarrow[t\to\infty]{}1.
\end{align*}
Consequently there exists $t_{0}\in\NN$ so that for all $x\in\Sigma_{{\bf A}}$,
$t>t_{0}$ and $N_{t}\leq n\leq m_{t}$,
\[
\left(T^{n}\right)'(x)=(1\pm\epsilon)I_{t}.
\]
By noticing that for $k\in\bigcup_{j=1}^{t}\left(\left[M_{j-1},N_{j}\right)\cup\left[M_{j-1}+n,N_{j}+n\right)\right),$
\[
P_{k-n}\left(x_{k},x_{k+1}\right)\neq P_{k}\left(x_{k},x_{k+1}\right)
\]
if and only if $x_{k}=1$ one can check that 
\[
I_{t}=\prod_{k=1}^{t}\left[\left(\frac{1+\varphi}{1+\varphi\lambda_{k}}\right)^{L_{k}\circ T^{n}(x)-L_{k}(x)}\lambda_{k}^{V_{k}\circ T^{n}(x)-V_{k}(x)}\right].
\]
\end{proof}
\begin{cor}
The shift $\left(\{1,2,3\}^{\ZZ},\mu,T\right)$ is conservative and
ergodic.
\end{cor}
\begin{proof}
Since the shift is a $K$-automorphism it is enough to prove conservativity. 

For every $j\in\NN$, $0\leq L_{k}(x),D_{k}(x)\leq n_{k}$. Whence
\begin{eqnarray*}
\left(\frac{1+\varphi}{1+\varphi\lambda_{k}}\right)^{L_{k}\circ T^{n}(x)-L_{k}(x)}\lambda_{k}^{V_{k}\circ T^{n}(x)-V_{k}(x)} & \geq & \left(\frac{1+\varphi}{1+\varphi\lambda_{k}}\right)^{L_{k}\circ T^{n}(x)}\lambda_{k}^{-V_{k}(x)}\\
 & \geq & \lambda_{k}^{-2n_{k}}\geq\lambda_{1}^{-2n_{k}},
\end{eqnarray*}
and for every $t\in\NN$, 
\[
\prod_{k=1}^{t}\left[\left(\frac{1+\varphi}{1+\varphi\lambda_{k}}\right)^{L_{k}\circ T^{n}(x)-L_{k}(x)}\lambda_{k}^{V_{k}\circ T^{n}(x)-V_{k}(x)}\right]\geq\lambda_{1}^{-2\sum_{k=1}^{t}n_{k}}\geq\lambda_{1}^{-2N_{t}}.
\]
By Lemma \ref{lem: R_N derivatives} there exists $t_{0}\in\NN$ such
that for all $t\geq t_{0}$, $N_{t}\leq n\leq m_{t}$ and $x\in\Sigma_{{\bf A}}$,
$\RN n(x)\geq\frac{\lambda_{1}^{-2N_{t}}}{2}$. Therefore for all
$x\in\Sigma_{{\bf A}}$,
\[
\sum_{n=1}^{\infty}\RN n(x)\geq\sum_{t=1}^{\infty}\sum_{n=N_{t}}^{m_{t}}\RN n(x)\geq\sum_{t=t_{0}}^{\infty}\frac{1}{2}\left(m_{t}-N_{t}\right)\lambda_{1}^{-2N_{t}}\overset{\eqref{m_l is large enough for shift conservative}}{=}\infty.
\]
By Hopfs criteria the shift is conservative. 
\end{proof}

\subsubsection{Proof of the type ${\rm III}_{1}$ property}

In order to prove that the ratio set is $[0,\infty)$ we are going
to use the following principle: since $R(T)$ is a multiplicative
subset it is enough to show that there exists $y_{n}\in R(T)\backslash\{1\}$
with $y_{n}\to1$ as $n\to\infty$. 
\begin{thm}
\label{thm: lambdas are in R(T)}Let $\mu$ be the Markov measure
constructed in Subsection \ref{subsec:The-construction.}. For every
$n\in\NN$, $\lambda_{n}\cdot\frac{1+\varphi}{1+\varphi\lambda_{n}}\in R(T)$
and therefore the shift is type ${\rm III}_{1}$.
\end{thm}
Fix $n\in\NN$. The first stage in proving that $\lambda_{n}\cdot\frac{1+\varphi}{1+\varphi\lambda_{n}}\in R(T)$
is to show that the ratio set condition is satisfied for all cylinders
with a positive proportion of the measure of the cylinder set. Then
for a general $A\in\BB_{+}$, we use the density of cylinder sets
in $\mathcal{B}$. 

Given $t\in\NN$, denote by $\mathcal{C}(t)$ the collection of all
$[c]_{0}^{N_{t}}$ cylinder sets such that
\begin{equation}
L_{t}(c)=\sum_{k=M_{t-1}}^{N_{t}-1}{\bf 1}_{\left[c_{k}=1\right]}\in\left(\frac{n_{t}}{4},\frac{n_{t}}{2}\right)\text{ and}\ \sum_{k=M_{t-1}}^{N_{t}-1}{\bf 1}_{\left[c_{k}=2,c_{k+1}=3\right]}\geq\frac{n_{t}}{15}.\label{eq: Free bits in C}
\end{equation}
Since 
\[
\mu\left([c]_{M_{t-1}}^{N_{t}}\right)=\nu_{\pi_{M_{t-1}},\mathbf{Q_{t}}}\left([c]_{0}^{n_{t}}\right),
\]
it follows from (\ref{Ergodic theorem cond 1}) and (\ref{eq:Ergodic theorem cond 2})
that for all $t$ large enough, 
\[
\mu\left(\bigcup_{C\in\mathcal{C}(t)}C\right)\geq1-\frac{1}{2t}.
\]
 In order to shorten the notation, given $M,j\in\NN$, $B\in\BB$
and $\epsilon>0$, let 
\[
\mathfrak{RSC}\left(M,B,j,\epsilon\right):=B\cap T^{-M}B\cap\left[\RN M=\lambda_{j}\cdot\frac{1+\varphi}{1+\varphi\lambda_{j}}\cdot\left(1\pm\epsilon\right)\right],
\]
and for $M\in\NN$, 
\[
\Sigma_{{\bf A}}\left(M\right):=\{1,2,3\}^{M}\cap\Sigma_{{\bf A}}.
\]
 
\begin{lem}
\label{Main step for EVC. }For every $[b]_{-n}^{n}$ cylinder set,
$\epsilon>0$ and $j\in\NN$, there exists a $t_{0}\in\NN$ so that
for all $t>t_{0}$ the following holds: 

For every $C=[c]_{0}^{N_{t}-1}\in\mathcal{C}(t)$ there exists $d=d(b,C)\in\Sigma_{{\bf A}}\left(N_{t}+n\right)$
such that for every $\NN\ni l\leq m_{t}/k_{t}$, 
\begin{equation}
C\cap[d]_{lk_{t}-n}^{lk_{t}+N_{t}-1}\subset T^{-lk_{t}}[b]_{-n}^{n}\cap\left[\RN{lk_{t}}=\lambda_{j}\cdot\frac{1+\varphi}{1+\varphi\lambda_{j}}\cdot\left(1\pm\epsilon\right)\right].\label{eq: cylinder for EVC}
\end{equation}
 
\end{lem}
Recall that $k_{t}>N_{t}$ is defined as a $\left(1\pm3^{-3N_{t}}\right)$
mixing time for ${\bf Q}$. 
\begin{proof}
Let $[b]_{-n}^{n},\epsilon>0$ and $j\in\NN$ be given. By Lemma \ref{lem: R_N derivatives}
there exists $\tau$ such that for every $t\geq\tau$ and $1\leq l\leq m_{t}/k_{t}$
(here $lk_{t}\in\left[N_{t},m_{t}\right)$),
\[
\RN{lk_{t}}(x)=(1\pm\epsilon)\prod_{k=1}^{t}\left[\left(\frac{1+\varphi}{1+\varphi\lambda_{k}}\right)^{L_{k}\circ T^{lk_{t}}(x)-L_{k}(x)}\cdot\lambda_{k}^{V_{k}\circ T^{lk_{t}}(x)-V_{k}(x)}\right].
\]
Choose $t_{0}$ to be any integer which satisfies $t_{0}>\max\left(\tau,j\right)$
and $M_{t_{0}}>n$. 

Let $t>t_{0}$ and choose a cylinder set $[c]_{0}^{N_{t}}\in\mathcal{C}(t)$
which intersects $[b]_{-n}^{n}.$ That is $c_{i}=b_{i}$ for $i\in[0,n]$.
We need now to choose $d\in\Sigma_{{\bf A}}\left(N_{t}+n\right)$
which satisfies (\ref{eq: cylinder for EVC}). Notice that for $x\in[d]_{lk_{t}-n}^{lk_{t}+N_{t}}\cap[c]_{0}^{N_{t}},$
\begin{eqnarray*}
 &  & \prod_{k=1}^{t}\left[\left(\frac{1+\varphi}{1+\varphi\lambda_{k}}\right)^{L_{k}\circ T^{lk_{t}}(x)-L_{k}(x)}\cdot\lambda_{k}^{V_{k}\circ T^{lk_{t}}(x)-V_{k}(x)}\right]\\
 & = & \prod_{k=1}^{t}\left[\left(\frac{1+\varphi}{1+\varphi\lambda_{k}}\right)^{L_{k}(d)-L_{k}(c)}\cdot\lambda_{k}^{V_{k}(d)-V_{k}(c)}\right],
\end{eqnarray*}
in this representation we look at $[d]_{-n}^{N_{t}}$. For all $k\in[0,M_{t-1}]$,
let 
\[
d_{k}=c_{k}
\]
and for all $k\in[-n,0)$, 
\[
d_{k}=b_{k}.
\]
Notice that this means that for $k\in[-n,n]$, $d_{k}=b_{k}$ and
thus 
\[
[d]_{lk_{t}-n}^{lk_{t}+N_{t}}\subset T^{-lk_{t}}[b]_{-n}^{n}.
\]
 Let $p(j,t)\leq\frac{n_{t}}{20}$ be the integer (condition (\ref{eq: n_t is very large w.r.t lattic condition}))
such that 
\[
\left(\lambda_{t}\cdot\frac{1+\varphi}{1+\varphi\lambda_{t}}\right)^{p(j,t)}=\lambda_{j}\cdot\frac{1+\varphi}{1+\varphi\lambda_{j}}.
\]
Set $d_{k}=1$ for all $k\in\left[M_{t-1},M_{t-1}+V_{t}(c)+p(j,t)\right]$
and then continue repeatedly with the sequence $"321"$ $L_{t}(c)-V_{t}(c)$
times. Since $c$ satisfies (\ref{eq: Free bits in C}), this construction
is well defined (e.g. we have not reached yet $k=N_{t}-1$). Continue
with sequences of $32$ till $k=N_{t}-1$. 

Thus we have defined $d$ in such a way so that 
\[
L_{t}(d)-L_{t}(c)=p(j,t)
\]
and 
\[
V_{t}(d)-V_{t}(c)=p(j,t).
\]
In addition for all $0\leq k<t$, $L_{k}(d)=L_{k}(c)\ {\rm and}\ V_{k}(c)=V_{k}(d)$.
Thus for all $x\in[d]_{lk_{t}-n}^{lk_{t}+N_{t}}\cap[c]_{0}^{N_{t}},$
\begin{eqnarray*}
\RN{lk_{t}}(x) & = & \left(1\pm\epsilon\right)\prod_{k=1}^{t}\left[\left(\frac{1+\varphi}{1+\varphi\lambda_{k}}\right)^{L_{k}(d)-L_{k}(c)}\cdot\lambda_{k}^{V_{k}(d)-V_{k}(c)}\right]\\
 & = & \left(1\pm\epsilon\right)\left(\lambda_{t}\cdot\frac{1+\varphi}{1+\varphi\lambda_{t}}\right)^{p(j,t)}\\
 & = & \left(1\pm\epsilon\right)\left(\lambda_{j}\frac{1+\varphi}{1+\varphi\lambda_{j}}\right).
\end{eqnarray*}
This proves the lemma. 
\end{proof}
In the course of the proof one sees that the event 
\[
\left([b]_{-n}^{n}\cap[c]_{0}^{N_{t}}\right)\cap\left(T^{-lk_{t}}[b]_{-n}^{n}\cap\left\{ \RN{lk_{t}}=\lambda_{j}\cdot\frac{1+\varphi}{1+\varphi\lambda_{j}}\cdot\left(1\pm\epsilon\right)\right\} \right)^{c}
\]
is $\F{}\left(lk_{t}-n,lk_{t}+N_{t}\right)$ measurable and does not
depend on $\cF\left(lk_{t}+N_{t},lk_{t}+2N_{t}\right)$. 
\begin{rem}
Given $[c]_{0}^{N_{t}}\in\mathcal{C}_{t}$ we have defined $d=d(c)\in\Sigma_{{\bf A}}\left(N_{t}+n\right)$.
The definition of $d$ is not necessarily one to one. This is because
if $\left[\tilde{c}\right]_{0}^{M_{t-1}}=[c]_{0}^{M_{t-1}}$, $V_{t}(c)=V_{t}\left(\tilde{c}\right)$
and $L_{t}(c)=L_{t}\left(\tilde{c}\right)$ then $d(c)=d\left(\tilde{c}\right)$.
In order to make it one to one we will use 
\[
[d(c),c]_{lk_{t}-n}^{lk_{t}+2N_{t}}
\]
instead of $[d(c)]_{lk_{t}}^{lk_{t}+N_{t}}$where by $\left[a,b\right]_{l}^{l+{\rm length}(a)+{\rm length}(b)}$
we mean the concatenation of $a$ and $b$. This can be thought of
as putting a Marker on $d(c)$. In order that the concatenation will
be in $\Sigma_{A}$ we need that 
\[
{\bf Q}\left(d(c)_{N_{t}-1},c_{0}\right)>0.
\]
This can be done by possibly changing the last two coordinates of
$d(c)$. This will change the value of $\RN{lk_{t}}$ by at most a
factor of $\lambda_{t}^{\pm4}$ , which is close enough to one. We
will denote by ${\bf d}(c):=(\tilde{d}(c),c)$. We still have 
\[
[\mathbf{d}(c)]_{lk_{t}-n}^{lk_{t}+2N_{t}}\subset T^{-lk_{t}}[b]_{-n}^{n}\cap\left[\RN{lk_{t}}=\lambda_{j}\cdot\frac{1+\varphi}{1+\varphi\lambda_{j}}\cdot\left(1\pm\epsilon\right)\right],
\]
but now the map $c\mapsto\mathbf{d}(c)$ is one to one. 
\end{rem}
In the proof of the next lemma we will make use of the fact that for
every cylinder set $\left([a]_{m}^{l}\right)^{c}$ is $\cF(m,l)$
measurable. 
\begin{lem}
\label{EVC condition}For every $[b]_{-n}^{n}$ cylinder set, $\epsilon>0$
and $j\in\NN$ there exists $t_{0}\in\NN$ such that for all $t>t_{0}$,
\[
\mu\left(\bigcup_{l=1}^{m_{t}/4k_{t}}\mathfrak{RSC}\left(4lk_{t},[b]_{-n}^{n},j,\epsilon\right)\right)\geq0.8\mu\left([b]_{-n}^{n}\right).
\]
\end{lem}
\begin{proof}
Let $[b]_{-n}^{n}$ be a cylinder set and $t_{0}$ be as in Lemma
\ref{Main step for EVC. }. For all $t\geq t_{0}$, $[c]_{0}^{N_{t}}\in\mathcal{C}(t)$
which intersects $[b]_{-n}^{n}$ and $1\leq l\leq m_{t}/4k_{t}$,
\[
\left([c]_{0}^{N_{t}}\cap[b]_{-n}^{n}\right)\cap\left(\mathfrak{RSC}\left(4lk_{t},[b]_{-n}^{n},j,\epsilon\right)\right)^{c}\subset[c]_{0}^{N_{T}}\cap[b]_{-n}^{n}\cap\left([\mathbf{d}(c)]_{4lk_{t}-n}^{4lk_{t}+N_{t}}\right)^{c}
\]
As ${\bf Q}_{1,3}=\min\left\{ {\bf Q}_{i,j}:\ 1\leq i,j\leq3\right\} $,
\begin{eqnarray*}
\nu_{\pi_{{\bf Q}},{\bf Q}}\left([d(c)]_{4lk_{t}-n}^{4lk_{t}+N_{t}}\right) & \geq & \pi_{{\bf Q}}\left(3\right)\mathbf{Q}_{1,3}^{2N_{t}+n-1}\gtrsim\frac{1}{3^{3N_{t}}}.
\end{eqnarray*}
Therefore, one has by repeated applications of (\ref{eq: Mixing time})
(mixing time condition), 
\begin{eqnarray*}
 &  & \mu\left(\left([b]_{-n}^{n}\cap[c]_{0}^{N_{t}}\right)\cap\left(\bigcup_{l=1}^{m_{t}/4k_{t}}\mathfrak{RSC}\left(lk_{t},[b]_{-n}^{n},j,\epsilon\right)\right)^{c}\right)\\
 & \leq & \mu\left(\left([b]_{-n}^{n}\cap[c]_{0}^{N_{t}}\right)\cap\left\{ \bigcap_{l=1}^{m_{t}/4k_{t}}\left([d_{c}]_{4lk_{t}-n}^{4lk_{t}+N_{t}}\right)^{c}\right\} \right)\\
 & \leq & \mu\left(\left([b]_{-n}^{n}\cap[c]_{0}^{N_{t}}\right)\right)\prod_{1=1}^{m_{t}/4k_{t}}\left[\left(1+3^{-3N_{t}}\right)\left(1-\nu_{\pi_{{\bf Q}},{\bf Q}}\left([d(c)]_{4lk_{t}}^{4lk_{t}+N_{t}}\right)\right)\right]\\
 & \leq & \mu\left(\left([b]_{-n}^{n}\cap[c]_{0}^{N_{t}}\right)\right)\left[\left(1+3^{-3N_{t}}\right)\left(1-3^{-3N_{t}}\right)\right]^{m_{t}/4k_{t}}\\
 & \overset{\eqref{m_l large so that the sum of bad events is small}}{\leq} & \frac{1}{t}\mu\left(\left([b]_{-n}^{n}\cap[c]_{0}^{N_{t}}\right)\right).
\end{eqnarray*}
Notice that in the application of the mixing time condition we used
that 
\[
\left(4\left(l+1\right)k_{t}-n\right)-\left(4lk_{t}+N_{t}\right)>(4l+3)k_{t}-\left(4l+1\right)k_{t}=2k_{t}.
\]
If $t$ is large enough then 
\[
\mu\left(\Sigma_{{\bf A}}\backslash\bigcup_{C\in\mathcal{C}(t)}C\right)<0.1\mu\left([b]_{-n}^{n}\right),
\]
and for all $[c]_{0}^{N_{t}}=C\in\mathcal{C}(t)$, 
\begin{eqnarray*}
\mu\left([b]_{-n}^{n}\cap[c]_{0}^{N_{t}}\cap\left(\bigcup_{l=1}^{m_{t}/4k_{t}}\mathfrak{RSC}\left(lk_{t},[b]_{-n}^{n},j,\epsilon\right)\right)\right) & > & \left(1-\frac{1}{t}\right)\mu\left([b]_{-n}^{n}\cap[c]_{0}^{N_{t}}\right)\\
 & \geq & 0.9\mu\left([b]_{-n}^{n}\cap[c]_{0}^{N_{t}}\right).
\end{eqnarray*}
The Lemma follows from 
\begin{eqnarray*}
 &  & \mu\left(\bigcup_{l=1}^{m_{t}/4k_{t}}\mathfrak{RSC}\left(lk_{t},[b]_{-n}^{n},j,\epsilon\right)\right)\\
 & \geq & \mu\left(\uplus_{[c]_{0}^{N_{t}}\in\mathcal{C}(t)}[b]_{-n}^{n}\cap[c]_{0}^{N_{t}}\cap\bigcup_{l=1}^{m_{t}/4k_{t}}\mathfrak{RSC}\left(lk_{t},[b]_{-n}^{n},j,\epsilon\right)\right)\\
 & \geq & 0.9\sum_{[c]_{0}^{N_{T}}\in\mathcal{C}(t)}\mu\left([b]_{-n}^{n}\cap[c]_{0}^{N_{t}}\right)\\
 & \geq & 0.8\mu\left(\left[b\right]_{-n}^{n}\right)
\end{eqnarray*}
\end{proof}
\begin{proof}[Proof of Theorem \ref{thm: lambdas are in R(T)}] 

This is a standard approximation technique. Let $j\in\NN$, $A\in\BB$,
$\mu(A)>0$ and $\epsilon>0$. Since the ratio set condition on the
derivative is monotone with respect to $\epsilon$ and

\[
1<\lambda_{j}\cdot\frac{1+\varphi}{1+\varphi\lambda_{j}}<2,
\]
 we can assume that 
\begin{equation}
1\leq\lambda_{j}\cdot\frac{1+\varphi}{1+\varphi\lambda_{j}}(1\pm\epsilon)\leq2.\label{eq:epsilon is small enough}
\end{equation}
 Since $\cF(-n,n)\uparrow\BB$ as $n\to\infty$, there exists a cylinder
set $\mathfrak{b}=[b]_{-n}^{n}$ such that 
\[
\mu\left(A\cap\mathfrak{b}\right)>0.99\mu\left(\mathfrak{b}\right).
\]
 By Lemma \ref{EVC condition} there exists $t\in\NN$ for which 
\[
\mu\left(\mathfrak{b\cap}\left\{ \bigcup_{l=1}^{m_{t}/4k_{t}}T^{-4lk_{t}}\mathfrak{b}\cap\left[\RN{4lk_{t}}=\lambda_{j}\cdot\frac{1+\varphi}{1+\varphi\lambda_{j}}\cdot\left(1\pm\epsilon\right)\right]\right\} \right)>0.8\mu(\mathfrak{b}).
\]
Denote by 
\[
B=\mathfrak{b\cap}\left\{ \bigcup_{l=1}^{m_{t}/4k_{t}}T^{-4lk_{t}}\mathfrak{b}\cap\left[\RN{4lk_{t}}=\lambda_{j}\cdot\frac{1+\varphi}{1+\varphi\lambda_{j}}\cdot\left(1\pm\epsilon\right)\right]\right\} .
\]
We can assume that for $x\in B$, there exists $C(x)=[c]_{0}^{N_{t}}\in\mathcal{C}_{t}$
so that $x\in C(x)$. Then by the proof of Lemma \ref{Main step for EVC. }
there exists $d(C(x))\in\Sigma_{{\bf A}}\left(2N_{t}+n\right)$ such
that if $x\in[d(C(x))]_{4lk_{t}-n}^{4lk_{t}+2N_{t}}$, then 
\begin{equation}
\RN{4lk_{t}}(x)=\lambda_{j}\cdot\frac{1+\varphi}{1+\varphi\lambda_{j}}\cdot\left(1\pm\epsilon\right)\ \text{and}\ x\in T^{-4lk_{t}}\mathfrak{b}.\label{eq: RN on D(C(x))}
\end{equation}
 Define $\phi:B\to\NN$ 
\[
\phi(x):=\inf\left\{ l\leq m_{t}/4k_{t}:\ [x]_{4lk_{t}-n}^{4lk_{t}+2N_{t}}=[{\bf d}(C(x))]_{4lk_{t}-n}^{4lk_{t}+2N_{t}}\right\} 
\]
and $S=T^{\phi}:B\to S(B)\subset\mathfrak{b}$. We claim that $S$
is one to one. Indeed, since the map $[c]_{0}^{N_{t}}\mapsto{\bf d}(c)$
is one to one, for every $x,y\in B$ such that $C(x)\neq C(y)$,
\[
[Sy]_{-n}^{2N_{t}}=[{\bf d}(C(y))]_{-n}^{2N_{t}}\neq[{\bf d}(C(x))]_{-n}^{2N_{t}}=[Sx]_{-n}^{2N_{t}},
\]
consequently $Sx\neq Sy$. In addition, by the definition of $\phi$,
if $x\neq y$ and $C(x)=C(y)$ then $Sx\neq Sy$. 

It follows from (\ref{eq: RN on D(C(x))}) and (\ref{eq:epsilon is small enough}),
that for all $x\in B$, 
\[
S'(x):=\frac{d\mu\circ S}{d\mu}(x)=\lambda_{j}\cdot\frac{1+\varphi}{1+\varphi\lambda_{j}}\cdot\left(1\pm\epsilon\right)\in\left[1,2\right].
\]
Therefore $\frac{d\mu\circ S^{-1}}{d\mu}(y)\geq\frac{1}{2}$ for all
$y\in S(B)$. A calculation shows that 
\begin{eqnarray*}
\mu(S(B)\cap A) & > & \mu(S(B))-\mu(\mathfrak{b}\backslash A)\\
 & > & \mu(B)-\mu(\mathfrak{b}\backslash A)\\
 & = & 0.79\mu\left(\mathfrak{b}\right),
\end{eqnarray*}
and 
\[
\mu\left(S^{-1}\left(S(B)\cap A\right)\right)>\frac{\mu\left(S(B)\cap A\right)}{2}>0.39\mu(\mathfrak{b}).
\]
So 
\begin{eqnarray*}
 &  & \sum_{l=1}^{m_{t}/4lk_{t}}\mu\left(A\cap\left\{ T^{-4lk_{t}}A\cap\left[\RN{4lk_{t}}=\lambda_{j}\cdot\frac{1+\varphi}{1+\varphi\lambda_{j}}\cdot\left(1\pm\epsilon\right)\right]\right\} \cap\left[\phi=4lk_{t}\right]\right)\\
 & \geq & \mu\left(\left(B\cap A\right)\cap S^{-1}\left(S(B)\cap A\right)\right)\ \ \ \ \ \left\{ \text{Notice that }B,\ S(B)\subset\mathfrak{b}\right\} \\
 & \geq & \mu\left(B\cap A\right)-\mu\left(\mathfrak{b}\backslash S^{-1}\left(S(B)\cap A\right)\right)\\
 & \geq & 0.18\mu\left(\mathfrak{b}\right),
\end{eqnarray*}
and thus there exists $l\in\NN$ such that 
\[
\mu\left(A\cap T^{-4lk_{t}}A\cap\left[\RN{4lk_{t}}=\lambda_{j}\cdot\frac{1+\varphi}{1+\varphi\lambda_{j}}\cdot\left(1\pm\epsilon\right)\right]\right)>0.
\]
This proves the Theorem. 

\end{proof}

\subsection*{Acknowledgement: Section \ref{sec:Type--Markov} is part of the Author's
PhD thesis in Tel Aviv University done under the supervision of Jon
Aaronson. I would like to thank many people among them Jon Aaronson,
Omri Sarig and Ian Melbourne for many helpful discussions and their
continuous support. This research was supported in part by the European
Advanced Grant StochExtHomog (ERC AdG 320977). }

\end{document}